\documentclass[11pt,a4paper]{article}
\pdfoutput=1

\usepackage[a4paper,
            top=2.5cm,
            bottom=2.5cm,
            inner=3.5cm,
            outer=3.5cm]{geometry}
            
\usepackage[utf8]{inputenc}
\usepackage{microtype}
\usepackage{amsmath}
\usepackage{amsfonts}
\usepackage{amssymb}
\usepackage{amsthm}
\usepackage{graphicx}
\usepackage[all]{xy}
\usepackage{rotating}
\usepackage{stmaryrd}

\usepackage{enumerate}
\usepackage{mathtools}
\usepackage{stackrel}
\usepackage[nottoc]{tocbibind}
\usepackage{hyperref} 
\usepackage{caption}

\usepackage{cleveref}
\usepackage{csquotes}

\theoremstyle{plain}
\newtheorem{thm}{Theorem}[section]
\newtheorem{lem}[thm]{Lemma}

\newtheorem{cor}[thm]{Corollary}

\newtheorem{introthm}{Theorem}

\theoremstyle{definition}
\newtheorem{defn}[thm]{Definition}

\newtheorem{example}[thm]{Example}

\newtheorem{assumption}[thm]{Assumption}
\newtheorem{nota}[thm]{Notation}

\theoremstyle{remark}
\newtheorem{rem}[thm]{Remark}

\newtheorem*{acknowledgements}{Acknowledgements}

\newcommand{\N}{\mathbb{N}}
\newcommand{\Z}{\mathbb{Z}}
\newcommand{\Q}{\mathbb{Q}}
\newcommand{\R}{\mathbb{R}}
\newcommand{\C}{\mathbb{C}}

\newcommand{\cO}{\mathcal{O}}
\newcommand{\cP}{\mathcal{P}}

\newcommand{\cX}{\mathcal{X}}
\newcommand{\cY}{\mathcal{Y}}
\newcommand{\cA}{\mathcal{A}}
\newcommand{\cB}{\mathcal{B}}
\newcommand{\cU}{\mathcal{U}}
\newcommand{\cV}{\mathcal{V}}
\newcommand{\cW}{\mathcal{W}}

\newcommand{\cC}{\mathcal{C}}

\newcommand{\Ct}{\mathrm{C}}
\newcommand{\Cz}{\Ct_0}
\newcommand{\Cb}{\Ct_b}

\newcommand{\fm}{\mathfrak{m}}

\newcommand{\fM}{\mathfrak{M}}

\newcommand{\coarseStr}{\mathcal{E}}

\newcommand{\Diag}[1][]{{\Delta_{#1}}}
\newcommand{\HilbertMod}{\mathfrak{H}}

\DeclareMathOperator{\im}{im}
\DeclareMathOperator{\id}{id}

\DeclareMathOperator{\dist}{dist}
\DeclareMathOperator{\Var}{Var}

\DeclareMathOperator{\supp}{supp}
\DeclareMathOperator{\Pen}{Pen}
\DeclareMathOperator{\Cone}{Cone}
\DeclareMathOperator{\Zyl}{Zyl}
\DeclareMathOperator{\Ad}{Ad}
\DeclareMathOperator{\Prop}{prop}
\DeclareMathOperator{\Hom}{Hom}
\DeclareMathOperator{\ind}{ind}

\newcommand{\ev}[1]{\mathop{\textrm{ev}_{#1}}}

\newcommand{\swapcross}{\mathbin{\bar\times}}

\newcommand{\indexcross}[1]{\mathbin{{\times}_{#1}}}

\newcommand{\Homol}[1][0]{\mathrm{E}^{{\textsc{\romannumeral #1}}}}
\newcommand{\Cohom}[1][0]{\mathrm{E}_{{\textsc{\romannumeral #1}}}}
\newcommand{\HomolX}[1][0]{\mathrm{EX}^{{\textsc{\romannumeral #1}}}}
\newcommand{\CohomX}[1][0]{\mathrm{EX}_{{\textsc{\romannumeral #1}}}}

\newcommand{\bound}[1][0]{\partial^{{\textsc{\romannumeral #1}}}}
\newcommand{\cobound}[1][0]{\delta_{{\textsc{\romannumeral #1}}}}

\newcommand{\susp}[1][0]{\sigma_*^{{\textsc{\romannumeral #1}}}}
\newcommand{\cosusp}[1][0]{\sigma^*_{{\textsc{\romannumeral #1}}}}

\newcommand{\HX}{\mathrm{HX}}
\newcommand{\CX}{\mathrm{CX}}
\newcommand{\KX}{\mathrm{KX}}

\newcommand{\K}{\mathrm{K}}

\newcommand{\EE}{\mathrm{E}}

\newcommand{\rL}{\mathrm{L}}

\newcommand{\textCstar}{\ensuremath{\mathrm{C}^*\!}}

\newcommand{\Roe}{\mathrm{C}^*}
\newcommand{\Loc}{\mathrm{C}^*_{\mathrm{L}}}

\newcommand{\FiProLoc}{\mathrm{E}^*_{\mathrm{L}}}
\newcommand{\sHigCom}{\overline{\mathfrak{c}}}
\newcommand{\sHigCor}{\mathfrak{c}}

\newcommand{\sHigCorRed}{\mathfrak{c}^{\mathrm{red}}}

\newcommand{\Kom}{{\mathfrak{K}}}
\newcommand{\Lin}{{\mathfrak{B}}}
\newcommand{\multiplier}{{\mathcal{M}}}

\newcommand{\Estack}[1]{{\setlength{\arraycolsep}{0pt}\begin{array}{rl}#1\end{array}}}
\newcommand{\blank}{-}

\newcommand\subsetsim{\mathrel{\ooalign{\raise0.2ex\hbox{$\subset$}\cr\hidewidth\raise-0.8ex\hbox{\scalebox{0.9}{$\sim$}}\hidewidth\cr}}}
\newcommand\subsetapprox{\mathrel{\ooalign{\raise0.4ex\hbox{$\subset$}\cr\hidewidth\raise-0.8ex\hbox{\scalebox{0.9}{$\approx$}}\hidewidth\cr}}}

\newcommand{\admissible}{{\mathfrak{C}}}

\newcommand{\nbhfmly}{\mathcal{N}}
\newcommand{\nbhfmlysgm}{\mathcal{N}^\sigma}

\crefname{thm}{Theorem}{Theorems}
\crefname{lem}{Lemma}{Lemmas}
\crefname{defn}{Definition}{Definitions}
\crefname{prop}{Proposition}{Propositions}
\crefname{cor}{Corollary}{Corollaries}
\crefname{equation}{}{}

\author{Christopher Wulff\thanks{Supported by the DFG through the Priority Programme ``Geometry at Infinity'' (SPP 2026,  WU 869/1-1, WU 869/1-2, ``Duality and the coarse assembly map'')}}
\title{Secondary cup and cap products in coarse geometry}

\begin{document}
\maketitle

\begin{abstract}
We construct secondary cup and cap products on coarse \mbox{(co-)}homology theories from given cross and slant products. 
They are defined for coarse spaces relative to weak generalized controlled deformation retracts.

On ordinary coarse cohomology, our secondary cup product agrees with a secondary product defined by Roe. For coarsifications of topological coarse \mbox{(co-)}homology theories, our secondary cup and cap products correspond to the primary cup and cap products on Higson dominated coronas via transgression maps. And in the case of coarse $\K$-theory and -homology, the secondary products correspond to canonical primary products between the $\K$-theories of the stable Higson corona and the Roe algebra under assembly and co-assembly.
\end{abstract}

\section{Introduction}

Although the usefulness of multiplicative structures on \mbox{(co-)}homology theories is undisputed in algebraic topology, their coarse geometric counterparts have been neglected for quite a long time and Roe's secondary product on his coarse cohomology (cf.\ \cite[Section 2.4]{RoeCoarseCohomIndexTheory}) remained the only one of his kind for many years.

Only recently there has been more research on this topic, which was primarily motivated by applications to coarse index theory and thereby also to positive scalar curvature.
Cross products were the main multiplicative structure of interest, in particular the cross product between the analytic structure group and $\K$-homology (cf.\ \cite{Siegel_Thesis, XieYu_PSCrhoLoc, ZeidlerPSCProductSecondary, DeeleyGoffeng_RealizingI, Zenobi_SurgerytoAnalysis, EngelWulffZeidler}, albeit not all of these references work in a truely coarse set-up) but also the cross product between the $\K$-theories of Roe algebras (cf.\ \cite{EngelWulffZeidler}).
All of these cross products were complemented by slant products in \cite{EngelWulffZeidler} and, moreover, the slant product between the $\K$-theory of the Roe algebra and the $\K$-theory of the stable Higson corona also generalizes the pairing introduced in \cite{EmeMeyDualizing} to dualize the coarse assembly map.
Further multiplicative structures are the ring and module multiplications between the $\K$-theory of the stable Higson corona and the $\K$-theory of the Roe algebra constructed in \cite{WulffCoassemblyRinghomo,WulffTwisted}, which should be understood as cup and cap products and have applications to the index theory of twisted operators (see also  \cite{WulffFoliations} for a related cap product, which was based on a conjecture in \cite{RoeFoliations}).

Finally, there is also the secondary (cup) product on the coarse $\K$-theory defined in \cite{WulffCoassemblyRinghomo}, which is defined for the pairs of spaces $(X,\{p\})$ if $X$ is contractible to $p$ in a coarse geometric sense.
The purpose of the present paper is to generalize it to a vast class of coarse cohomology theories and to also introduce dual cap products between coarse homology and cohomology theories. Additionally, we define our new products relative to more general subspaces than just single points.

More precisely, the coarse \mbox{(co-)}homology theories under consideration are ones that satisfy the Eilenberg--Steenrod like axioms of \cite{wulff2020equivariant} including the so-called strong homotopy axiom, but here we shall stay within the non-equivariant world, because interesting group actions are rarely compatible with the subspaces and homotopies under consideration anyway.
Given three such coarse cohomology theories $\CohomX[1]^*,\CohomX[2]^*,\CohomX[3]^*$ and a cross product between them or two coarse homology theories $\HomolX[1]_*,\HomolX[3]_*$, a coarse cohomology theory $\CohomX[2]^*$ and a slant product between them, our secondary cup and cap products take the form
\begin{align*}
\Cup\colon &\CohomX[1]^m(X,A)\otimes\CohomX[2]^n(X,B) \to \CohomX[3]^{m+n-1}(X,A\cup B)\,,
\\\Cap\colon &\HomolX[1]_m(X,A\cup B)\otimes\CohomX[2]^n(X,B)\to \HomolX[3]_{m-n+1}(X,A)\,,
\end{align*}
respectively, and they exist whenever the subspaces $A,B\subset X$ are weak generalized controlled deformation retracts.

In a nutshell, the construction works as follows: Recall that one standard construction of primary cup products is pulling back a cross product under the diagonal map $\Delta\colon (X,A\cup B)\to (X,A)\times (X,B)$, i.\,e.\ $x\cup y\coloneqq\Delta^*(x\times y)$, and similarly primary cap products are obtained from slant products via $x\cap y\coloneqq \Delta_*(x)/y$.
If $A$ or $B$ is a weak generalized controlled deformation retract, then $\Delta$ is homotopic to maps $(X,A\cup B)\to (A,A)\times (X,B)$ or $(X,A\cup B)\to (X,A)\times (B,B)$, respectively, and hence the primary cup and cap products vanish.
Now, for the secondary products we combine the two homotopies to obtain a ``suspended diagonal map'' 
\[\Gamma\colon (X,A\cup B)\times([-1,1],\{-1,1\})\to (X,A)\times (X,B)\,,\]
whose domain equipped with a certain coarse structure is a kind of suspension of $(X,A\cup B)$. The secondary cup and cap products are then defined up to a sign by replacing $\Delta$ with $\Gamma$ in the equations above and additionally applying a suspension homomorphism.
This construction is compatible with the usual slogan that secondary products arise by comparing two different reasons for the vanishing of primary products.

Note also that whilst in algebraic topology one often goes the other way around and defines cross products from cup products by $x\times y\coloneqq \pi_1^*x\cup\pi_2^*y$ using the two projections $\pi_1\colon X\times Y\to X$ and $\pi_2\colon X\times Y\to Y$ (and similarily slant products from cap products by $x/y\coloneqq (\pi_1)_*(x\cap \pi_2^*y)$), this does not work in coarse geometry, because $\pi_1,\pi_2$ are not proper and hence not coarse maps. Thus, it is indeed imperative to start our construction with the more substantial cross and slant products.

This paper does not limit itself to the construction of the secondary cup and cap products and proving their basic properties. The main results are the following comparison theorems to other primary and secondary cup and cap products.

First of all, Roe's secondary (cup) product from \cite[Section 2.4]{RoeCoarseCohomIndexTheory}, which is defined at the level of cocycles, can be directly generalized to a secondary product in ordinary coarse cohomology
\[*\colon \HX^m(X,A;M_1)\otimes\HX^n(X,B;M_2)\to\HX^{m+n-1}(X,A\cup B;M_3)\]
for all coarsely connected countably generated coarse spaces $X$ with non-empty subspaces $A,B\subset X$ which are coarsely excisive. Here, the coefficients are multiplied using a given homomorphism of abelian groups $M_1\otimes M_2\to M_3$. However, we point out that we had to slightly correct the signs in Roe's definition. 
Similarly, a secondary (cap) product 
\[\divideontimes\colon \HX_m(X,A\cup B;M_1)\otimes\HX^n(X,B;M_2)\to\HX_{m-n+1}(X,A\cup B;M_3)\] 
between ordinary coarse homology and ordinary coarse cohomology can be defined under the same circumstances.

\begin{introthm}[cf.\ \Cref{lem:secprodcomparison,thm:seccapcomparison}]\label{introthm:secondaryproductcomparison}
Let $X$ be a coarsely connected countably generated coarse space, $A,B\subset X$ weak generalized controlled deformation retracts which are coarsely excisive and let $M_1\otimes M_2\to M_3$ be a homomorphism of abelian groups. Then the secondary product $*$ coincides with our secondary cup product $\Cup$ and the secondary product $\divideontimes$ coincides with our secondary cap product.
\end{introthm}

Secondly, a large class of coarse \mbox{(co-)}homology theories $\HomolX_*,\CohomX^*$ on the category of pairs of countably generated coarse spaces of bornologically bounded geometry can be obtained by coarsifying topological \mbox{(co-)}homology theories $\Homol_*,\Cohom^*$ for $\sigma$-locally compact spaces via a Rips complex construction.
If $\Homol_*$ or $\Cohom^*$ even satisfies the strong excision axiom and $(\partial X,\partial A)$ is a pair of Higson dominated coronas for the pair of coarse spaces $(X,A)$, then there are the transgression maps
\[T_*^{X,A}\colon\HomolX_{*+1}(X,A)\to \Homol_*(\partial X,\partial A)\,,\qquad T^*_{X,A}\colon\Cohom^*(\partial X,\partial A)\to\CohomX^{*+1}(X,A)\]
relating the coarse \mbox{(co-)}homology of the pair of coarse spaces to the topological \mbox{(co-)}homology of the coronas.
Furthermore, a cross or slant product between topological theories induces a cross or slant product between the coarse theories. Hence we obtain in particular a topological primary cup or cap product and a secondary coarse cup or cap product, and the next main theorem says that transgression is compactible with them.

\begin{introthm}[cf.\ \Cref{thm:transgressioncompatiblewithcupandcap}]  \label{introthm:transgression}
If $X$ is a countably generated coarse spaces of bornologically bounded geometry with a Higson dominated corona $\partial X$ and if $A,B\subset X$ are weak generalized controlled deformtion retracts with corresponding coronas $\partial A,\partial B$, then the associated above-mentioned transgression maps and cup or cap products are compatible in the following way:
\begin{align*}
\forall x\in\Cohom[1]^*(\partial X,\partial A),&y\in\Cohom[2]^*(\partial X,\partial B)\colon &T^*_{X,A\cup B}(x\cup y)&=T^*_{X,A}(x)\Cup T^*_{X,B}(y)
\\\forall x\in\HomolX[1]_*(X,A\cup B),&y\in\Cohom[2]^*(\partial X,\partial B)\colon&T_*^{X,A}(x\Cap T^*_{X,B}(y))&=T_*^{X,A\cup B}(x)\cap y
\end{align*}
\end{introthm}

Part of the importance of this theorem stems from the fact that the transgression maps are isomorphisms in many cases, in particular for pairs of open cones $(\cO K,\cO L)$ over pairs of compact metric spaces $(K,L)$ if we choose the base spaces themselves as coronas, $(\partial\cO K,\partial\cO L)=(K,L)$. Then the theorem implies that the theory of our secondary cup and cap products is at least as rich as the corresponding theory of primary cup and cap products on compact metrizable spaces, see \Cref{ex:ConesTransgression}.

Thirdly and most importantly, for all pairs coarsely connected proper metric spaces $(X,A)$ and all coefficient \textCstar-algebras $D$ there are the assembly and coassembly maps
\begin{align*}
\mu\colon\KX_*(X,A;D)&\to \K_*(\Roe(X,A;D))\,,
\\\mu^*\colon\K_{-*}(\sHigCor(X,A;D))&\to\KX^{1+*}(X,A;D)\,,
\end{align*}
respectively. On $\KX^*$ and $\KX_*$ we have our secondary cup and cap products, whereas on the $\K$-theory of the stable Higson coronas we have a primary cup product
\[\cup\colon\K_{-m}(\sHigCor(X,A;D))\otimes\K_{-n}(\sHigCor(X,B;E))\to \K_{-m-n}(\sHigCor(X,A\cup B;D\otimes E))\]
and if in addition $X$ has coarsely bounded geometry and $A$ is non-empty, then there is a primary cap product between the $\K$-theory of the Roe algebra and the $\K$-theory of the stable Higson corona,
\[\cap\colon\K_m(\Roe(X,A\cup B;D))\otimes \K_{-n}(\sHigCor(X,B;E))\to \K_{m-n}(\Roe(X,A;D\otimes E))\,.\]

\begin{introthm}[cf.\ \Cref{thm:cupcapassembly}] \label{introthm:assembly}
Let $X$ be a coarsely connected proper metric space, $A,B\subset X$ weak generalized controlled deformation retracts and let $D,E$ be \textCstar-algebras.
Then 
\[\mu^*(x\cup y)=\mu^*(x)\Cup\mu^*(y)\]
for all $x\in \K_*(\sHigCor(X,A;D)), y\in \K_*(\sHigCor(X,B;E))$. If in addition $X$ has coarsely bounded geometry, then
\[\mu(x\Cap\mu^*(y))=\mu(x)\cap y\]
for all $x\in \KX_*(X,A\cup B;D), y\in \K_*(\sHigCor(X,B;E))$.
\end{introthm}

The assembly map $\mu$ plays an important role in non-commutative coarse geometry. It is subject to the important coarse Baum--Connes isomorphism conjecture and it is the index map in coarse index theory, i.\,e.\ it maps fundamental classes of elliptic operators to their coarse indices. In these fields, the additional algebraic structures provided by the theorem could turn out to be very useful.

We do not have any concrete applications at the moment, but one short-term objective of the cap products could be to develop a coarse version of Poincaré duality.

Furthermore, the author is also thankful to Alexander Engel for pointing out that the second part of \Cref{introthm:assembly} seems to explain a strange phenomenon concerning the index formula of \cite[Theorem 8.7]{WulffTwisted}. Let $M$ be a complete Riemannian manifold of bounded geometry, $K\subset M$ a non-empty compact subset and $C_1,C_2$ \textCstar-algebras. Furthermore, let $D$ be the Dirac operator of a Dirac bundle $S\to M\setminus K$ of Hilbert $C_1$-modules and let $E\to M\setminus K$ be a bundle of Hilbert-$C_2$-modules of vanishing variation. Then the index formula says that the coarse index $\ind(D_E)\in \K_*(\Roe(M,K;C_1\otimes C_2))$ of $D$ twisted by $E$ is equal to the cap product of the coarse index $\ind(D)\in \K_*(\Roe(M,K;C_1))$ of $D$ with a class $\llbracket E\rrbracket\in\K_*(\sHigCor(M,K;C_2))$, and a more general index formula probably holds if we replace $K$ by arbitrary closed subsets $A,B\subset M$.
Now, the strange phenomenon is as follows: The condition on the bundle $E$ says that it can be trivialized by a smooth projection valued function $P\colon M\setminus K\to C_2\otimes\Kom$ whose variation vanishes at infinity, but this function itself is not part of the presumed data. Different choices of $P$ might result in different classes $\llbracket E\rrbracket$, but the coarse index $\ind(D_E)$ is independent of it according to a simple application of bordism invariance. 
In the case where $K$ is additionally a weak generalized controlled deformation retract of $M$, \Cref{introthm:assembly} provides an explanation: The coarse index $\ind(D)$ is the image of a fundamental class $[D]\in\KX_*(M,K;C_1)$ under the assembly map $\mu$ and hence we have
\[\ind(D_E)=\ind(D)\cap\llbracket E\rrbracket=\mu([D])\cap\llbracket E\rrbracket=\mu([D]\Cap\mu^*(\llbracket E\rrbracket))\,.\]
Here, the class $\mu^*(\llbracket E\rrbracket)\in\KX^{1-*}(X,K;C_2)$ should be a purely topological invariant of $E$, i.\,e.\ independent of $P$ (the author has not checked the details), and therefore the cap product is indeed independent of $P$.

The big remaining mystery about our secondary products is the question of when we really need the subspaces $A,B\subset X$ to be weak generalized controlled deformation retracts. 
As we have seen in \Cref{introthm:secondaryproductcomparison}, for the secondary products on ordinary coarse \mbox{(co-)}homology we can exchange this condition for the apparently unrelated coarse excisiveness. 
Furthermore, \Cref{thm:localizedcap} will show that the secondary cap product between coarse $\K$-homology and coarse $\K$-theory
\[\cap\colon\KX_m(X,A\cup B;D)\otimes\KX^n(X,B;E)\to \KX_{m-n+1}(X,A;D\otimes E)\]
pulls back under the coassembly map $\mu^*\colon \K_{1-n}(\sHigCor(X,B;E))\to\KX^n(X,B;E)$ to a primary cap product which, surprisingly, can be defined without any conditions on $A,B\subset X$ except that $A$ should be non-empty.
For now, the controlled deformations of $X$ onto $A$ and $B$ are an essential part of the construction, but the above examples indicate that this condition should possibly not play such a big role.

The present paper is organized as follows. First of all, we should mention that \emph{we expect the reader to be fully familiar with the coarse geometric language and axiomatic framework for coarse \mbox{(co-)}homology theories presented in \cite{wulff2020equivariant}}. 
In fact, the contents of that paper (in the non-equivariant set-up) were originally planned to be part of the present paper, before the author realized that they are so extensive that outsourcing them into a separate publication is fully justified.
In the further course of this paper, we will only recall the less intuitive definitions.

One of the new notions from \cite{wulff2020equivariant} is generalized coarse homotopies. \Cref{sec:homotopies} generalizes them even further to generalized controlled homotopies, which are fundamental to the construction of our secondary products, and provides a toolkit for working with them.
Afterwards we can formulate in \Cref{sec:framework} the framework and standing assumptions on the \mbox{(co-)}homology theories under consideration and the underlying categories of coarse spaces.

The further ingredients for our secondary products, namely coarse suspension and the coarse cross and slant products, will be introduced in \Cref{sec:suspension,sec:crossslant}.
The secondary products will then be constructed in \Cref{sec:secondaryfromdeformations}, where we will also prove their basic properties.

The final three Sections \ref{sec:ordinary}, \ref{sec:transgression}, \ref{sec:assembly} discuss the matters leading to Theorems \ref{introthm:secondaryproductcomparison}, \ref{introthm:transgression}, \ref{introthm:assembly}, respectively.

\begin{acknowledgements}
The author would like to thank Alexander Engel for insightful conversations.
\end{acknowledgements}

\tableofcontents

\section{Generalized controlled and coarse homotopies}
\label{sec:homotopies}

We refer the reader to \cite[Section 2.1]{wulff2020equivariant} for most of the coarse geometric terminology used in this paper. However, we need to expand more on the novel notion of \emph{generalized coarse homotopies}, which was introduced and shown to generalize the classical coarse homotopies in \cite[Section 2.2]{wulff2020equivariant}.
More precisely, this section is about introducing the even more general \emph{generalized controlled coarse homotopies} and developing a large toolkit for working with it, which consists mainly of non-trivial statements analogue to trivial properties of topological homotopies.

Recall that for $E\subset X\times X$ and $F\subset Y\times Y$ we write
\[E\swapcross F\coloneqq \{(x,z,y,w)\mid (x,y)\in E\wedge (z,w)\in F)\}\subset (X\times Y)\times(X\times Y)\]
and if $E,F\subset X\times X$ and $A\subset X$ we furthermore write
\begin{align*}
E\circ F&\coloneqq \{(x,z)\mid\exists y\colon (x,y)\in E\wedge (y,z)\in F\}
\\E\circ A&\coloneqq \{x\mid\exists y\in A\colon (x,y)\in E\}\,.
\end{align*}
If $E$ is an entourage in a coarse space $X$ and $A\subset X$, then we denote the latter set also by $\Pen_E(A)$.

\begin{defn}Let $(X,\coarseStr_X)$ be a coarse space and $K$ a connected compact Hausdorff space.
We denote by $\nbhfmly(X;K)$ the set of all families $\cU=\{U_x\}_{x\in X}$ of neighborhoods $U_x$ of the diagonal $\Diag[K]$ in $K\times K$ indexed over points $x$ of $X$.
For each such $\cU$ we define the \emph{warped cartesian product} $X\indexcross{\cU}K$ to be the set $X\times K$ equipped with the coarse structure $\coarseStr_{\cU}$ generated by the entourages
\begin{align*}
E\swapcross\Diag[K]=&\{((x,s),(y,s))\mid (x,y)\in E\wedge s\in K\}&\text{for }&E\in\coarseStr_X\text{ and}
\\E_{\cU}\coloneqq&\{((x,s),(x,t))\mid x\in X\wedge (s,t)\in U_x\}\,.
\end{align*}
We call another such collection $\cV= \{V_x\}_{x\in X}\in\nbhfmly(X;K)$ a \emph{refinement} of $\cU= \{U_x\}_{x\in X}$ if $V_x\subset U_x$ for all $x\in X$.
If $(Y,\coarseStr_Y)$ is another a coarse space, $L$ another compact Hausdorff space, $f\colon Y\to X$ a controlled map and $g\colon L\to K$ continuous, then we define the \emph{pullback} of $\cU$ under $f\times g$ as the collection $(f\times g)^*(\cU)\coloneqq\{(g\times g)^{-1}(U_{f(y)})\}_{y\in Y}\in\nbhfmly(Y;L)$ of neighborhoods of the diagonal in $L\times L$.
\end{defn}

\begin{lem}\label{lem:warpedtrivialproperties}
With the notation as in the definition we have:
\begin{enumerate}
\item\label{item:warpedbounded} The bounded subsets of a warped cartesian product $X\indexcross{\cU}K$ are exactly the subsets which are contained in some $A\times K$ with $A\subset X$ bounded.

\item\label{item:warpedrefinement} If $\cV$ is a refinement of $\cU$, then the coarse structure $\coarseStr_\cV$ is finer than $\coarseStr_\cU$ and hence the identity map $X\indexcross{\cV}K\to X\indexcross{\cU}K$ is a coarse map.

\item\label{item:warpedfunctorial} The warped cartesian product with respect to the pullback $(f\times g)^*(\cU)$ has the property that $f\times g\colon Y\indexcross{(f\times g)^*(\cU)}L\to X\indexcross{\cU}K$ is a controlled map. If $f$ is not only a controlled but even a coarse map, then $f\times g$ is also a coarse map.

\end{enumerate}
\end{lem}

\begin{proof}
As $K$ is compact and Hausdorff, \ref{item:warpedbounded} can be proven just like the special case of $K$ being an interval in \cite[Section 2.1]{wulff2020equivariant}.
Point \ref{item:warpedrefinement} follows trivially from $E_\cV\subset E_\cU$ and \ref{item:warpedfunctorial} follows from $(f\times g\times f\times g)(E_{(f\times g)^*(\cU)})\subset E_\cU$ for controlledness together with \ref{item:warpedbounded} for properness.
\end{proof}

\begin{lem}\label{lem:warpedcartesianproducts}
Let $X$ be a coarse space and $K,L$ be compact connected Hausdorff spaces.
\begin{enumerate}
\item\label{item:warpedfusion} Given $\cU=\{U_x\}_{x\in X}\in\nbhfmly(X;K)$ and $\cV=\{V_{x,s}\}_{(x,s)\in X\times K}\in\nbhfmly(X\indexcross{\cU}K;L)$ there is a collection $\cW=\{W_x\}_{x\in X}\in\nbhfmly(X;K\times L)$ such that the identity map
\[X\indexcross{\cW}(K\times L)\leftrightarrow(X\indexcross{\cU}K)\indexcross{\cV}L\]
is a coarse equivalence
\item\label{item:warpedcommutative} Given $\cU=\{U_x\}_{x\in X}\in\nbhfmly(X;K)$ and $\cV=\{V_x\}_{x\in X}\in\nbhfmly(X;L)$, then upon identifying them with the pullback collections $\{U_x\}_{(x,t)\in X\times L}\in\nbhfmly(X\indexcross{\cV}L;K)$ and $\{V_x\}_{(x,s)\in X\times K}\in\nbhfmly(X\indexcross{\cU}K;L)$, respectively,
the canonical bijection  
\[(X\indexcross{\cU}K)\indexcross{\cV}L\leftrightarrow(X\indexcross{\cV}L)\indexcross{\cU}K\]
is a coarse equivalence.
\item\label{item:warpedfission} Given $\cW=\{W_x\}_{x\in X}\in\nbhfmly(X;K\times L)$, there are collections $\cU=\{U_x\}_{x\in X}\in\nbhfmly(X,K)$ and $\cV=\{V_{x,s}\}_{(x,s)\in X\times K}\in\nbhfmly(X\indexcross{\cU}K;L)$ such that the identity map
\[(X\indexcross{\cU}K)\indexcross{\cV}L\to X\indexcross{\cW}(K\times L)\]
is a coarse map.
\end{enumerate}
\end{lem}

\begin{proof}
Note that properness of all of the maps is clear from \Cref{lem:warpedtrivialproperties}.\ref{item:warpedbounded} and it remains to show controlledness.

For all $x\in X$, $(s',s)\in U_x$ and $(t',t)\in V_{(x,s)}$ we have
\[((x,s',t'),(x,s,t))=((x,s',t'),(x,s,t'))\circ ((x,s,t'),(x,s,t))\in (E_\cU\swapcross\Diag[L]) \circ E_{\cV}\,,\]
so when we define $\cW$ as the collection of the open neighborhoods
\[W_x\coloneqq\bigcup_{(s,t)\in K\times L}(U_x\circ\{s\})\times (V_{(x,s)}\circ\{t\}))\times (U_x\circ\{s\})\times (V_{(x,s)}\circ\{t\}))\]
and we recall that $U_x\circ\{s\}\coloneqq\{s'\in K\mid (s',s)\in U_x\}$ etc.,
then $E_\cW\subset (E_\cU\swapcross\Diag[L]) \circ E_{\cV} \circ E_{\cV}^{-1}\circ (E_\cU\swapcross\Diag[L])^{-1}$, showing that the map from the left to the right is controlled. For the other direction we use that $E_\cU\swapcross\Diag[L]\subset E_\cW$ and $E_\cV\subset E_\cW$.

The second claim is trivial. 

For the third claim we note that due to the compactness of $L$ we can find for each $(x,s)\in X\times K$ a neighborhood $U'_{x,s}\subset K$ of $s$ such that $U'_{x,s}\times\{t\}\times U'_{x,s}\times\{t\}\subset W_x$ for all $t\in L$ and similarily the compactness of $K$ implies for each $(x,t)\in X \times L$ the existence of a neighborhood $V'_{x,t}\subset L$ of $t$ such that $\{s\}\times V'_{x,t}\times\{s\}\times V'_{x,t}\subset W_x$ for all $s\in K$.
Defining $U_x\coloneqq\bigcup_{s\in K}U'_{x,s}\times U'_{x,s}$ and $V_x\coloneqq \bigcup_{t\in L}V'_{x,t}\times V'_{x,t}$ we see that $E_\cU\swapcross \Diag[L]\subset E_\cW\supset E_\cV$.
\end{proof}

Now we recall \cite[Definition 2.14]{wulff2020equivariant} about generalized coarse homotopies and simultaneously introduce generalized controlled homotopies.

\begin{defn}\label{defn:controlledcoarsehomotopy}
A \emph{generalized controlled/coarse homotopy} between two con\-trolled/ coarse maps $f,g\colon X\to Y$ is a controlled/coarse map $H\colon X\indexcross{\cU}I\to Y$, where $I=[a,b]$ is a closed interval and $\cU\in\nbhfmly(X;I)$, which restricts to $f$ on $X\times\{a\}$ and to $g$ on $X\times\{b\}$.

If $f,g\colon (X,A)\to (Y,B)$ are controlled/coarse maps between pairs of coarse spaces, then a \emph{generalized controlled/coarse homotopy} between them is a generalized coarse homotopy between the absolute controlled/coarse maps $f,g\colon X\to Y$ which takes $A\times I$ to $B$.

In both cases we call $f$ \emph{generalized controlledly/coarsely homotopic} to $g$.
\end{defn}

\begin{lem}\label{lem:homotopytogeneralizedhomotopy}
Let $f\colon Y\to X$ be a controlled/coarse map between coarse spaces, $g_0,g_1\colon L\to K$ two continuous maps between compact Hausdorff spaces and $\cU=\{U_x\}_{x\in X}\in\nbhfmly(X;K)$. If $I$ is a closed interval and $H\colon L\times I\to K$ is a homotopy between $g_0,g_1$, then there is a common refinement $\cU'\in\nbhfmly(X;K)$ of $(f\times g_i)^*(\cU)$ ($i=0,1$) 
 and $\cV'=\{V'_y\}_{y\in Y}\in\nbhfmly(Y;I)$ such that $f\times H\colon (Y\indexcross{\cU'}L)\indexcross{\cV'}I\to X\indexcross{\cU}K$ is a generalized controlled/coarse homotopy between $f\times g_0,f\times g_1\colon Y\indexcross{\cU'}L\to X\indexcross{\cU}K$.
\end{lem}

\begin{proof}
If $\cW$ ist the pullback of $\cU$ under the map $f\times H$, then according to \Cref{lem:warpedtrivialproperties}.\ref{item:warpedfunctorial} and \Cref{lem:warpedcartesianproducts}.\ref{item:warpedfusion} we have a controlled map
\[(Y\indexcross{\cU'}L)\indexcross{\cV'}I\xrightarrow{\id} Y\indexcross{\cW}(L\times I)\xrightarrow{f\times H} X\indexcross{\cU}K
\]for certain collections $\cU',\cV'$. If necessary, $\cU'$ can be refined to be finer than  $(f\times g_i)^*(\cU)$ for $i=0,1$.

If in addition $f$ is proper, then so is $f\times H$.
\end{proof}

The lemma does not give too much information about how much finer $\cU'$ is in comparison to the pullback collections $(f\times g_i)^*(\cU)$. Although this is almost always sufficient, we need a more sophisticated version in one special case.

\begin{lem}\label{lem:homotopydomaindeformationretraction}
Let $X$ be a coarse space, $I$ a closed interval and $\cU\in\nbhfmly(X;I)$. Then there is a refinement $\cU'\in\nbhfmly(X;I)$ of $\cU$ and $\cV'\in\nbhfmly(X;[0,1])$ such that for all Lipschitz maps $g_0,g_1\colon I\to I$ the maps $\id_X\times g_0,\id_X\times g_1\colon  X\indexcross{\cU'}I\to X\indexcross{\cU'}I$ are coarse maps and
\[H\colon  (X\indexcross{\cU'}I)\indexcross{\cV'}[0,1] \to X\indexcross{\cU'}I\,,\quad (x,s,t)\mapsto (x,(1-t)g_0(s)+tg_1(s))\]
is a generalized coarse homotopy between them.

In particular, if $g_0=\id_I$ and $g_1\colon I\twoheadrightarrow J\subset I$ is the canonical retraction onto a closed interval or a point, then $X\indexcross{\cU'}J$ is a strong generalized coarse deformation retract of $X\indexcross{\cU'}I$.
\end{lem}

Here, we have denoted the restriction of $\cU'$ to a collection in $\nbhfmly(X;J)$ by the same letter and `strong' is supposed to mean that the subspace $X\indexcross{\cU'}J$ stays pointwise fixed throughout the homotopy.

\begin{proof}
Given $\cU=\{U_x\}_{x\in X}$, we can choose for each $x\in X$ an $\varepsilon_x>0$ such that
\[U'_x\coloneqq\{(s_1,s_2)\in I^2\mid |s_1-s_2|<\varepsilon_x\}\subset U_x\]
and subsequently we also define 
$V'_x\coloneqq\{(s_1,s_2)\in[0,1]^2\mid |s_1-s_2|<\varepsilon_x\}$. Then $\cU'\coloneqq\{U'_x\}_{x\in X}$ and $\cU'\coloneqq\{V'_x\}_{x\in X}$ do the job, because $H\times H$ clearly maps $E_{\cU'}\swapcross\Diag[{[0,1]}]$ into the entourage $E_{\cU'}^{m}=E_{\cU'}\circ\dots\circ E_{\cU'}$ for $m\in\N$ bigger than the Lipschitz constants of $g_0,g_1$, it maps $E_{\cV'}$ into $E_{\cU'}^{n}$ for $n\in\N$ bigger than the diameter of $I$ and of course it also maps $E\swapcross\Diag[{I\times [0,1]}]$ into $E\swapcross\Diag[I]$ for all entourages $E$ of $X$.
\end{proof}

\section{Framework and standing assumptions}
\label{sec:framework}

In the present paper we always use the axiomatic framework of \cite[Section 3]{wulff2020equivariant} for (non-equivariant) coarse \mbox{(co-)}homology theories on an admissible category $\admissible$ of pairs of coarse spaces.
Most of the axioms obvious analogues to Eilenberg--Steenrod axioms for topological \mbox{(co-)}homology theories, so we do not see the necessity for recalling all of them in detail here. 
There are just two axioms, one of them mandatory and one of them optional, which we should mention to fix some terminology.

All coarse \mbox{(co-)}homology theories are postulated to satisfy the \emph{excision axiom}, which says that all excisions as defined in the following definition induce isomorphisms.

\begin{defn}
An inclusion map between objects of $\admissible$ of the form $(X\setminus C,X\setminus A)\to (X,A)$ with $C\subset A\subset X$ is called an \emph{excision} if for every entourage $E$ of $X$ there is an entourage $F$ of $X$ such that
$\Pen_E(A)\setminus C\subset\Pen_F(A\setminus C)$.

A \emph{triad} in $\admissible$ is a triple $(X;A,B)$ such that both $(X,A)$ and $(X,B)$ are objects in $\admissible$.
We call a triad $(X;A,B)$ in $\admissible$ \emph{(coarsely) excisive} in $\admissible$ if for every entourage $E$ of $X$ there is an entourage $F$ of $X$ such that $\Pen_E(A)\cap\Pen_E(B)\subset\Pen_F(A\cap B)$.
\end{defn}
It is readily verified that $(X\setminus C,X\setminus A)\to (X,A)$ is an excision if an only if $(X;A,X\setminus C)$ is an excisive triad. 
We emphasize, however, that in this case we have $X=A\cup (X\setminus C)$, whereas general excisive triads $(X;A,B)$ do not need to satisfy $X=A\cup B$.

The other noteworthy axiom is the \emph{strong homotopy axiom}, which holds for a coarse \mbox{(co-)}homology theory if generalized coarsely homotopic coarse maps (cf.\ \Cref{defn:controlledcoarsehomotopy}) induce the same map on \mbox{(co-)}homology.
This axiom is optional, because in general we only demand the weaker \emph{homotopy axiom}, which says that maps which are coarsely homotopic in the classical sense induce the same map.

We are going to construct our secondary products only for coarse \mbox{(co-)}homology theories which satisfy this strong homotopy axiom.
These include the ordinary coarse \mbox{(co-)}homology theories and coarsifications of topological \mbox{(co-)}homology theories for $\sigma$-locally compact spaces, which we will analyze later on in \Cref{sec:ordinary,sec:transgression}.

There are two main examples of coarse \mbox{(co-)}homology theories for which we do not know if the strong homotopy axiom holds, namely the $\K$-theory of the stable Higson coronas and the $\K$-theory of the Roe algebras.
Thus, they cannot be equipped with our secondary products, but that does not need to bother us at all, because they have even better primary products instead. This will be the topic of Section \Cref{sec:assembly}.

We also have to pose some additional assumptions on the admissible category $\admissible$. 
Recall that admissibility means that $\admissible$ is a full sub-category of the category of pairs of coarse spaces which contains $\emptyset=(\emptyset,\emptyset)$ and such that if $(X,A)$ is an object in $\admissible$, then so are $(X,X),(A,A),X=(X,\emptyset)$ and $A=(A,\emptyset)$.
However, this is not enough to make our constructions work.
On the one hand, the theory of generalized controlled homotopies will play an essential role and therefore we need $\admissible$ to be closed under warped products with intervals.
On the other hand, our secondary cup and cap products will be obtained from cross and slant products and hence we also assume that $\admissible$ is closed under taking products of coarse spaces. 
More precisely, what we need is the following.

\begin{assumption}\label{ass:assumptiononadmissible}
We assume that $\admissible$ is an admissible category with the following properties:
\begin{itemize}
\item If $(X,A),(Y,B)$ are objects in $\admissible$, then so are
\begin{gather*}
(X,A)\times (Y,B)\coloneqq (X\times Y, X\times B \cup A\times Y)\,,
\\(X\times B \cup A\times Y,X\times B)\quad\text{and}\quad (X\times B \cup A\times Y,A\times Y)\,.
\end{gather*}
 
\item If $(X,A)$ is an object in $\admissible$, $I=[a,b]$ is a closed interval with subsets $I_2\subset I_1\subset I$ which are closed and have finitely many components (i.\,e.\ $I_1,I_2$ are finite disjoint unions of closed intervals and points) and $\cU\in\nbhfmly(X;I)$, then 
\[(X,A)\indexcross{\cU}(I,I_i)\coloneqq (X\indexcross{\cU}I,X\times I_i\cup A\times I)\]
for $i=1,2$ and 
\[(X\times I_1\cup A\times I,X\times I_2\cup A\times I)\]
equipped with the subspace coarse structure of $X\indexcross{\cU}I$ are also objects of $\admissible$.
\end{itemize}
\end{assumption}

\section{Coarse suspension}
\label{sec:suspension}

As a first step towards the secondary products, we construct the coarse suspension homomorphisms, which relate the \mbox{(co-)}homology of a pair of coarse spaces $(X,A)$ with the \mbox{(co-)}homology of a warped product $(X,A)\indexcross{\cU}(I,\partial I)$. But this works only if the collection $\cU=\{U_x\}_{x\in X}$ makes the intervall $I$ become arbitrary wide as we go far away from $A$, that is, the $U_x$ become arbitrarily small, because otherwise $X\indexcross{\cU}I$ would be coarsely equivalent to the subspace $X\times\partial I\cup A\times I$ and hence the \mbox{(co-)}homology of $(X,A)\indexcross{\cU}(I,\partial I)$ would vanish.

\begin{defn}\label{def:suspensioncondition}
Let $(X,A)$ be a pair of coarse spaces and $I$ a closed interval. We say that $\cU=\{U_x\}_{x\in X}\in\nbhfmly(X;I)$ satisfies the \emph{suspension condition} if  for each $\varepsilon>0$ there is a penumbra $\Pen_{E}(A)$ of $A$ such that $\sup_{(s,t)\in U_x}|s-t|\leq\varepsilon$ for all $x\in X\setminus\Pen_E(A)$. The set of all such $\cU$ is denoted by $\nbhfmlysgm(X,A;I)$.
\end{defn}

\begin{rem}\label{rem:suspensioncondition}
Several things should be pointed out about this definition.
\begin{enumerate}
\item If $\cU$ satisfies the suspension condition and $\cU'$ is a refinement of $\cU$, then $\cU'$ also satisfies the suspension condition.
\item\label{item:suspensioncountability} The set $\nbhfmlysgm(X,A;I)$ is non-empty if and only if there is a sequence $E_1,E_2,\dots$ of entourages of $X$ such that $X=\bigcup_{i=1}^\infty \Pen_E(A)$. In this case, each family in $\nbhfmly(X;I)$ has a refinement in $\nbhfmlysgm(X,A;I)$.

In practice, this is not a big issue, as many interesting admissible categories are countable generated.

\item The suspension condition can only be satisfied if each coarse component of $X$ intersects $A$ nontrivially. Coincidentally, this is the only case which we will be interested in later on. 
If a coarse component $X_0\subset X$ is disjoint from $A$, then the classical way of performing coarse suspension on $X_0$ seems more reasonable, that is, simply taking its coarse product with $\R$ instead of the warped product with a closed interval relative to the endpoints.
\end{enumerate}
\end{rem}

\begin{lem}\label{lem:suspensionspace}
Let $(X,A)$ be a pair of coarse spaces, $I$ a closed interval with boundary $\partial I= \{e_1,e_2\}$ (we do not specify which of $e_1,e_2$ is which end) and $\cU\in\nbhfmlysgm(X,A;I)$.
Then the inclusion 
\[(X,A)\quad\to\quad\left(X\times\partial I \cup A\times I\,,\quad X\times\{e_2\}\cup A\times I\right)\,,\quad x\mapsto (x,e_1)\,,\]
where the spaces on the right are equipped with the subspace coarse structure of $X\indexcross{\cU}I$,
is an excision and hence induces isomorphisms on coarse \mbox{(co-)}homology.
\end{lem}

\begin{proof}
Any entourage of $X\indexcross{\cU} I$ is contained in one of the form 
\[E=(E_1\swapcross \Diag[I])\circ E_\cU\circ (E_2\swapcross \Diag[I])\circ E_\cU\circ \dots\circ E_\cU\circ (E_n\swapcross \Diag[I])\circ E_\cU\]
with $n\in\N$ and entourages $E_1,\dots,E_n$ of $X$. Let $E'$ be an entourage of $X$ such that $\sup_{(s,t)\in U_x}|s-t|\leq\frac{|e_1-e_2|}{n+1}$ for all $x\in X\setminus\Pen_ {E'}(A)$ and define the entourage 
\[F'\coloneqq E'\cup (E_1\circ E')\cup (E_1\circ E_2\circ E')\cup\dots\cup (E_1\circ\dots\circ E_n\circ E')\]
of $X$. 
Then for each $(x,e_1)\in \Pen_E(X\times\{e_2\})$ we can find points $x_0=x,x_1,\dots,x_n\in X$ and $s_0=e_1,s_1,\dots,s_n=e_2 \in I$ such that $(x_{i-1},x_i)\in E_i$ and $(s_{i-1},s_i)\in U_{x_i}$ for all $i=1,\dots,n$. But this is only possible if $|s_{i-1}-s_i|\geq\frac{|e_1-e_2|}n$ for at least one $i$, which implies $x_i\in \Pen_{E'}(A)$ and therefore $x\in\Pen_{F'}(A)$.
This implication shows
\[\Pen_E(X\times\{e_2\}\cup A\times I)\cap (X\times\{e_1\})\subset\Pen_F(A\times\{e_1\})\]
for the entourage $F\coloneqq (F'\cup (E_1\circ\dots\circ E_n))\swapcross\Diag[I]$, which is exactly the excisiveness condition of \cite[Definition 3.2]{wulff2020equivariant} for the spaces under consideration.
\end{proof}

\begin{lem}\label{lem:suspensionhomomorphism}
Let $(X,A)$ be a pair of coarse spaces in $\admissible$, $I$ a closed interval with boundary $\partial I= \{e_1,e_2\}$ and $\cU\in\nbhfmlysgm(X,A;I)$.  
Then for any coarse homology theory $\HomolX_*$ or coarse cohomology theory $\CohomX^*$ on $\admissible$ satisfying the strong homotopy axiom there are homomorphisms
\begin{align*}
\susp&\colon \HomolX_*(X,A)\to \HomolX_{*+1}((X,A)\indexcross{\cU}(I,\partial I ))
\\\cosusp&\colon \CohomX^{*+1}((X,A)\indexcross{\cU}(I,\partial I ))\to \CohomX^*(X,A)
\end{align*}
such that $\sigma_*$ is a right inverse to
\[\HomolX_{*+1}((X,A)\indexcross{\cU}(I,\partial I ))\xrightarrow{\partial}\HomolX_*(X\times\partial I \cup A\times I, X\times\{e_2\}\cup A\times I)\cong\HomolX_*(X,A)\]
and $\sigma^*$ is a left inverse to
\[\CohomX^*(X,A)\cong \CohomX^*(X\times\partial I \cup A\times I, X\times\{e_2\}\cup A\times I) \xrightarrow{\delta}\CohomX^{*+1}((X,A)\indexcross{\cU}(I,\partial I ))\]
where the connecting homomorphisms are those associated to the triple
\[\left(X\indexcross{\cU}I\,,\quad X\times\partial I \cup A\times I\,,\quad X\times\{e_2\}\cup A\times I\right)\]
and the isomorphisms are those induced by the excision from \Cref{lem:suspensionspace}. 
Furthermore, they are natural under coarse maps $f\colon (Y,B)\to (X,A)$, homeomorphism $g\colon J\to I$ of intervals which preserves the chosen distinction of the end points and the associated coarse map on the warped products
\[f\times g\colon (Y,B)\indexcross{\cV}(J,\partial J)\to (X,A)\indexcross{\cU}(I,\partial I)\]
from \Cref{lem:warpedtrivialproperties}.\ref{item:warpedfunctorial}, with $\cV$ being a refinement of $(f\times g)^*(\cU)$ which automatically is an element of $\nbhfmlysgm(Y,B;J)$.
Also, each $\cU$ has a refinement $\cU'$ for which $\susp,\cosusp$ are isomorphisms.
\end{lem}

\begin{proof}
According to the \Cref{lem:homotopydomaindeformationretraction} there is a refinement $\cU'$ of $\cU$ such that 
\((X,A)\indexcross{\cU'}(I,\{e_2\})\)
is generalized coarsely homotopy equivalent to $(X,X)$ and therefore $\HomolX_*$, $\CohomX^*$ vanish on them. 
Naturality of the connecting homomorphisms associated to the triples 
\[\left(X\indexcross{\cU^{(\prime)}}I\,,\quad X\times\partial I \cup A\times I\,,\quad X\times\{e_2\}\cup A\times I\right)\]
under the coarse map $X\indexcross{\cU'}I\to X\indexcross{\cU}I$ together with the excision isomorphisms yields the two commutative triangles
\[\xymatrix{
\HomolX_{*+1}((X,A)\indexcross{\cU}(I,\partial I ))\ar[r]&\HomolX_*(X,A)&\CohomX^{*+1}((X,A)\indexcross{\cU'}(I,\partial I ))
\\\HomolX_{*+1}((X,A)\indexcross{\cU'}(I,\partial I ))\ar[u]\ar[ur]_-{\cong}&\CohomX^*(X,A)\ar[r]\ar[ur]_-{\cong}&\CohomX^{*+1}((X,A)\indexcross{\cU}(I,\partial I ))\ar[u]
}\]
which show the existence of one-sided inverses for $\cU$ and two-sided inverses for $\cU'$. 

To prove naturality of this construction, we pick a common refinement $\cV'$ of $\cV$ and $(f\times g)^*(\cU')$ with the property of \Cref{lem:homotopydomaindeformationretraction} and then the claim follows directly from the naturality of connecting homomorphisms under the maps of the diagram
\[\xymatrix{Y\indexcross{\cV}J\ar[r]&X\indexcross{\cU}I\\Y\indexcross{\cV'}J\ar[r]\ar[u]&X\indexcross{\cU'}I\ar[u]}\]
relative to certain subspaces.
In particular, the case $f=\id_X$, $g=\id_I$ shows that the maps $\susp$ and $\cosusp$ are independent of the choice of the refinement $\cU'$ in the construction.
\end{proof}

\begin{defn}
We call $(X,A)\indexcross{\cU}(I,\partial I)$ as in \Cref{lem:suspensionspace} a \emph{suspension} of $(X,A)$. If $e_1$ is the lower endpoint of the interval $I$, then we call the homomorphisms $\sigma_*$ and $\sigma^*$ from \Cref{lem:suspensionhomomorphism} the associated \emph{suspension homomorphisms/isomorphisms}.
\end{defn}

Note that under the condition of \Cref{rem:suspensioncondition}.\ref{item:suspensioncountability}, the definition of the suspension homomorphisms can be extended to all $\cU\in\nbhfmly(X;I)$ by choosing a refinement $\cU'\in\nbhfmlysgm(X,A;I)$ of $\cU$ and composing the suspension homomorphisms for $\cU'$ with the homomoprhisms induced by the refinement. 
However, they are harder to work with, because the connecting homomorphisms, to which they are supposed to be one-sided inverses, do not exist for $\cU$. 

The following lemma tells us how the choice of the distinguished endpoints and the order of consecutive suspensions affects the signs.

\begin{lem}\label{lem:suspensionsign}
Let $\HomolX_*$ be a coarse homology or $\CohomX^*$ a coarse cohomology theory on $\admissible$ satisfying the strong homotopy axiom. Furthermore, let $(X,A)$ be a pair of coarse spaces in $\admissible$, $I=[a,b]$ a closed interval and $\cU\in\nbhfmlysgm(X,A;I)$.
\begin{enumerate}
\item\label{item:intervalorientation} If $J$ is another closed interval, $f\colon J\to I$ is an orientation preserving and $g\colon J\to I$ an orientation reversing homeomorphism, then there is a common refinement $\cV$ of $(\id\times f)^*(\cU)$ and $(\id\times g)^*(\cU)$ such that the homomorphisms on $\HomolX_*$ or $\CohomX^*$ induced by the two coarse maps
\[\id_X\times f,\id_X\times g\colon (X,A)\indexcross{\cV}(J,\partial J)\to (X,A)\indexcross{\cU}(I,\partial I)\]
are negatives of one another.
\item\label{item:suspensionreflection} If we interchange the roles of $e_1$ and $e_2$ in \Cref{lem:suspensionhomomorphism}, then the homomorphisms $\sigma_*$ and $\sigma^*$ will change signs. In particular, if $e_1$ is the upper endpoint of the interval, then the resulting maps are exactly the negative of the suspension homomorphisms.
\item\label{item:doublesuspension} If $J$ is another closed interval  and $\cV\in\nbhfmlysgm(X,A;J)$, then the suspensions with repect to $I$ commute with the suspensions with respect to $J$ up to the sign $-1$, that is, the diagrams
\[\xymatrix{
\HomolX_*(X,A)\ar[r]^{\sigma_*}\ar[d]^{\sigma_*}&\HomolX_{*+1}((X,A)\indexcross{\cU}(I,\partial I))\ar[d]^{\sigma_*}
\\\HomolX_{*+1}((X,A)\indexcross{\cV}(J,\partial J))\ar[r]^{\sigma_*}
&\HomolX_{*+2}((X,A)\indexcross{\cU}(I,\partial I)\indexcross{\cV}(J,\partial J))
\\\CohomX^{*+2}((X,A)\indexcross{\cU}(I,\partial I)\indexcross{\cV}(J,\partial J))\ar[r]^{\sigma^*}\ar[d]^{\sigma^*}
&\CohomX^{*+1}((X,A)\indexcross{\cV}(J,\partial J))\ar[d]^{\sigma^*}
\\\CohomX^{*+1}((X,A)\indexcross{\cU}(I,\partial I))\ar[r]^{\sigma^*}&\CohomX^*(X,A)
}\]
commute up to the sign $-1$.
\end{enumerate}
\end{lem}

\begin{proof}
Consider the coarse space $Y\coloneqq X\indexcross{\cU}I\cup_{X\times\{a\}}X\indexcross{\cU}I$ obtained by gluing together two copies of $X\indexcross{\cU}I$ along the common subspace $X\times\{a\}$ and equipping it with the canonical coarse structure for which the two canonical inclusions $\iota_1,\iota_2\colon X\indexcross{\cU}I\to Y$ are inclusions as coarse subspaces, that is, it is generated by entourages of the form $(\iota_i\times\iota_i)(E)$ for entourages $E$ of $X\indexcross{\cU}I$ and $i=1,2$. 
Then there are three canonical coarse maps $p,p_1,p_2\colon Y\to X\indexcross{\cU}I$: The map $p$ is simply defined to be the identity on both halves $\im(\iota_i)$ of $Y$, whereas $p_i$ maps only $\im(\iota_i)$ identically to $X\indexcross{\cU}I$ and projects the other half $\im(\iota_{3-i})$ onto the subspace $X\times\{a\}$.

Note that the coarse space $Y$ can be identified with some warped cartesian product $X\indexcross{\cW}K$ for the intervall $K=I\cup_{\{a\}}I$ and it contains the two subspaces 
\begin{align*}
Z&\coloneqq \iota_1(X\times\{a,b\}\cup A\times I)\cup\iota_2(X\indexcross{\cU'}I)\,,
\\W&\coloneqq \iota_1(X\times\{a,b\}\cup A\times I)\cup\iota_2(X\times\{a,b\}\cup A\times I)\,.
\end{align*}
Then \Cref{ass:assumptiononadmissible} tells us that $(Y,Z),(Y,W),(Z,W)$ are also objects in $\admissible$ up to this identification.
Furthermore we note that the two inclusions $\iota_1\colon (X,A)\indexcross{\cU}(I,\partial I)\to (Y,Z)$ and $\iota_2\colon (X,A)\indexcross{\cU}(I,\partial I)\to (Z,W)$ are obviously excisions. With these identifications it is now clear that the long exact sequence in $\HomolX_*$ associated to the triple $(Y,Z,W)$ decomposes into the split short exact sequences
\[\xymatrix{
\HomolX_*((X,A)\indexcross{\cU}(I,\partial I))\ar@{^(->}[r]_-{\iota_2}&\HomolX_*(Y,W)\ar@{->>}[r]_-{p_1}\ar@{-->>}@/_2pc/[l]_-{p_2}&\HomolX_*((X,A)\indexcross{\cU}(I,\partial I))\ar@{_(-->}@/_2pc/[l]_-{\iota_1}
}\]
and hence we get the vertical isomorphisms in the diagram
\[\xymatrix{
\HomolX_*((X,A)\indexcross{\cV}(J,\partial J))\ar[r]^-{(\id\times h)_*}\ar@/_/[dr]_-{((\id\times f)_*,(\id\times g)_*)\qquad}&\HomolX_*(Y,W)\ar[r]^-{p}\ar@<-0.5ex>[d]_{((p_1)_*,(p_2)_*)}&\HomolX_*((X,A)\indexcross{\cU}(I,\partial I))
\\&\qquad\qquad\mathclap{\HomolX_*((X,A)\indexcross{\cU}(I,\partial I))^2}\qquad\qquad\ar@/_/[ur]_-{+}\ar@<-0.5ex>[u]_{(\iota_1)_*+(\iota_2)_*}&
}\]
which are inverse to each other and make the right part commute.
The map $h$ and commutativity of the left part is obtained as follows. 
We can choose some homeomorphism $h\colon J\to K$ such that $p_1\circ (\id\times h)=\id\times f'$ and $p_2\circ (\id\times h)=\id\times g'$ for some continuous maps $f',g'\colon (J,\partial J)\to (I,\partial I)$ which are homotopic to $f$ and $g$, respectively, and $p\circ (\id\times h)=\id\times h'$ with a continuous map $h'\colon (J,\partial J)\to (I,\{b\})\subset (I,\partial I)$ which is homotopic to the constant map with image $b$. Then, using \Cref{lem:homotopytogeneralizedhomotopy} shows that there is a common refinement $\cV$ of $(\id\times f)^*(\cU)$ and $(\id\times g)^*(\cU)$ (among others) such that $p_1\circ (\id\times h)$, $p_2\circ (\id\times h)$ and $p\circ (\id\times h)$ are generalized coarsely homotopic to $\id\times f$, $\id\times g$ and $\id\times\mathrm{const}_b$, respectively, as maps $(X,A)\indexcross{\cV}(J,\partial J)\to (X,A)\indexcross{\cU}(I,\partial I)$. Therefore, the left part of the diagram commutes, too, and the composition of the horizontal arrows factor through $\HomolX_*(X\times\{b\},X\times\{b\})=0$. The homological version of the first claim follows, and the cohomological version is completely dual.

For the second claim we choose an orientation preserving homeomorphism $f\colon J\to I$ and an orientation reversing homeomorphism $g\colon J\to I$ and let $\cV$ be the associated collection from the first statement, which automatically also satisfies the suspension condition. Then naturality of the homomorphisms from \Cref{lem:suspensionhomomorphism} shows that we have a commutative diagram
\[\xymatrix{
\HomolX_*(X,A)\ar[r]^{\sigma}\ar[dr]^{\sigma}\ar[d]_{\sigma}&\HomolX_{*+1}((X,A)\indexcross{\cU}(I,\partial I))
\\\HomolX_{*+1}((X,A)\indexcross{\cU}(I,\partial I))&\HomolX_{*+1}((X,A)\indexcross{\cV}(J,\partial J))\ar[u]_{(\id\times f)_*}\ar[l]^{(\id\times g)_*}
}\]
where the horizontal and the diagonal $\sigma$ are associated to choosing the lower endpoints of the interval and the vertical one is associated to choosing the upper endpoint. The homological version of the second property now follows from the first, and dually we obtain the cohomological version.

For the third claim we choose a homeomorphism $g\colon (I,\partial I)\times (J,\partial J)\to (I,\partial I)\times (J,\partial J)$ which rotates the rectangle $I\times J$ by ninety degrees up to the obvious distortion. It is homotopic to the identity and hence \Cref{lem:homotopytogeneralizedhomotopy} in combination with \Cref{lem:warpedcartesianproducts}.\ref{item:warpedfission} gives us refinements $\cU',\cV'$ of $\cU,\cV$ such that
\[\id_X\times g,\id_{X\times I\times J} \colon (X,A)\indexcross{\cU'}(I,\partial I)\indexcross{\cV'}(J,\partial J)\to (X,A)\indexcross{\cU}(I,\partial I)\indexcross{\cV}(J,\partial J)\]
are two coarse maps which are generalized coarsly homotopic and hence induce the same map on $\HomolX_*,\CohomX^*$.
Composing the double suspension 
\[\HomolX_*(X,A)\xrightarrow{\sigma}\HomolX_{*+1}((X,A)\indexcross{\cU'}(I,\partial I))\xrightarrow{\sigma}\HomolX_{*+2}((X,A)\indexcross{\cU'}(I,\partial I)\indexcross{\cV'}(J,\partial J))\]
with $(\id_{X\times I\times J})_*$ yields the top right composition of the first diagram by naturality of the suspension map and similiarily composing it with $(\id_X\times g)_*$ yields the negative of the bottom left composition, because here the orientation of one of the intervals is reversed and hence part 2 gives us a sign $-1$. For $\CohomX^*$ we just have to dualize the proof.
\end{proof}

\section{Coarse cross and slant products}
\label{sec:crossslant}

We will now introduce our notions of cross and slant products between coarse homology and coarse cohomology theories. 
This is essentially done by turning diagrams similar to the ones found in \cite[Chapter VII]{DoldTopology}, which describe the compatibility of cross and slant products with connecting homomorphisms, into axioms.
An instant consequence of these axioms will be the compatibility of cross and slant products with suspension.

In topology, cross and slant products in the relative case are usually subject to a certain excisiveness condition. Let us note right away that the corresponding coarse geometric analogue is always satisfied.
\begin{lem}\label{lem:crossproductexcisiveness}
If $(X,A)$, $(Y,B)$ are objects in $\admissible$, then $(X\times Y;A\times Y,X\times B)$ is an excisive triad in $\admissible$. In particular, the inclusions $i\colon A\times (Y,B)\to (A\times Y\cup X\times B,X\times B)$ and $j\colon (X,A)\times B\to (A\times Y\cup X\times B,A\times Y)$ are excisions.
\end{lem}

\begin{proof}
If $E$ and $F$ are entourages in $X$ and $Y$, respectively, then 
\begin{align*}
\Pen_{E\swapcross F}(A\times Y)\cap \Pen_{E\swapcross F}(X\times B) &=\Pen_E(A)\times Y\cap X\times \Pen_F(B)\nonumber
\\&=\Pen_E(A)\times \Pen_F(B)\nonumber
\\&=\Pen_{E\swapcross F}(A\times B)\,. \qedhere
\end{align*}
\end{proof}

\begin{defn}\label{defn:crossproducts}
A \emph{cross product} between two coarse cohomology theories $\CohomX[1]^*,\CohomX[2]^*$ on $\admissible$ with values in a third one $\CohomX[3]^*$ is a collection of transformations
\[\times=\times^{m,n}\colon \CohomX[1]^m(X,A)\otimes\CohomX[2]^n(Y,B)\to\CohomX[3]^{m+n}((X,A)\times (Y,B))\]
for all $m,n\in\Z$ which are natural under morphisms of $\admissible\times\admissible$
and which are compatible with the connecting homomorphisms $\cobound[1]$, $\cobound[2]$, $\cobound[3]$ of the three cohomology theories in the sense that the diagram
\[\xymatrix{
\CohomX[1]^m(A)\otimes \CohomX[2]^n(Y,B)\ar[r]^\times\ar[dd]^{\cobound[1]\otimes\id}&\CohomX[3]^{m+n}(A\times (Y,B))\ar[d]^{\cong}_{(i^*)^{-1}}
\\&\CohomX[3]^{m+n}(A\times Y\cup X\times B,X\times B)\ar[d]^{\cobound[3]}
\\\CohomX[1]^{m+1}(X,A)\otimes\CohomX[2]^n(Y,B)\ar[r]^{\times}&\CohomX[3]^{m+n+1}((X,A)\times (Y,B))
}\]
commutes and the diagram
\[\xymatrix{
\CohomX[1]^m(X,A)\otimes \CohomX[2]^n(B)\ar[r]^\times\ar[dd]^{\id\otimes\cobound[2]}&\CohomX[3]^{m+n}((X,A)\times B)\ar[d]^{\cong}_{(j^*)^{-1}}
\\&\CohomX[3]^{m+n}(A\times Y\cup X\times B,A\times Y)\ar[d]^{\cobound[3]}
\\\CohomX[1]^{m+1}(X,A)\otimes\CohomX[2]^n(Y,B)\ar[r]^{\times}&\CohomX[3]^{m+n+1}((X,A)\times (Y,B))
}\]
commutes up to a sign $(-1)^{m}$.
We shall denote such cross products simply by $\times\colon\CohomX[1]^*\otimes\CohomX[2]^*\to\CohomX[3]^*$.
\end{defn}

\begin{defn}\label{defn:slantproducts}
A \emph{slant product} between a coarse homology theory $\HomolX[1]_*$ and a coarse cohomology theory $\CohomX[2]^*$ with values in another coarse homology theory $\HomolX[3]_*$ is a collection of transformations of the form
\[/\coloneqq/^n_m\colon \HomolX[1]_m((X,A)\times (Y,B))\otimes\CohomX[2]^n(Y,B)\to\HomolX[3]_{m-n}(X,A)\]
for all $m,n\in\Z$ which are natural with respect to both pairs of spaces in the sense that
the corresponding transformations 
\[\HomolX[1]_m((X,A)\times (Y,B))\to\Hom(\CohomX[2]^n(Y,B),\HomolX[3]_{m-n}(X,A))\]
are natural under morphisms of $\admissible\times\admissible$
and which are compatible with the connecting homomorphisms $\bound[1]$, $\cobound[2]$, $\bound[3]$ of the three theories in the sense that the diagram
\[\xymatrix{
\HomolX[1]_m((X,A)\times (Y,B))\otimes \CohomX[2]^n(Y,B)\ar[r]^-{/}\ar[d]^{\bound[1]\otimes\id}
&\HomolX[3]_{m-n}(X,A)\ar[dd]^{\bound[3]}
\\\HomolX[1]_{m-1}(A\times Y\cup X\times B,X\times B)\otimes \CohomX[2]^n(Y,B)\ar[d]_{(i_*)^{-1}}^{\cong}&
\\\HomolX[1]_{m-1}(A\times (Y,B))\otimes\CohomX[2]^n(Y,B)\ar[r]^-{/}
&\HomolX[3]_{m-n-1}(A)
}\]
commutes and the diagram
\[\xymatrix{
{\begin{array}{c}\HomolX[1]_m((X,A)\times (Y,B))\\\otimes \CohomX[2]^n(B)\end{array}}
\ar[d]^-{\bound[1]\otimes\id}\ar[r]^-{\id\otimes\cobound[2]}
&{\begin{array}{c}\HomolX[1]_m((X,A)\times (Y,B))\\\otimes \CohomX[2]^{n+1}(Y,B)\end{array}}\ar[dd]^{/}
\\{\begin{array}{c}\HomolX[1]_{m-1}(A\times Y\cup X\times B,A\times Y)\\\otimes \CohomX[2]^n(B)\end{array}}\ar[d]_{(j_*)^{-1}}^{\cong}&
\\\HomolX[1]_{m-1}((X,A)\times B)\otimes\CohomX[2]^{n}(B)\ar[r]^{/}
&\HomolX[3]_{m-n-1}(X,A)
}\]
commutes up to a sign $(-1)^{m}$.
We shall denote such slant products simply by $/\colon\HomolX[1]_*\otimes\CohomX[2]^*\to\HomolX[3]_*$.
\end{defn}

Note that the signs in both definitions are chosen in accordance with the usual sign heuristics, in which connecting homomorphisms are symbols of order 1:
\begin{align*}
\cobound[3](x\times y)&=(\cobound[1]x)\times y & \cobound[3](x\times y)&=(-1)^{\deg(x)}\cdot x\times (\cobound[2]y)
\\\bound[3](x/y)&=(\bound[1]x)/y & (\bound[1]x)/y&=(-1)^{\deg(x)}\cdot x/(\cobound[2]y)
\end{align*}

\begin{lem}\label{lem:crossslantcompatiblewithsuspension}
Let $(X,A)$ and $(Y,B)$ be pairs of coarse spaces in $\admissible$ and let $p_X\colon X\times Y\to X$, $p_Y\colon X\times Y\to Y$ be the canonical controlled projections.
If $I$ is a closed interval and $\cU\in\nbhfmlysgm(X,A;I)$, then 
\[\widetilde\cU\coloneqq(p_X\times\id_I)^*\cU\in\nbhfmlysgm((X,A)\times Y;I)\subset\nbhfmlysgm((X,A)\times (Y,B);I)\]
and $((X,A)\indexcross{\cU}(I,\partial I))\times (Y,B)$ is canonically coarsely equivalent to $((X,A)\times(Y,B))\indexcross{\widetilde\cU}(I,\partial I)$. Furthermore, the associated suspension homomorphisms are compatible with cross and slant products in the sense that the diagrams 
\[\xymatrix{
\CohomX[1]^{m+1}((X,A)\indexcross{\cU}(I,\partial I))\otimes\CohomX[2]^n(Y,B)
\ar[r]^{\times}\ar[d]^{\cosusp[1]\otimes\id}
&\CohomX[3]^{m+n+1}(((X,A)\times(Y,B))\indexcross{\widetilde\cU}(I,\partial I))
\ar[d]^{\cosusp[3]}
\\\CohomX[1]^m(X,A)\otimes\CohomX[2]^n(Y,B)
\ar[r]^{\times}
&\CohomX[3]^{m+n}((X,A)\times(Y,B))
}\]
and
\[\xymatrix{
\HomolX[1]_{m-1}((X,A)\times(Y,B))\otimes\CohomX[2]^n(Y,B)
\ar[d]^{\susp[1]\otimes\id}\ar[r]^-{/}
&\HomolX[3]_{m-n-1}(X,A)
\ar[d]^{\susp[3]}
\\\HomolX[1]_{m}(((X,A)\times(Y,B))\indexcross{\widetilde\cU}(I,\partial I))\otimes\CohomX[2]^n(Y,B)
\ar[r]^-{/}
&\HomolX[3]_{m-n}((X,A)\indexcross{\cU}(I,\partial I))
}\]
commute. Similarily, if If $J$ is a closed interval, $\cV\in\nbhfmlysgm(Y,B;J)$ and
\(\widetilde\cV\coloneqq(p_Y\times\id_J)^*\cV\),
then we have the diagrams
\[\xymatrix{
\CohomX[1]^{m}(X,A)\otimes\CohomX[2]^{n+1}((Y,B)\indexcross{\cV}(J,\partial J))
\ar[r]^{\times}\ar[d]^{\id\otimes\cosusp[2]}
&\CohomX[3]^{m+n+1}(((X,A)\times(Y,B))\indexcross{\widetilde\cV}(J,\partial J))
\ar[d]^{\cosusp[3]}
\\\CohomX[1]^m(X,A)\otimes\CohomX[2]^n(Y,B)
\ar[r]^{\times}
&\CohomX[3]^{m+n}((X,A)\times(Y,B))
}\]
and
\[\xymatrix{
\Estack{&\HomolX[1]_{m-1}((X,A)\times(Y,B))
\\\otimes&\CohomX[2]^{n+1}((Y,B)\indexcross{\cV}(J,\partial J))}
\ar[d]^-{\susp[1]\otimes\id}
\ar[r]^-{\id\otimes\cosusp[2]}
&\Estack{\HomolX[1]_{m-1}((X,A)\times&(Y,B))
\\\otimes\CohomX[2]^{n}&(Y,B)}
\ar[d]^-{/}
\\\Estack{\HomolX[1]_{m}(((X,A)\times(Y,B))\indexcross{\widetilde\cV}(J,\partial J))&
\\\otimes\CohomX[2]^{n+1}((Y,B)\indexcross{\cV}(J,\partial J))&}
\ar[r]^-{/}
&\HomolX[3]_{m-n-1}(X,A)
}\]
commuting up to the sign $(-1)^{m}$.
\end{lem}
\begin{proof}
The first statement about $\tilde\cU$ and the warped products is obvious. 

Let $I=[e_1,e_2]$.
If $\cU'$ is the refinement of $\cU$ with the properties of \Cref{lem:homotopydomaindeformationretraction} (such that in particular $X\times \{e_2\}$ is canonically a strong generalized coarse deformation retract of $X\indexcross{\cU'}I$),
then $\widetilde\cU'\coloneqq(p_X\times\id_I)^*(\cU')$ is a refinement of $\widetilde U$ with the analogous properties.
By naturality of the cross products under the coarse maps induced by these refinements, it is sufficient to prove the first diagram for $\cU'$ and $\widetilde\cU'$ instead of $\cU$ and $\widetilde\cU$.
Recall that for these refinements the suspension homomorphisms are even isomorphisms.

For this task we use the abbreviations
\begin{align*}
Z&\coloneqq X\times\partial I \cup A\times I
&\widetilde Z&\coloneqq X\times Y\times\partial I \cup (A\times Y\cup X\times B)\times I
\\C&\coloneqq X\times\{e_2\}\cup A\times I
&\widetilde C&\coloneqq X\times Y\times \{e_2\}\cup (A\times Y\cup X\times B)\times I
\end{align*}
where the left two are coarse subspaces of $X\indexcross{\cU'}I$ and the right two of $(X\times Y)\indexcross{\widetilde\cU'}I$.
Recall that by passing to the refinements $\cU'$ and $\widetilde\cU'$ of $\cU$ and $\widetilde\cU$, respectively, the suspension homomorphisms become isomorphisms inverse to the connecting homomorphisms asociated to the triples $(X\indexcross{\cU'}I,Z,C)$ and $((X\times Y)\indexcross{\widetilde\cU'}I,\widetilde Z,\widetilde C)$ under the isomorphisms induced by the excisions $(X,A)\cong(X,A)\times\{e_1\}\subset (Z,C)$ and $(X,A)\times(Y,B)\cong(X,A)\times(Y,B)\times\{e_1\}\subset (\widetilde Z,\widetilde C)$.
Note furthermore that the inclusion 
\[(Z,C)\times(Y,B)\subset (Z\times Y\cup(X\indexcross{\cU'}I)\times Y,C\times Y\cup(X\indexcross{\cU'}I)\times B)\cong(\widetilde Z,\widetilde C)\]
is an excision, too, because for all entourages $E$ of $X\indexcross{\cU'}I$ and $F$ of $Y$ we have
\begin{align*}
\Pen_{E\swapcross F}(Z\times Y)&\cap\Pen_{E\swapcross F}(C\times Y\cup(X\indexcross{\cU'}I)\times B)=
\\&=\Pen_E(Z)\times Y\cap (\Pen_E(C)\times Y\cup (X\indexcross{\cU'}I)\times \Pen_F(B))
\\&=\Pen_E(C)\times Y\cup\Pen_E(Z)\times \Pen_F(B)
\\&=\Pen_{E\swapcross F}(C\times Y\cup Z\times B)\,.
\end{align*}
Now consider the following diagram, in which all arrows not labeled as cross products or connecting homomorphisms are induced by inclusions.
\[\xymatrix{
\Estack{\CohomX[1]^m(&X,A)\\\otimes\CohomX[2]^n(&Y,B)}
\ar[rr]^-{\times}
&&\Estack{\CohomX[3]^{m+n}(&(X,A)\\&\times(Y,B))}
\\\Estack{\CohomX[1]^m(&Z,C)\\\otimes\CohomX[2]^n(&Y,B)}
\ar[u]_-{\cong}\ar[d]\ar[r]^-{\times}
&\Estack{\CohomX[3]^{m+n}(&(Z,C)\\\times&(Y,B))}
\ar[ur]
\ar[d]
&\CohomX[3]^{m+n}(\widetilde Z,\widetilde C)
\ar[u]_-{\cong}\ar[d]\ar[l]_-{\cong}
\\\Estack{\CohomX[1]^m(&Z)\otimes\\\CohomX[2]^n(&Y,B)}
\ar[d]^-{\cobound[1]\otimes\id}\ar[r]^-{\times}
&\CohomX[3]^{m+n}(Z\times(Y,B))
&\CohomX[3]^{m+n}(\widetilde Z,X\times B\times I)
\ar[l]_-{\cong}
\ar[d]^-{\cobound[3]}
\\\Estack{\CohomX[1]^{m+1}(&(X,A)\\&\indexcross{\cU'}(I,\partial I))\\\otimes\CohomX[2]^n(&Y,B)}
\ar[rr]^-{\times}
&&\Estack{\CohomX[3]^{m+n+1}(&((X,A)\\&\phantom{(}\times(Y,B))\\&\indexcross{\widetilde\cU'}(I,\partial I))}
}\]
The upper two third of the diagram commutes trivially or by naturality of the cross products.
Some of the arrows are marked as isomorphisms, because we know that they are induced by excisions. If we invert all of them, then the upper two third of the diagram still commutes and the lower third becomes commutative by the postulated compatibility of the cross products with connecting homomorphisms.
Thus, the cross products are compatible with the left and right vertical compositions, which are exactly the inverses to the suspension isomorphisms.

The claimed commutativity of the second and third diagram (up to the signs $+1$, $(-1)^m$, respectively) are proven completely analogously. The fourth diagram should better first be rewritten as
\[\xymatrix{
\HomolX[1]_{m-1}((X,A)\times(Y,B))
\ar[r]^-{/}\ar[d]^-{\susp[1]}
&\Hom(\CohomX[2]^n(Y,B),\HomolX[3]_{m-n-1}(X,A))
\ar[d]^{(\cosusp[2])^*}
\\\HomolX[1]_{m}(((X,A)\times(Y,B))\indexcross{\widetilde\cV}(J,\partial J))
\ar[r]^-{/}
&\Estack{\Hom(&\CohomX[2]^{n+1}((Y,B)\indexcross{\cV}(J,\partial J)),\\&\HomolX[3]_{m-n-1}(X,A))}
}\]
and then commutativity up to the sign $(-1)^m$ can also be proven in exactly the same manner.
\end{proof}

\section{Secondary cup and cap products from weak generalized controlled deformation retracts}
\label{sec:secondaryfromdeformations}

In this section we finally define our secondary cup and cap products and prove their basic properties. All constructions and proofs are closely modelled after \cite[Section 8]{WulffCoassemblyRinghomo}. The main difference is that here we work solely in the coarse geometric world, whereas in that reference the secondary product was constructed first in the topological set-up and then coarsified in a later section, which is similar to what we will do in \Cref{sec:transgression}.

\begin{defn}\label{defn:GWCDR}
We call a subset $A$ of a coarse space $X$ a \emph{weak generalized controlled deformation retract} if there is $\cU\in\nbhfmlysgm(X,A;[0,1])$ and a generalized controlled homotopy 
\begin{equation}\label{eq:generalizeddeformationpair}
H\colon (X,A)\indexcross{\cU} ([0,1],\{1\})\to (X,A)
\end{equation}
between the identity map $H_0\coloneqq H|_{(X,A)\times\{0\}}=\id_{(X,A)}$ and, as \eqref{eq:generalizeddeformationpair} already encodes, a controlled map $H_1\coloneqq H|_{(X,A)\times\{1\}}\colon (X,A)\to (A,A)\subset (X,A)$. 
\end{defn}
The word `generalized' refers to the usage of a generalized controlled homotopy and the word `weak' indicates that the restriction of $H_1\colon X\to A$ to $A$ is not necessarily the identity.
\begin{example}
\label{ex:conedeformations}
The big advantage of allowing this weakness is its flexibility, which can be seen, for example, in the case of open cones  over compact metric spaces (cf.\ \cite[Section 3.1]{EngelWulff}).
If $L$ is a closed subspace of a compact metric space $K$, then the open cone $\cO L$ is a weak generalized controlled deformation retract of the open cone $\cO$, where $H$ can simply be chosen to be a contraction along straight lines to the apex of the cone and we do not have to worry about reaching the identity on $A$ at time $1$ or any other time. 
\end{example}

\begin{example}
Let $G$ be a finitely generated group with word metric $d$ and neutral element $e$.
Assume that a generalized controlled contraction 
\[H\colon (G,\{e\})\indexcross{\cU} ([0,1],\{1\})\to (G,\{e\})\]
with $\cU=\{U_g\}_{g\in G}\in\nbhfmly(G;[0,1])$ is given. Then there are $R,S\in\N$ such that $H\times H$ maps $E_1\swapcross\Diag[{[0,1]}]$ into $E_R$ and $E_\cU$ into $E_S$. For each $g\in G$ we pick $0=t_{g,0}<t_{g,1}<\dots<t_{g,N_g}=1$ such that $(t_{g,n-1},t_{g,n})\in U_g$ and hence $d(H(g,t_{g,n-1}),H(g,t_{g,n}))\leq S$ for all $n=1,\dots, N_g$. Thus, we find a finally constant function $\gamma_g\colon \N\to G$ such that $\gamma_g(S\cdot n)= H(g,t_{g,n})$ for all $n=1,\dots, N_g$, $\gamma_g(m)=g$ for all $m\geq S\cdot N_g$ and $d(\gamma_g(m+1),\gamma_g(m))\leq 1$ for all $m\in\N$. It is now easy to see that $\gamma_g,\gamma_h$ are asynchronous $(R+3S)$-fellow-travellers for all $g,h\in G$ with $d(g,h)\leq 1$ and therefore $G$ is asynchronously combable as defined in \cite[Definition 2.6]{BridsonGilman_FormalLanguageTheory}.

It would be interesting to know if or under which conditions the converse also holds true, but it is not clear, how a construction of a generalized controlled contraction should work. The problem is that each two combing pathes in an asynchronously combable group can be reparametrized such that they stay close to each other, but all of these reparametrizations together are a priori not assumed to be compatible in any way.
\end{example}

Note that we did not include the collection $\cU$ and the homotopy $H$ itself into the data of weak generalized controlled deformation retracts. This is, because up to refining $\cU$ sufficiently the homotopy is unique up to another generalized controlled homotopy, as the special case $f=\id$ of the following lemma shows.

\begin{lem}\label{lem:homotopieshomotopic}
Let 
\begin{align*}
H\colon (X,A)\indexcross{\cU} ([0,1],\{1\})&\to (X,A)
\\H'\colon (X',A')\indexcross{\cU'} ([0,1],\{1\})&\to (X',A')
\end{align*}
be two generalized controlled homotopies associated to weak generalized controlled deformation retracts $A\subset X$ and $A'\subset X'$, respectively, and let $f\colon (X,A)\to (X',A')$ be a coarse map. Then there is a common refinement $\cV\in\nbhfmlysgm(X,A;[0,1])$ of $\cU$ and $(f\times\id)^*\cU'$ such that
\[f\circ H,H'\circ(f\times \id)\colon (X,A)\indexcross{\cV} ([0,1],\{1\})\to (X',A')\]
are generalizedly controlledly homotopic via a homotopy of coarse maps $\tilde H_t$ which agree with $f$ on $(X,A)\times\{0\}$.
\end{lem}

\begin{proof}
Consider the topological homotopy 
\[h\colon [0,1]\times [-1,1]\to [0,1]^2\,,\quad (s,t)\mapsto \begin{cases}(s(1+t),s)&t\leq 0\\(s,s(1-t))&t\geq 0\,.\end{cases}\]
According to \Cref{lem:warpedtrivialproperties}.\ref{item:warpedfunctorial} together with parts \ref{item:warpedfusion} and \ref{item:warpedfission} of \Cref{lem:warpedcartesianproducts} there are $\cV,\cV'$ such that 
\[\tilde H\coloneqq H'\circ(f\circ H\times\id_{[0,1]})\circ (\id_X\times h)\colon (X\indexcross{\cV}[0,1])\indexcross{\cV'}[-1,1]\to X\]
is a controlled map.
As $\cV$ can be chosen to be a refinement of $\cU\in\nbhfmlysgm(X,A;[0,1])$, we also have $\cV\in\nbhfmlysgm(X,A;[0,1])$.

We have $\tilde H(x,s,-1)=H'(f\circ H(x,0),s)=H'(f(x),s)$ and $\tilde H(x,s,+1)=H'(f\circ H(x,s),0)=f\circ H(x,s)$. 
Furthermore, for each $t\in[-1,1]$ we have $\tilde H(x,0,t)=H'(f\circ H(x,0),0)=f(x)$ and 
\[\tilde H(x,1,t)=\begin{cases}H'(f\circ H(x,1+t),1)\in A'&t\leq 0\\H'(\underbrace{f\circ H(x,1)}_{\in A},1-t)\in A'&t\geq 0\,.\end{cases}\]
As $\tilde H$ clearly maps $A\times [0,1]\times[-1,1]$ into $A'$, too, the claim follows.
\end{proof}

\begin{defn}
We call a triad $(X;A,B)$ in $\admissible$ a \emph{deformation triad} if $A,B\subset X$ are weak generalized controlled deformation retracts.
\end{defn}

\begin{lem}\label{lem:GammaXABdef}
If $(X;A,B)$ is a deformation triad and $H^A$ and $H^B$ are associated generalized controlled homotopies, then \begin{align*}
\Gamma^{X,A,B}\colon (X,&A\cup B)\indexcross{\cU}([-1,1],\{-1,1\})\to (X,A)\times (X,B)
\\(x,&t)\mapsto \begin{cases} (H^A(x,-t),x)& t\leq 0\\(x,H^B(x,t))&t\geq 0\end{cases}
\end{align*}
is a coarse map
for sufficently fine $\cU\in\nbhfmlysgm(X,A\cup B;[-1,1])$. In this case we say that $\cU$ \emph{supports} $\Gamma^{X,A,B}$.
\end{lem}
\begin{proof}
The collection $\cU$ is obtained in the obvious way from the collections 
\begin{align*}
\cU^A&\in \nbhfmlysgm(X,A;[0,1])\subset \nbhfmlysgm(X,A\cup B;[0,1])
\\\cU^B&\in \nbhfmlysgm(X,B;[0,1])\subset \nbhfmlysgm(X,A\cup B;[0,1])
\end{align*}
defining the coarse structures on the domains of $H^A$ and $H^B$, respectively, and it is clear that it satisfies the suspension condition.
The map $\Gamma$ is controlled, because $H^A,H^B$ are, and it is proper, because there is always one component of $\Gamma^{X,A,B}(x,t)$ equal to $x$.
\end{proof}

The map $\Gamma^{X,A,B}$ defined in this lemma will from now on be omnipresent. Therefore, it is convenient to introduce some abreviations to simplify the formulas.
\begin{nota}
For the remainder of this paper we shall use the letter $I$ exclusively for the interval $[-1,1]$. Furthermore, we define $I_+\coloneqq [0,1]$ and $I_-\coloneqq [-1,1]$.

Given any real number $t\in\R$, we let $t^+\coloneqq\max(t,0)$ and $t^-\coloneqq-\min(t,0)$. 
Then, for any map of the form $H\colon X\times I_+\to X$ we define the maps 
\[H_\pm\colon X\times I\to X\,,\quad (x,t)\mapsto H(x,t^\pm)\]
and for each $t\in[0,1]$ we write $H_t$ for $H(-,t)\colon X\to X$.
\end{nota}
Using this notation we can rewrite the map $\Gamma^{X,A,B}$ very concisely as
\[\Gamma^{X,A,B}=(H^A_-,H^B_+)\colon (x,t)\mapsto (H^A_{t^-}(x), H^B_{t^+}(x))\,.\]

\begin{defn}\label{def:secondarycupcap}
Let $(X;A,B)$ be a deformation triad in $\admissible$ and define $\Gamma^{X,A,B}$ as in \Cref{lem:GammaXABdef} with $\cU\in\nbhfmlysgm(X,A\cup B;I)$ supporting it.
\begin{itemize}
\item Given a cross product $\times\colon\CohomX[1]^*\otimes\CohomX[2]^*\to\CohomX[3]^*$ between coarse cohomology theories on $\admissible$ satisfying the strong homotopy axiom, the associated secondary cup product $\Cup$ is defined as $(-1)^m$ times the composition
\begin{align*}
\CohomX[1]^m(X,A)\otimes\CohomX[2]^n(X,B)&\xrightarrow{\times}\CohomX[3]^{m+n}((X,A)\times (X,B))
\\&\xrightarrow{(\Gamma^{X,A,B})^*}\CohomX[3]^{m+n}((X,A\cup B)\indexcross{\cU}(I,\partial I))
\\&\xrightarrow{\sigma^*}\CohomX[3]^{m+n-1}(X,A\cup B)\,.
\end{align*}
\item Given a slant product $/\colon\HomolX[1]_*\otimes\CohomX[2]^*\to\HomolX[3]_*$ between coarse homology and cohomology theories on $\admissible$ satisfying the strong homotopy axiom, the associated secondary cap product $\Cap$ is defined as $(-1)^{m+1}$ times the composition
\begin{align*}
\HomolX[1]_m(X,A\cup B)&\otimes\CohomX[2]^n(X,B)\to
\\&\xrightarrow{\sigma_*\otimes\id}\HomolX[1]_{m+1}((X,A\cup B)\indexcross{\cU}(I,\partial I))\otimes\CohomX[2]^n(X,B)
\\&\xrightarrow{(\Gamma^{X,A,B})_*\otimes\id}\HomolX[1]_{m+1}((X,A)\times (X,B))\otimes\CohomX[2]^n(X,B)
\\&\xrightarrow{/}\HomolX[3]_{m-n+1}(X,A)\,.
\end{align*}
\end{itemize}
\end{defn}

\begin{lem}[Naturality and well-definedness]
Let $(X;A,B)$, $(X';A',B')$ be deformation triads in $\admissible$ and let $f\colon X\to X'$ be a coarse map with $f(A)\subset A'$ and $f(B)\subset B'$.
Then the secondary cross and slant products are natural under $f$ in the sense that $f^*(x\Cup y)=(f^*x)\Cup(f^*y)$ and $f_*(x\Cap f^*y)=(f_*x)\Cap y$.

In particular, the special case $f=\id$ shows that the definition of the secondary cup and cap products is independent of the choice of $H^A,H^B$ and $\cU$.
\end{lem}

\begin{proof}
First we note that the secondary products are independent of the choice of $\cU$, because the suspension maps are natural under refinement of $\cU$.
Then, if we construct $\Gamma^{X,A,B}$ and $\Gamma^{X',A',B'}$ from arbitrary generalized coarse homotopies associated to the four deformation pairs,  \Cref{lem:homotopieshomotopic} tells us that the coarse maps 
\begin{align*}
(f\times f)\circ \Gamma^{X,A,B}\,,&\quad\Gamma^{X',A',B'}\circ (f\times\id_I)\colon
\\& (X,A\cup B)\indexcross{\cU}(I,\partial I)\to (X',A')\times (X',B')
\end{align*}
are generalized coarsely homotopic for sufficiently fine $\cU$ and hence induce the same map on $\CohomX[3]^*$ and $\HomolX[1]_*$. Naturality of the cross and slant products finishes the proofs of the claims.
\end{proof}

\begin{thm}[Graded commutativity of $\Cup$]
Let $\CohomX[1]^*$, $\CohomX[2]^*$, $\CohomX[3]^*$ be coarse cohomology theories on $\admissible$ satisfying the strong homotopy axiom. Assume that $\times\colon\CohomX[1]^*\otimes\CohomX[2]^*\to\CohomX[3]^*$ and $\times\colon\CohomX[2]^*\otimes\CohomX[1]^*\to\CohomX[3]^*$ are cross products which are mutually graded commutative in the sense that 
\[\xymatrix{
\CohomX[1]^m(X,A)\otimes\CohomX[2]^n(Y,B)\ar[d]^{\cong}_{\operatorname{flip}}\ar[r]^{\times}&\CohomX[3]^{m+n}((X,A)\times(Y,B))\ar[d]^{\cong}_{(\operatorname{flip})^*}
\\\CohomX[2]^n(Y,B)\otimes\CohomX[1]^m(X,A)\ar[r]^{\times}&\CohomX[3]^{m+n}((Y,B)\times(X,A))
}\]
commutes up to a sign $(-1)^{mn}$ for all pairs $(X,A)$ and $(Y,B)$ in $\admissible$. Then the associated secondary cup products are mutually graded commutative in the sense that 
\[\xymatrix{
\CohomX[1]^m(X,A)\otimes\CohomX[2]^n(X,B)\ar[d]^{\cong}_{\operatorname{flip}}\ar[r]^{\Cup}&\CohomX[3]^{m+n-1}(X,A\cup B)
\\\CohomX[2]^n(X,B)\otimes\CohomX[1]^m(X,A)\ar[ur]_{\Cup}&
}\]
commutes up to a sign $(-1)^{(m-1)(n-1)}$ for every deformation triad $(X;A,B)$ in $\admissible$.
\end{thm}

Of course, it is readily verified that if any cross products is given then there is a uniquely determined second one such that the hypothesis of the theorem is satisfied.
Note also that we really want to have this sign $(-1)^{(m-1)(n-1)}$, because the ``correct'' degree of $\CohomX^m$ with respect to secondary products is $m-1$.

\begin{proof}
Consider the diagram
\[\xymatrix{
\CohomX[1]^m(X,A)\otimes\CohomX[2]^n(X,B)\ar[r]_{\cong}^{\operatorname{flip}}\ar[d]^{\times}
&\CohomX[2]^n(X,B)\otimes\CohomX[1]^m(X,A)\ar[d]^{\times}
\\\CohomX[3]^{m+n}((X,A)\times(X,B))\ar[r]_{\cong}^{(\operatorname{flip})^*}\ar[d]^{(\Gamma^{X,A,B})^*}
&\CohomX[3]^{m+n}((X,B)\times(X,A))\ar[d]^{(\Gamma^{X,B,A})^*}
\\\CohomX[3]^{m+n}((X,A\cup B)\indexcross{\cV}([-1,1],\{-1,1\}))\ar[r]\ar[d]^{\sigma^*}
&\CohomX[3]^{m+n}((X,B\cup A)\indexcross{\cU}([-1,1],\{-1,1\}))\ar[d]^{\sigma^*}
\\\CohomX[3]^{m+n-1}(X,A\cup B)\ar[r]^{-\id}
&\CohomX[3]^{m+n-1}(X,B\cup A)
}\]
in which the third horizontal arrow is induced by the reflection map on $[-1,1]$ with $\cV$ chosen fine enough depending on $\cU$ as in \Cref{lem:suspensionsign}.\ref{item:intervalorientation}. The lower square commutes by \Cref{lem:suspensionsign}.\ref{item:suspensionreflection}, the middle square commutes and the upper square commutes by the sign $(-1)^{mn}$. 
Recalling that the left and right columns are the secondary cup products up to the signs $(-1)^m$ and $(-1)^n$, respectively, the claim follows.
\end{proof}

\begin{thm}[Associativity of $\Cup$]
Let $\CohomX[1]^*$, $\CohomX[2]^*$, $\CohomX[3]^*$, $\CohomX[4]^*$, $\CohomX[5]^*$, $\CohomX[6]^*$, be coarse cohomology theories satisfying the strong homotopy axiom and assume that there are cross products $\times\colon\CohomX[1]^*\otimes\CohomX[2]^*\to\CohomX[5]^*$, $\times\colon\CohomX[2]^*\otimes\CohomX[3]^*\to\CohomX[4]^*$, $\times\colon\CohomX[1]^*\otimes\CohomX[4]^*\to\CohomX[6]^*$, $\times\colon\CohomX[5]^*\otimes\CohomX[3]^*\to\CohomX[6]^*$ which are associative in the sense that the diagram
\[\xymatrix{
\CohomX[1]^*\otimes\CohomX[2]^*\otimes\CohomX[3]^*\ar[r]^-{\id\otimes\times}\ar[d]^{\times\otimes\id}
&\CohomX[1]^*\otimes\CohomX[4]^*\ar[d]^{\times}
\\\CohomX[5]^*\otimes\CohomX[3]^*\ar[r]^{\times}&\CohomX[6]^*
}\]
commutes in the obvious sense, that is,
\[\xymatrix{
\CohomX[1]^l(X,A)\otimes\CohomX[2]^m(Y,B)\otimes\CohomX[3]^n(Z,C)\ar[d]^-{\times\otimes\id}\ar[r]^{\id\otimes\times}
&\CohomX[1]^l(X,A)\otimes\CohomX[4]^{m+n}((Y,B)\times(Z,C))\ar[d]^{\times}
\\\CohomX[5]^{l+m}((X,A)\times(Y,B))\otimes\CohomX[3]^n(Z,C)\ar[r]^{\times}
&\CohomX[6]^{l+m+n}((X,A)\times (Y,B)\times (Z,C))
}\]
commutes for all pairs of coarse spaces $(X,A),(Y,B),(Z,C)$ in $\admissible$ and all $l,m,n\in\Z$. Then the corresponding secondary cup products are also associative in the sense that 
\[\xymatrix{
\CohomX[1]^l(X,A)\otimes\CohomX[2]^m(X,B)\otimes\CohomX[3]^n(X,C)\ar[d]^-{\Cup\otimes\id}\ar[r]^{\id\otimes\Cup}
&\CohomX[1]^l(X,A)\otimes\CohomX[4]^{m+n-1}(X,B\cup C)\ar[d]^{\Cup}
\\\CohomX[5]^{l+m-1}(X,A\cup B)\otimes\CohomX[3]^n(X,C)\ar[r]^{\Cup}
&\CohomX[6]^{l+m+n-2}(X,A\cup B\cup C)
}\]
commutes for every coarse space $X$ with subspaces $A,B,C\subset X$ such that $(X;A,B)$, $(X;B,C)$, $(X;A\cup B,C)$, $(X;A,B\cup C)$ are deformation triads in $\admissible$ and all $l,m,n\in\Z$.
\end{thm}

We omit the proof, because it is just a generalization of a coarse geometric analogue of \cite[Theorem 8.4]{WulffCoassemblyRinghomo} and we prefer to work out the details in the proof of the following theorem about the associativity of secondary cap with secondary cup products, which works completely analogously.

\begin{thm}[Associativity of $\Cap$ with $\Cup$]\label{thm:AssociativityCupCap}
Let $\HomolX[1]_*$, $\HomolX[5]_*$, $\HomolX[6]_*$ be coarse homology theories and $\CohomX[2]^*$, $\CohomX[3]^*$, $\CohomX[4]^*$, be coarse cohomology theories, all satisfying the strong homotopy axiom,  and assume that there are slant products $/\colon\HomolX[1]_*\otimes\CohomX[3]^*\to\HomolX[5]_*$, $/\colon\HomolX[1]_*\otimes\CohomX[4]^*\to\HomolX[6]_*$, $/\colon\HomolX[5]_*\otimes\CohomX[2]^*\to\HomolX[6]_*$ and a cross products $\times\colon\CohomX[2]^*\otimes\CohomX[3]^*\to\CohomX[4]^*$ which are associative in the sense that the diagrams
\[\xymatrix{
\Estack{&\HomolX[1]_k((X,A)\times(Y,B)\times(Z,C))\\&\otimes\CohomX[2]^m(Y,B)\otimes\CohomX[3]^n(Z,C)}\ar[r]^-{\id\otimes\times}
\ar[d]^{\id\otimes\operatorname{flip}}&\Estack{&\HomolX[1]_k((X,A)\times(Y,B)\times(Z,C))\\&\otimes\CohomX[4]^{m+n}((Y,B)\times(Z,C))}\ar[dd]^{/}
\\\Estack{&\HomolX[1]_k((X,A)\times(Y,B)\times(Z,C))\\&\otimes\CohomX[3]^n(Z,C)\otimes\CohomX[2]^m(Y,B)}\ar[d]^-{/\otimes\id}&
\\\HomolX[5]_{k-n}((X,A)\times(Y,B))\otimes\CohomX[2]^m(Y,B)\ar[r]^{/}&\HomolX[6]_{k-m-n}(X,A)
}\]
commute up to a sign $(-1)^{mn}$ for all pairs of coarse spaces $(X,A)$, $(Y,B)$, $(Z,C)$ in $\admissible$ and all $l,m,n\in\Z$. Then the corresponding secondary cap and cup products are also associative in the sense that 
\[\xymatrix{
\Estack{&\HomolX[1]_k(X,A\cup B\cup C)\\&\otimes\CohomX[2]^m(X,B)\otimes\CohomX[3]^n(X,C)}\ar[r]^-{\id\otimes\Cup}
\ar[d]^{\id\otimes\operatorname{flip}}
&\Estack{&\HomolX[1]_k(X,A\cup B\cup C)\\&\otimes\CohomX[4]^{m+n-1}(X,B\cup C)}\ar[dd]^{\Cap}
\\\Estack{&\HomolX[1]_k(X,A\cup B\cup C)\\&\otimes\CohomX[3]^n(X,C)\otimes\CohomX[2]^m(X,B)}\ar[d]^-{\Cap\otimes\id}&
\\\HomolX[5]_{k-n+1}(X,A\cup B)\otimes\CohomX[2]^m(X,B)\ar[r]^-{\Cap}&\HomolX[6]_{k-m-n+2}(X,A)
}\]
commutes  up to a sign $(-1)^{(m-1)(n-1)}$ for every coarse space $X$ with subspaces $A,B,C\subset X$ such that $(X;A,B)$, $(X;B,C)$, $(X;A\cup B,C)$, $(X;A,B\cup C)$ are deformation triads in $\admissible$ and all $l,m,n\in\Z$.
\end{thm}

Again, the signs are in accordance with the sign heuristics, because in the formulas $x/(y\times z)=(-1)^{\deg(y)\deg(z)}(x/z)/y$ and $x\Cap(y\Cup z)=(-1)^{(\deg(y)-1)(\deg(z)-1)}(x\Cap z)\Cap y$ the symbols $y,z$ are being interchanged.

The important step in the proofs of both theorems is the following lemma.

\begin{lem}\label{lem:associativity}
Let $X$ be a coarse space with subspaces $A,B,C\subset X$ such that $(X;A,B)$, $(X;B,C)$, $(X;A\cup B,C)$, $(X;A,B\cup C)$ are deformation triads in $\admissible$. Furthermore, let $J$ also denote the interval $[-1,1]$, just like $I$ does. Here we use the two different letters to distinguish the two copies of the same interval.
Then there are $\cU\in\nbhfmlysgm(X,A\cup B;I)$, $\cV\in\nbhfmlysgm(X,B\cup C;J)$ and $\cW\in \nbhfmlysgm(X,A\cup B\cup C;I\times J)$ such that the diagram
\[\xymatrix@C=20ex{
{\begin{matrix}(X,A\cup B\cup C)\indexcross{\cW}\\ \big((I,\partial I)\times (J,\partial J)\big)\end{matrix}}
\ar[r]^-{\Gamma^{X,A,B\cup C}\times\id_{(J,\partial J)}}
\ar[d]^-{\tilde\Gamma^{X,A\cup B,C}}
&{\begin{matrix}(X,A)\times\\\big((X,B\cup C)\indexcross{\cV}(J,\partial J)\big)\end{matrix}}
\ar[d]^-{\id_{(X,A)}\times\Gamma^{X,B,C}}
\\{\begin{matrix}\big((X,A\cup B)\indexcross{\cU}(I,\partial I)\big)\\\times(X,C)\end{matrix}}
\ar[r]^-{\Gamma^{X,A,B}\times\id_{(X,C)}}
&(X,A)\times(X,B)\times(X,C)
}\]
commutes up to generalized coarse homotopy, where $\tilde\Gamma^{X,A\cup B, C}$ denotes $\Gamma^{X,A\cup B,C}\times\id_{(I,\partial I)}$ up to the obvious identifications. That is, the horizontal $\Gamma$-maps take the $I$-interval as input and the vertical ones use $J$.
\end{lem}

\begin{proof}
Consider the map $H\colon X\times I\times J\times[0,4]\to X^3$ which is given by 
\[\begin{array}{rlcccr}
H(x,s,t,0)&\coloneqq(&H^A_{s^-}(x),&H^B_{t^-}\circ H^{B\cup C}_{s^+}(x),&H^C_{t^+}\circ H^{B\cup C}_{s^+}(x)&)
\\H(x,s,t,1)&\coloneqq(&H^A_{s^-}(x),&H^B_{s^+}\circ H^B_{t^-}\circ H^{B\cup C}_{s^+}(x),&H^C_{t^+}\circ H^{B\cup C}_{s^+}(x)&)
\\H(x,s,t,2)&\coloneqq(&H^A_{s^-}(x),&H^B_{s^+}\circ H^B_{t^-}(x),&H^C_{t^+}(x)&)
\\H(x,s,t,3)&\coloneqq(&H^A_{s^-}\circ H^{A\cup B}_{t^-}(x),&H^B_{s^+}\circ H^B_{t^-}\circ H^{A\cup B}_{t^-}(x) ,&H^{\cC}_{t^+}(x)&)
\\H(x,s,t,4)&\coloneqq(&H^A_{s^-}\circ H^{A\cup B}_{t^-}(x),& H^B_{s^+}\circ H^{A\cup B}_{t^-}(x) ,& H^{\cC}_{t^+}(x)&)
\end{array}\]
if the homotopy parameter is one of the integers between $0$ and $4$ and interpolate between these values in the obvious way on $[0,4]\setminus\{0,1,2,3,4\}$. So for example, for $u\in[0,1]$ we have
\[H(x,s,t,u)=(H^A_{s^-}(x),H^B_{u\cdot s^+}\circ H^B_{t^-}\circ H^{B\cup C}_{s^+}(x),H^C_{t^+}\circ H^{B\cup C}_{s^+}(x))\,.\]
It is straightforward to verify that it maps the set
\[\big((A\cup B\cup C)\times I\times J
\cup X\times (\partial I\times J \cup I\times\partial J)\big)\times[0,4]\]
into
\( (A\times X^2)\cup(X\times B\times X)\cup(X^2\times C)\).
For $s=-1$ and $t=+1$ this follows easily from $H^A_1$ having image in $A$ and $H^C_1$ having image in $C$, respectively.
We check the other cases exemplarily for $u\in[0,1]$, the rest being similar.
\begin{itemize}
\item If $t=-1$, then the middle component lies in $\im(H^B_{u\cdot s^+}\circ H^B_1)\subset H^B_{u\cdot s^+}(B)\subset B$.
\item If $s=+1$ or $x\in B\cup C$, then $H^{B\cup C}_{s^+}(x)$ lies in $B$ or in $C$. In the first case the middle component lies in $B$ and in the second case the third component lies in $C$.
\item If $x\in A$, then the first component is $H^A_{s^-}(x)\in A$.
\end{itemize}
Thus, the construction gives rise to a map of pairs
\[H\colon (X,A\cup B\cup C)\times(I,\partial I)\times (J,\partial J)\times[0,4]\to(X,A)\times(X,B)\times(X,C)\]
and we have
\begin{align*}
(\id_{(X,A)}\times&\Gamma^{X,B,C})\circ(\Gamma^{X,A,B\cup C}\times\id_{(J,\partial J)})(x,s,t)=
\\&=(\id_{X,A}\times\Gamma^{X,B,C}) (H^A_{s^-}(x),H^{B\cup C}_{s^+}(x),t)
\\&=(H^A_{s^-}(x),\,H^B_{t^-}\circ H^{B\cup C}_{s^+}(x),\,H^C_{t^+}\circ H^{B\cup C}_{s^+}(x))=H(x,s,t,0)
\end{align*}
as well as
\begin{align*}
(\Gamma^{X,A,B}&\times\id_{X,C})\circ \tilde\Gamma^{X,A\cup B,C}(x,s,t)=
\\&=(\Gamma^{X,A,B}\times\id_{X,C}) (H^{A\cup B}_{t^-}(x),s,H^{\cC}_{t^+}(x))
\\&=(H^A_{s^-}\circ H^{A\cup B}_{t^-}(x),\, H^B_{s^+}\circ H^{A\cup B}_{t^-}(x) ,\,H^{\cC}_{t^+}(x))=H(x,s,t,4)
\end{align*}
for all $x\in X,s,t\in [-1,1]$.
The collections of neighborhoods $\cU$, $\cV$ are chosen as in \Cref{lem:GammaXABdef} and by using \Cref{lem:warpedtrivialproperties,lem:warpedcartesianproducts} in the same way as before we also obtain as sufficiently fine $\cW$ such that the maps in the diagram are coarse maps and $H$ is a generalized controlled homotopy between the two compositions. Note furthermore that for all $s,t,u$ at least one component of $H(x,s,t,u)$ is equal to $x$ and therefore $H$ is also proper.
\end{proof}

\begin{proof}[Proof of \Cref{thm:AssociativityCupCap}]
Consider the diagram in Figure \ref{fig:ProofofAssociativity}, in which the single squares commute up to the signs which are displayed within them. We explain it in a bit more detail, starting at the bottom right corner:
\begin{sidewaysfigure}
$\xymatrix@C=15ex@R=10ex{
\HomolX[1]_k(X,A\cup B\cup C)
\ar[r]^{\susp[1]}
\ar[d]^{\susp[1]}
\ar@{}[dr]|*+[o][F-]{-1}
&{\Estack{\HomolX[1]_{k+1}(&(X,A\cup B\cup C)\\&\indexcross{\cU'}(I,\partial I))}}
\ar[d]^{\susp[1]}
\ar[r]^-{(\Gamma^{X,A,B\cup C})_*}
\ar@{}[dr]|*+[o][F-]{+1}
&{\Estack{\HomolX[1]_{k+1}(&(X,A)\times \\ &(X,B\cup C))}}
\ar[r]^-{/}
\ar[d]^{\susp[1]}
&{\Estack{\Hom(\CohomX[4]^{m+n-1}(X,B\cup C)&,\\\HomolX[6]_{k-m-n+2}(X,A)&)}}
\ar[d]^{(\cosusp[4])^*}
\\%Zweite Zeile
{\Estack{\HomolX[1]_{k+1}(&(X,A\cup B\cup C)\\&\indexcross{\cV'}(J,\partial J))}}
\ar[r]^{\susp[1]}
\ar[d]_-{(\Gamma^{X,A\cup B,C})_*}
\ar@{}[dr]|*+[o][F-]{+1}
&{\Estack{\HomolX[1]_{k+1}((X,A\cup &B\cup C)\\\indexcross{\cW}(&(I,\partial I)\\&\times(J,\partial J)))}}
\ar[d]_-{(\Gamma^{X,A\cup B,C})_*}
\ar[r]^-{(\Gamma^{X,A,B\cup C})_*}
\ar@{}[dr]|*+[o][F-]{+1}
&{\Estack{\HomolX[1]_{k+2}((&X,A)\times\\ (&(X,B\cup C)\\&\indexcross{\cV}(J,\partial J)))}}
\ar[d]_-{(\Gamma^{X,B,C})_*}
\ar[r]^-{/}
\ar@{}[dr]|*+[o][F-]{+1}
\ar@{}[ur]|*++[o][F-]{(-1)^{k+2}}
&{\Estack{\Hom(\CohomX[4]^{m+n}((X,B\cup C)&\\\indexcross{\cV}(J,\partial J)&),\\\HomolX[6]_{k-m-n+2}(X,A&))}}
\ar[d]_-{((\Gamma^{X,B,C})^*)^*}
\\%Dritte Zeile
{\Estack{\HomolX[1]_{k+1}(&(X,A\cup B) \\ &\times(X,C))}}
\ar[r]^{\susp[1]}
\ar[d]_-{/}
\ar@{}[dr]|*+[o][F-]{+1}
&{\Estack{\HomolX[1]_{k+2}((&(X,A\cup B)\\&\indexcross{\cU}(I,\partial I)) \\ &\times(X,C))}}
\ar[r]^-{(\Gamma^{X,A,B})_*}
\ar[d]_-{/}
\ar@{}[dr]|*+[o][F-]{+1}
&{\Estack{\HomolX[1]_{k+2}(&(X,A) \\ &\times(X,B)\\&\times(X,C))}}
\ar[d]_-{/}
\ar[r]^-{/}
\ar@{}[dr]|*++[o][F-]{(-1)^{mn}}
&{\Estack{\Hom(\CohomX[4]^{m+n}((X,B)&\\\times (X,C)&),\\\HomolX[6]_{k-m-n+2}(X,A&))}}
\ar[d]_-{(\times)^*}
\\%Vierte Zeile
{\Estack{\Hom(\CohomX[3]^n(&X,C),\\\HomolX[5]_{k-n+1}(&X,A\cup B))}}
\ar[r]^-{(\susp[1])_*}
&{\Estack{\Hom(\CohomX[3]^n(&X,C),\\\HomolX[5]_{k-n+2}(&(X,A\cup B) \\ &\indexcross{\cU}(I,\partial I)))}}
\ar[r]^-{(\Gamma^{X,A,B})_*}
&{\Estack{\Hom(\CohomX[3]^n(&X,C),\\\HomolX[5]_{k-n+2}(&(X,A) \\ &\times(X,B)))}}
\ar[r]^-{/}
&{\Estack{\Hom(\CohomX[2]^m(X,B)&\\\otimes\CohomX[3]^n(X,&C),\\\HomolX[6]_{k-m-n+2}(X,&A))}}
}$\caption{Proving Theorem \ref{thm:AssociativityCupCap}.}
\label{fig:ProofofAssociativity}
\end{sidewaysfigure}

Commutativity of the bottom right square up to the sign $(-1)^{mn}$ is the postulated associativity between the given slant and cross products. 

Next, $\cU\in\nbhfmlysgm(X,A\cup B;I)$ and $\cV\in\nbhfmlysgm(X,B\cup C;J)$ are chosen fine enough such that $\Gamma^{X,A,B}\colon (X,A\cup B)\indexcross{\cU}(I,\partial I)\to (X,A)\times (Y,B)$ and $\Gamma^{X,B,C}\colon (X,B\cup C)\indexcross{\cV}(J,\partial J)\to (X,B)\times (Y,C)$ become coarse maps, and then the middle bottom an middle right square commute by naturality of the cross and slant products, respectively.

The top right and bottom left square commute up to the signs $(-1)^{k+2}$, $+1$, respectively, by \Cref{lem:crossslantcompatiblewithsuspension}. Commutativity of the middle square follows from \Cref{lem:associativity} for a certain $\cW\in\nbhfmlysgm(X,A\cup B\cup C;I\times J)$ which depends on $\cU$ and $\cV$.

\Cref{lem:warpedcartesianproducts} tells us that by refining $\cW$ we can assume that $X\indexcross{\cW}(I\times J)$ is canonically coarsely equivalent to $(X\indexcross{\cU'}I)\indexcross{\cV'}J$ and $(X\indexcross{\cV'}J)\indexcross{\cU'}I$ for some $\cU'\in\nbhfmlysgm(X,A\cup B\cup C;I)$ and $\cV'\in\nbhfmlysgm(X,A\cup B\cup C;J)$ and by additionally chosing them finer than $\cU,\cV$, respectively, we can ensure that
\begin{align*}
\Gamma^{X,A,B\cup C}&\colon (X,A\cup B\cup C)\indexcross{\cU'}(I,\partial I)\to(X,A)\times(X,B\cup C)\quad\text{and}
\\\Gamma^{X,A\cup B,C}&\colon (X,A\cup B\cup C)\indexcross{\cV'}(J,\partial J)\to(X,A\cup B)\times(X,C)
\end{align*}
are coarse maps.
Then the middle top and middle left square commute by naturality of the suspension maps and the top left square commutes up to the sign $-1$ by \Cref{lem:suspensionsign}.\ref{item:doublesuspension}.

The maps on the left hand side compose to $(-1)^{k+1}\cdot \Cap$, the maps at the bottom compose to $(-1)^{k-n+2}\cdot \Cap$, the maps at the top compose to $(-1)^{k+1}\cdot \Cap$ and the maps on the right compose to $((-1)^m\cdot \Cup)^*$.
Collecting all the signs mentioned above we end up with $(-1)^{(m-1)(n-1)}$, which was to be shown.
\end{proof}

\section{Secondary cup and cap products on ordinary coarse \mbox{(co-)}homology}
\label{sec:ordinary}

The first secondary product in coarse geometry ever was the one Roe constructed on his coarse cohomology $\HX^*$ in \cite[Section 2.4]{RoeCoarseCohomIndexTheory}. The multiplication maps $\HX^m(X)\otimes\HX^n(X)\to\HX^{m+n-1}(X)$ are explicitely constructed on the level of cocycles, but they are only well-defined for $m,n\geq 1$. 
The issues with well-definedness disappear if one considers the secondary cup product not on the absolute cohomology groups but on the groups $\HX^*(X,\{o\})$ relative to a base-point $o\in X$ instead.

Roe actually only considered the absolute groups for proper metric spaces and with coefficients in $\R$, but it is indeed straightforward to generalize them to a coarse cohomology theory $\HX^*(\blank,\blank;M)$ on the admissible category of pairs of countably generated coarse spaces, which happens to satisfy \Cref{ass:assumptiononadmissible}, for any choice of abelian group $M$ as coefficients. It is called the \emph{ordinary coarse cohomology} and satisfies the strong homotopy axiom, see \cite[Section 4.2]{wulff2020equivariant}.

In \Cref{defn:RoeSecCup} below we also give a straightforward generalization of Roe's secondary product to maps
\[*\colon \HX^m(X,A;M_1)\otimes\HX^n(X,B;M_2)\to\HX^{m+n-1}(X, A\cup B;M_3)\]
for all excisive triads $(X;A,B)$ of coarsely connected countably generated coarse spaces and abelian groups $M_1,M_2,M_3$ together with a bilinear map $M_1\times M_2\to M_3\colon (r,s)\mapsto r\cdot s$, e.\,g.\ $M_1=M_2=M_3=R$ could be a ring and $\cdot$ the multiplication map on $R$.
However, it should be pointed out that we use different signs in our definition than Roe did in his, because Roe's claim in \cite[Example 5.28(v)]{RoeCoarseCohomIndexTheory} is only true for our corrected choice of signs.

The purpose of this section is to show that the secondary product $*$ agrees with our secondary cup product $\Cup$ obtained from a canonical cross product if $(X;A,B)$ is an excisive deformation triad.
Note that it is somewhat of a mystery why the two products $*$ and $\Cup$ can only be constructed under different, seemingly unrelated assumptions, namely excisive triads versus deformation triads, although the products are equal.

Dual to the ordinary coarse cohomology there is also a ordinary coarse homology $\HX_*(\blank,\blank;M)$. In a special case, its definition has already been given in \cite[page 453]{YuCyclicCohomology}, and the general version was developed in \cite[Section 6.3]{BunkeEngel_homotopy}. We refer to \cite[Section 4.1]{wulff2020equivariant} for proofs of basic properties, including that it is a coarse homology theory satisfying the strong excision axiom.
A secondary product between ordinary coarse homology and ordinary coarse cohomology can be defined in a similar manner to Roe's product at the level of cycles. We will show that it agrees with the secondary cap product obtained from a canonical slant product.

We start by recalling the definition of ordinary coarse cohomology, which can be written down for arbitrary coarse spaces, but only yields a coarse cohomology theory for countably generated ones.

\begin{defn}
Let $X$ be a coarse space and $M$ an abelian group. For $m\in\N$ we define $\CX^m(X;M)$ as the group of all functions $\varphi\colon X^{m+1}\to M$ whose support \(\supp(\varphi)\) intersects each penumbra of the multidiagonal in a precompact subset, i.\,e.\ a finite union of bounded subsets.
For $m\in\Z\setminus\N$ we define $\CX^m(X;M)\coloneqq 0$. 
These groups together with the Alexander--Spanier coboundary maps $\delta\colon\CX^m(X;M)\to\CX^{m+1}(X;M)$, that is
\begin{equation}\label{eq:AScoboundary}
\delta\varphi(x_0,\dots,x_{m+1})\coloneqq\sum_{i=0}^{m+1}(-1)^i\varphi(x_0,\dots,x_{i-1},x_{i+1},\dots,x_{m+1})\,,
\end{equation}
constitute a cochain complex.
If $A\subset X$ is a subspace, then we define the cochain complex $\CX^*(X,A;M)$ as the kernel of the surjectiv cochain map $\CX^*(X;M)\to \CX^*(A;M)$ and its cohomology is the \emph{ordinary coarse cohomology} $\HX^*(X,A;M)$ of the pair $(X,A)$ with coefficients in $M$.
\end{defn}

In the upcoming constructions, excisiveness enters the game through the following subcomplexes:
Given a countably generated coarse space $X$ with subspaces $A,B\subset X$ we define 
\[\CX^*(X,A+B;M)\coloneqq \ker\big(\CX^*(X;M)\to\CX^*(A,M)\oplus\CX^*(B;M)\big)\,.\]
It contains $\CX^*(X,A\cup B;M)$ as a subcomplex.

\begin{lem}\label{lem:cochaininclusion}
Assume that $(X;A,B)$ is an excisive triad.
Then the inclusion
\begin{equation}\label{eq:cochaininclusion}
\iota\colon\CX^*(X,A\cup B;M)\xrightarrow{\subset}\CX^*(X,A+B;M)
\end{equation} 
is a cochain homotopy equivalence.
\end{lem}

\begin{proof}
As the coarse structure on $X$ is assumed to be countably generated, there is a sequence $E_0\subset E_1\subset E_2\subset\dots$ of entourages such that each entourage is contained in one of them. Let $F_0\subset F_1\subset F_2\subset \dots$ be the corresponding sequence of entourages given to us by the excisiveness condition of $A,B$, that is, 
$\Pen_{E_n}(A)\cap\Pen_{E_n}(B)\subset \Pen_{F_n}(A\cap B)$ for all $n\in\N$. We may assume that all $E_n,F_n$ are symmetric.

To nonempty each finite subset $Q\subset X$ we assign a ``barycenter'' $b(Q)\in X$ as follows: 
If $Q\subset A\cup B$ and $Q\cap A\not=\emptyset\not=Q\cap B$, then we let $n\in \N$ be the smallest number such that $Q\times Q\subset E_n$. It follows that $Q\subset \Pen_{E_n}(A)\cap\Pen_{E_n}(B)\subset \Pen_{F_n}(A\cap B)$ and we pick $b(Q)\in A\cap B$ such that $Q\cap \Pen_{F_n}(\{b(Q)\})\not=\emptyset$ and hence $Q\subset \Pen_{E_n\circ F_n}(\{b(Q)\})$. In all other cases we just pick any $b(Q)\in Q$. We may assume $b(\{q\})=q$ for all $q\in X$, e.\,g.\ simply by demanding $E_0=F_0=\Diag[X]$.

Now, the ``barycentric subdivision'' of $\mathbf{x}=(x_0,\dots,x_m)\in X^{m+1}$ is the collection of $(m+1)$-tuples
\[\mathbf{x}_\sigma\coloneqq \big(x_{\sigma(0)}, b(\{x_{\sigma(0)},x_{\sigma(1)}\}),b(\{x_{\sigma(0)},x_{\sigma(1)},x_{\sigma(2)}\}),\dots, b(\{x_{\sigma(0)},\dots,x_{\sigma(m)}\})\big)\]
for all permutations $\sigma\in S_{m+1}$.
Note that the barycenters have been chosen in such a way that if every two points of $x$ are at most $E_n$ apart, then every two points of each $\mathbf{x}_\sigma$ are at most $F_n\circ E_n\circ E_n\circ F_n$ apart. 
In other words: for each penumbra $P$ of the multidiagonal in $X^{p+1}$ there is another penumbra $P'$ such that $\mathbf{x}\in P\implies \mathbf{x}_\sigma\in P'$. This implies that for each coarse cochain $\varphi\in \CX^m(X;M)$ the map
\[S\varphi\colon X^{m+1}\to M\,,\quad \mathbf{x}\mapsto\sum_{\sigma\in S_{m+1}}\pm\varphi(\mathbf{x}_\sigma)\]
is also a coarse cochain. Note furthermore that if $\mathbf{x}\in (A\cup B)^{m+1}$, then $\mathbf{x}_\sigma\in A^{m+1}\cup B^{m+1}$ for each $\sigma$, so $S\varphi\in \CX^m(X,A\cup B;M)$ for all $\varphi\in \CX^m(X,A+B;M)$.

Now, we all know from our first course on algebraic topology that there is a certain set of signs which nobody dares to write down for which the standard calculations show that $S\colon \CX^*(X,A+B;M)\to \CX^*(X,A\cup B;M)$ is in fact a cochain map. Furthermore, the same properties of our barycenters ensure that the standard construction known from algebraic topology also gives us cochain homotopies $S\circ \iota\simeq \id$ and $\iota\circ S\simeq\id$.
\end{proof}

We can apply the lemma right away to  the excisive triad from \Cref{lem:crossproductexcisiveness}.
It is then straightforward to check that we really have a cochain map in the following definition and obtain a well-defined cross product which satisfies the axioms of \Cref{defn:crossproducts}.

\begin{defn}
Let $(X,A)$ and $(Y,B)$ be countably generated coarsely connected coarse spaces. 
Then the cross product
\begin{align*}
\HX^m(X,A;M_1)\otimes \HX^n(Y,B;M_2)&\to \HX^{m+n}(X\times Y,A\times Y+X\times B;M_3)
\\&\cong \HX^{m+n}((X,A)\times (Y,B);M_3)
\end{align*}
is induced by the cochain map
\begin{align*}
\CX^*(X,A;M_1)\otimes \CX^*(Y,B;M_2)&\to \CX^*(X\times Y,A\times Y+X\times B;M_3)
\end{align*}
which takes the tensor product of $\varphi\in\CX^m(X,A;M_1)$ and $\psi\in\CX^n(Y,B;M_2)$ to 
\[\varphi\times\psi\colon \big((x_0,y_0),\dots,(x_{m+n},y_{m+n})\big)\mapsto \varphi(x_0,\dots, x_m)\cdot \psi(y_m,\dots,y_{m+n})\,.\]
\end{defn}

Thus, \Cref{def:secondarycupcap} gives rise to a secondary cup product 
\[\HX^m(X,A;M_1)\otimes\HX^n(X,B;M_2)\to\HX^{m+n-1}(X,A\cup B;M_3)\]
on ordinary coarse cohomology for all deformation triads $(X;A,B)$ of countably generated coarse spaces.

Next, we introduce the generalization of Roe's secondary product. To this end, let $\varphi\cup\psi$ denote the usual cup product of functions $\varphi\colon X^{m+1}\to M_1$ and $\psi\colon X^{m+1}\to M_2$, that is, 
\[\varphi\cup\psi\colon X^{m+n+1}\to M_3\,,\quad (x_0,\dots,x_{m+n})\mapsto \varphi(x_0,\dots,x_m)\cdot\psi(x_m,\dots,x_{m+n})\,.\]
Furthermore, for any function $\varphi$ on $X^{m+1}$ and $o\in X$ we define the function $s_o\varphi\colon (x_1,\dots, x_m)\mapsto \varphi(o,x_0,\dots, x_m)$ on $X^m$.

\begin{defn}\label{defn:RoeSecCup}
Let $(X;A,B)$ be an excisive triad of coarsely connected countably generated coarse spaces with $A,B$ nonempty and choose base-points $a\in A$, $b\in B$.
Then the \emph{Roe secondary (cup) product} of two coarse cohomology classes $[\varphi]\in\HX^m(X,A;M_1)$ and $[\psi]\in\CX^n(X,B;M_2)$ is the class
$[\varphi]*[\psi]\in\HX^{m+n-1}(X,A\cup B;M_3)$ represented by the coarse cocycle
\begin{equation*}\varphi*\psi\coloneqq(-1)^{m+1}(s_a\varphi)\cup\psi+\varphi\cup (s_b\psi)\in \CX^{m+n-1}(X,A+B;M_3)\,.
\end{equation*}
\end{defn}

Although $s_a\varphi,s_b\psi$ do not satisfy the support condition of coarse cochains, $\varphi*\psi$ clearly does and lies in the claimed cochain group.
Note that here we have used that $X$ is coarsely connected, because otherwise $(a,x_0,\dots,x_{m-1})$ or $(b,x_0,\dots,x_{n-1})$ lies outside of all penumbras of the diagonal if the other points do not lie in the same coarse components as $a$ or $b$, respectively.

It is readily verified that $\varphi*\psi$ is even a cocycle and that the resulting cohomology class is independent of the choice of the base-points $a,b$.
The details of the calculations work exactly as in the proof of \cite[Proposition 2.33]{RoeCoarseCohomIndexTheory},
just that our definition of the product differs from Roe's, as mentioned above, by the sign $(-1)^{m+1}$. 
What is needed in the calculations is that the  Alexander--Spanier coboundary map \eqref{eq:AScoboundary} satisfies $\cobound (\varphi\cup\psi)=(\cobound \varphi)\cup\psi+ (-1)^m \varphi\cup(\cobound\psi)$ and we have $\cobound s_a\varphi+s_a\cobound\varphi=\varphi$ for $m\geq 1$ as well as $\cobound s_a\varphi+s_a\cobound\varphi=\varphi-\varphi(a)$ for $m=0$ and similarily for $s_b,\psi$. 
The summands $\varphi(a)$ if $m=0$ and $\psi(b)$ if $n=0$ appearing in the latter formulas are actually the reason why Roe's secondary product is not well defined on the absolute cohomology groups in these degrees, but in our case they vanish because $\varphi,\psi$ are cocycles relative to $A,B$, respectively.

\begin{thm}\label{lem:secprodcomparison}
Let $(X;A,B)$ be an excisive deformation triad of coarsely connected countably generated coarse spaces.
Then the associated secondary products $*$ and $\Cup$ agree.
\end{thm}

\begin{proof}
The first step in the proof is to reformulate the suspension homomorphism in a Mayer--Vietoris like fashion. As before, we write $I\coloneqq [-1,1]$, $I_+\coloneqq [0,1]$, $I_-\coloneqq [-1,0]$ and let $\cU\in\nbhfmlysgm(X,A\cup B;I)$ be a collection which supports $\Gamma^{X,A,B}$.

Then \Cref{lem:homotopydomaindeformationretraction} provides us with a refinement $\cU'\in\nbhfmlysgm(X,A\cup B;I)$ of $\cU$ such that in particular all of the inclusions
\[X\times\{-1\}\subset X\indexcross{\cU'}I_-\subset X\indexcross{\cU'}I\supset X\indexcross{\cU'}I_+\supset X\times\{1\}\]
are strong generalized coarse deformation retracts in the canonical way and, furthermore,
\[f\colon X\indexcross{\cU'} I_+\to X\indexcross{\cU'}I\,,\quad (x,s)\mapsto (x,2s-1)\]
is a coarse equivalence which is canonically generalized coarsely homotopic to the inclusion.
We then obtain the diagram
\[\xymatrix{
\HX^*((X,A\cup B)\indexcross{\cU'}(I,I_-\cup\{1\});M_3)\ar[r]_-{(\mathrm{incl.})^*}\ar[d]^-\cong_-{(\mathrm{incl.})^*}
&\HX^*((X,A\cup B)\indexcross{\cU'}(I,\partial I);M_3)\ar[d]^{\cong}_{\sigma^*}\ar[dl]^-{f^*}
\\\HX^*((X,A\cup B)\indexcross{\cU'}(I_+,\partial I_+);M_3)\ar[r]_-{\cong}^-{\sigma^*}
&\HX^{*-1}(X,A\cup B;M_3)
}\]
where the upper left triangle commutes by homotopy invariance and the lower right triangle commutes by naturality of the suspension homomorphisms (\Cref{lem:suspensionhomomorphism}). The two suspension homomorphisms and the left vertical arrow are in fact isomorphisms by the above-mentioned properties of $\cU'$. 
Thus, all other arrows in the diagram are also isomorphisms.

Now, as $(X;A,B)$ was assumed to be an excisive triad, so are 
\begin{align*}
(X\indexcross{\cU'} I_+\,;\quad X\times\{+1\}\cup B\times I_+\,,\quad A\times I_+)&\qquad\text{and}
\\(X\indexcross{\cU'} I_-\,;\quad X\times\{-1\}\cup A\times I_-\,,\quad B\times I_-)&\,.
\end{align*}
Thus, \Cref{lem:cochaininclusion} and the choice of $\cU'$ imply that the cohomologies of the complexes
\begin{align*}
\CX^*_+&\coloneqq \CX^*(X\indexcross{\cU'}I_+, X\times\{1\}\cup B\times I_++A\times I_+;M_3)
\\\CX^*_-&\coloneqq \CX^*(X\indexcross{\cU'}I_-, X\times\{-1\}\cup A\times I_-+B\times I_-;M_3)
\end{align*}
vanish and hence, the induced upper horizontal arrow $\kappa'$ in the commutative diagram with exact columns 
\[\xymatrix@C=5em{
0\ar[d]&0\ar[d]
\\\CX^*((X,A\cup B)\indexcross{\cU'}(I_+,\partial I_+);M_3)\ar[d]\ar[r]^-{\kappa'}
&\CX^*_\cup\ar[d]
\\\CX^*((X,A\cup B)\indexcross{\cU'}(I_+,\{1\});M_3)\ar[r]^-{(\mathrm{incl.},0)}\ar[d]
&\CX^*_+\oplus\CX^*_-
\ar[d]_{(\iota_+)^*-(\iota_-)^*}
\\ \Estack{\CX^*(&X\times\partial I_+\cup (A\cup B)\times I_+,\\&X\times \{1\}\cup (A\cup B)\times I_+;M_3)}\ar[r]^-{(\iota_+)^*}\ar[d]
&\CX^*(X,A+B)\ar[d]
\\0&0
}\]
induces an isomorphism on cohomology, because the cohomology of the middle cochain complexes vanish and the bottom horizontal map is induced by an excision. Here, $\CX^*_\cup$ is simply the kernel of the map below it and $\iota_\pm\colon X\to X\indexcross{\cU'_\pm}I_\pm$ denote the inclusions as the subspace $X\times\{0\}$.

Furthermore, the cochain complexes $\CX^*((X,A\cup B)\indexcross{\cU'}(I,I_-\cup\{1\});M_3)$ and $\CX^*((X,A\cup B)\indexcross{\cU'}(I,\partial I);M_3)$ map canonically to $\CX^*_\cup$ and we obtain a diagram
\[\xymatrix@C=7em@R=3em{
\CX^*((X,A\cup B)\indexcross{\cU'}(I,\partial I);M_3)
\ar[dd]_{f^*}^{\cong}\ar[drr]^{\kappa}&&
\\&{}\save[]+<-4em,0em>*{\CX^*((X,A\cup B)\indexcross{\cU'}(I,I_-\cup\{1\});M_3)}\ar[ul]_{\cong}\ar[dl]^{\cong}\ar[r]\restore&\CX^*_\cup
\\\CX^*((X,A\cup B)\indexcross{\cU'}(I_+,\partial I_+);M_3)
\ar[urr]_{\cong}^{\kappa'}&&
}\]
where the left square triangle commutes up to cochain homotopy (see \cite[Lemma 4.9]{wulff2020equivariant}) and the right two triangles commute trivially.
In particular, the outer triangle commutes in cohomology and combining this with the naturality of connecting homomorphisms under $f^*$ and the diagram of exact sequences above, we see that the diagram
\[\xymatrix{
\Estack{\HX^{*-1}(&X\times\partial I\cup (A\cup B)\times I,\\&X\times \{1\}\cup (A\cup B)\times I)}
\ar[d]_{\cobound}^{\cong}\ar[r]^-{\cong}
&\HX^{*-1}(X,A+ B)\ar[d]_{\cobound}^{\cong}
\\\HX^*((X,A\cup B)\indexcross{\cU'}(I,\partial I);M_3)\ar[r]^-{\kappa_*}_-{\cong}&\HX^*_\cup
}\]
commutes. At the left we have the connecting homomorphism in the long exact sequence of a triple and at the right we have the Mayer--Vietoris like connecting homomorphism coming from the right exact column of the diagram above. Both of them are isomorphism by what we have seen before. Moreover, we also know that the canonical map $\kappa_*$ at the bottom and the excision at the top are also isomorphisms.

With this diagram at hand, proving the claim is a simple calculation. Let $\varphi\in\CX^m(X,A;M_1)$ and $\psi\in\CX^n(X,B;M_2)$ be coarse cocycles. Then $\varphi*\psi$ is the image of 
\[(\varphi\cup (H^B_+)^*(s_b\psi),(-1)^{m} (H^A_-)^*(s_a\varphi)\cup\psi)\in \CX^{m+n-1}_+\oplus \CX^{m+n-1}_- \] under $(\iota_+)^*-(\iota_-)^*$ and it is mapped to 
\[(-1)^m\cdot(\varphi\cup (H^B_+)^*\psi,(H^A_-)^*\varphi\cup\psi)\in \CX^{m+n}_\cup\subset \CX^{m+n}_+\oplus \CX^{m+n}_- \]
under the coboundary map. But the latter is exactly the image of 
\[(-1)^m\cdot(\Gamma^{X,A,B})^*(\varphi\times\psi)\in\CX^{m+n}(X\indexcross{\cU'}I, X\times\{-1\}\cup A\times I+X\times\{1\}\cup B\times I;M_3)\]
under the canonical map $\kappa_*$, 
showing that $(-1)^m\cdot(\Gamma^{X,A,B})^*([\varphi]\times[\psi])=\cobound ([\varphi]*[\psi])$. By applying the left inverse $\sigma^*$ of $\cobound$ to this equation we obtain $[\varphi]\Cup[\psi]=[\varphi]*[\psi]$.
\end{proof}

Now we turn our attention to ordinary coarse homology and first recall its definition.

\begin{defn}
Let $X$ be a coarse space and $M$ an abelian group. For $m\in\N$ we define $\CX_m(X;M)$ as the group of all infinite formal sums $c=\sum_{\mathbf{x}\in X^{p+1}}m_{\mathbf{x}}\mathbf{x}$ such that
\begin{itemize}
\item the set $\supp(c)\coloneqq\{\mathbf{x}\in X^{m+1}\mid m_{\mathbf{x}}\not=0\}$ is a penumbra of the multidiagonal;
\item the set $\supp(c)\cap K$ is finite for all bounded $K\subset X^{m+1}$, or equivalently, for all precompact $K$. 
\end{itemize}
For $m\in\Z\setminus\N$ we define $\CX_m(X;M)\coloneqq 0$. These groups together with the boundary maps 
$\partial\colon\CX_{m}(X;M)\to\CX_{m-1}(X;M)$ which are defined on the single summands by
\[\partial(x_0,\dots,x_m)\coloneqq \sum_{i=0}^m(-1)^i(x_0,\dots,x_{i-1},x_{i+1},\dots,x_m)\]
constitute a chain complex. 
If $A\subset X$ is a coarse subspace, then  $\CX_*(A;M)$ is a subcomplex of $\CX_*(X;M)$ and we define the \emph{coarse chain complex} $\CX_*(X,A;M)\coloneqq \CX_*(X;M)/\CX_*(A;M)$. Its homology is the \emph{ordinary coarse homology} $\HX_*(X,A;M)$ of the pair of coarse spaces $(X,A)$ with coefficients in $M$.
\end{defn}

For the construction of slant products we need to make use of the excisiveness of \Cref{lem:crossproductexcisiveness} again, just as it was the case for cross products. For subspaces $A,B\subset X$ we define
\[\CX_*(X,A+B;M)\coloneqq \CX_*(X;M)/(\CX_*(A;M)+\CX_*(B;M))\,.\]
The proof of the following lemma is completely dual to the proof of \Cref{lem:cochaininclusion}.
\begin{lem} Assume that $(X;A,B)$ is an excisive triad of countably generated coarse spaces. Then the canonical quotient map
\[\pi\colon \CX_*(X,A+B;M) \twoheadrightarrow \CX_*(X,A\cup B;M)\]
is a chain homotopy equivalence.\qed
\end{lem}

\begin{defn}
Let $(X,A)$ and $(Y,B)$ be pairs of countably generated coarse spaces. Then the slant product
\begin{align*}
\HX_m((X,A)&\times(X,B);M_1)\otimes \HX^n(Y,B;M_2)\cong
\\&\cong \HX_m(X\times Y,A\times Y+X\times B;M_1)\otimes \HX^n(Y,B;M_2)
\\&\to \HX_{m-n}(X,A;M_3)
\end{align*}
is induced by the chain map
\[\CX_*(X\times Y,A\times Y+X\times B;M_1)\to \CX_*(X,A;M_3)\otimes(\CX^*(Y,B;M_2))^*\]
which is defined on the single summands $m_1\cdot (x_0,y_0,\dots,x_m,y_m)$ of a coarse chain in $\CX_m(X\times Y,A\times Y+X\times B;M_1)$ and coarse cochains $\varphi\in\CX^n(Y,B;M_2)$  by
\[(m_1\cdot (x_0,y_0,\dots,x_m,y_m))/\varphi\coloneqq (m_1\cdot \varphi(y_{m},\dots,y_{m-n})) (x_0,\dots,x_{m-n})\,,\]
yielding a coarse chain in $\CX_{m-n}(X,A;M_3)$.
\end{defn}

Again it is straighforward to check that we really have a chain map in the definition and obtain a well-defined slant product satisfying the axioms of \Cref{defn:slantproducts}. The computations rely on the observation that $(\bound c)/\varphi=\partial(c/\varphi)+(-1)^mc/(\cobound\varphi)$ if $m$ is the degree of the chain $c$.

This definition differs from the one found in most textbooks in so far as our cocycle swallows not the first $n+1$ entries but the last $n+1$ ones and this in reversed order. The change is necessary to fulfill the sign conventions of  \Cref{defn:slantproducts}. To be precise, our definition is the right one for slanting away cohomology classes from the right, whereas many textbooks prefer to slant them away from the left.

Dually to the Roe secondary cup product we can now define a secondary cap product as follows. For any infinte formal linear combination $c=\sum_{\mathbf{x}\in X^{m+1}}m_{\mathbf{x}}\mathbf{x}$ of $(m+1)$-tuples and any point $o\in X$ we define the infinite formal linear combination of $(m+2)$-tuples $r_oc\coloneqq \sum_{\mathbf{x}\in X^{m+1}}m_{\mathbf{x}}(o,\mathbf{x})$. Also, we define the cap product of any linear combination $c$ as above with all $m_{\mathbf{x}}\in M_1$ and any $\varphi\colon X^{n+1}\to M_2$ with $n\leq m$ as the linear combination 
\[c\cap\varphi\coloneqq \sum_{(x_0,\dots,x_m)\in X^{m+1}}(m_{(x_0,\dots,x_m)}\cdot\varphi(x_m,\dots,x_{m-n}))(x_0,\dots,x_{m-n})\]
with coefficients in $M_3$. Again, this cap product differs from the one usually found in textbooks. Note that the equality $(\bound c)\cap\varphi=\bound (c\cap\varphi) +(-1)^mc\cap(\cobound \varphi)$ holds for all $m$ and we have $\bound r_oc+r_o\bound c=c$ for $m\geq 1$ as well as $\bound r_oc+r_o\bound c=c-(o)$ for $m=0$.
In the following, $o=a$ will be a point in $A$ and hence the summand $(o)$ will be divided out when passing to the coarse chain complex relative to $A$.

\begin{defn}\label{defn:RoeSecCap}
Let $(X;A,B)$ be an excisive triad of coarsely connected countably generated coarse spaces with $A,B$ nonempty and choose base-points $a\in A$, $b\in B$.
Then the \emph{Roe secondary cap product} of a coarse homology class $[c]\in\HX_m(X,A\cup B;M_1)$ represented by a coarse cycle $c\in\CX_m(X,A+B;M_1)$ and a coarse cohomology class $[\varphi]\in\HX^n(X,B;M_2)$ is the class
$[c]\divideontimes[\varphi]\in\HX_{m-n+1}(X,A\cup B;M_3)$ represented by the coarse cycle
\begin{equation*}
c\divideontimes\varphi\coloneqq - r_a(c\cap\varphi) -(-1)^m c\cap(s_b\varphi)\in \CX_{m-n+1}(X,A;M_3)\,.
\end{equation*}
\end{defn}
It is again straightforward to check that $c\divideontimes\varphi$ is indeed a well-defined coarse cycle in the claimed chain group and that the resulting homology class is independent of the choice of the base-points $a,b$.

\begin{thm}\label{thm:seccapcomparison}
Let $(X;A,B)$ be an excisive deformation triad of coarsely connected countably generated coarse spaces.
Then the associated secondary products $\divideontimes$ and $\Cap$ agree.
\end{thm}

\begin{proof}
Unfortunately, only part of the proof is truely dual to the proof of \Cref{lem:secprodcomparison} and some work remains to be done. The complexes and diagrams can be dualized without problems and we obtain a short exact sequence of chain complexes
\[0\to\CX_*(X,A+B;M_1)\xrightarrow{((\iota_+)_*,-(\iota_-)_*)}\CX^+_*\oplus\CX^-_*\to\CX^\cup_*\to 0\]
and a canonical map $\CX^\cup_*\to\CX_*((X,A\cup B)\indexcross{\cU'}(I,\partial I);M_1)$ such that the diagram in homology
\[\xymatrix{
\HX^\cup_{*+1}\ar[r]\ar[d]_{\bound}^{\cong}
&\HX_{*+1}((X,A\cup B)\indexcross{\cU'}(I,\partial I);M_1)\ar[d]_-{\bound}^-{\cong}
\\\HX_{*}(X,A+B;M_1)\ar[r]^-{\cong}
&{\Estack{\HX_{*}(&X\times\partial I\cup(A\cup B)\times I,
\\&X\times\{1\}\cup (A\cup B)\times I)}}
}\]
commutes.

The diagram allows us to compute the suspension homomorphism, but at this time the calculations become more complicated than in the case of the secondary cup products. To this end, we start with a coarse cycle $c\in \CX_m(X,A+B;M_1)$. Now there are chain contractions $P^\pm\colon \CX_*(X,A+B;M_1)\to \CX^\pm_{*+1}$ for $(\iota_\pm)^*$, i.\,e.\ $\bound P^\pm+P^\pm\bound=(\iota_\pm)^*$. For example, we can take the prism operators constructed in the proof of \cite[Lemma 4.4]{wulff2020equivariant} with $H=\id$. Then, using the diagram above we see that $\sigma_*[c]$ is represented by the coarse cycle 
\[P^+c-P^-c\in \CX_{m+1}(X\indexcross{\cU'}I,X\times\{-1\}\cup A\times I+ X\times\{1\}\cup B\times I;M_1)\,.\] 
It follows that
\begin{align*}
[c]\Cap [\varphi]&=((\Gamma^{X,A,B})_*[P^+c-P^-c])/[\varphi]
\\&=[((\mathrm{pr}_X\times H_B)_*(P^+c))/\varphi-((H_A\times\mathrm{pr}_X)_*P^-c)/\varphi]\,.
\end{align*}
Just as in \Cref{defn:RoeSecCap} it is straightforward to verify that 
\[d\coloneqq (-1)^m ((\mathrm{pr}_X\times H_B)_*(P^+c))/(s_b\varphi)- r_a(((H_A\times\mathrm{pr}_X)_*(P^-c))/\varphi) \]
is a coarse cochain in $\CX_{m-n}(X,A;M_3)$ and then the calculation
\begin{align*}
\bound d
&= (-1)^m ((\mathrm{pr}_X\times H_B)_*(\bound P^+c))/(s_b\varphi) 
+ ((\mathrm{pr}_X\times H_B)_*(P^+c))/(\cobound s_b\varphi)
\\&\qquad +r_a\bound (((H_A\times\mathrm{pr}_X)_*(P^-c))/\varphi)
+((H_A\times\mathrm{pr}_X)_*(P^-c))/\varphi
\\&= (-1)^m ((\mathrm{pr}_X\times H_B)_*(\iota_+)_*c)/(s_b\varphi)
+ ((\mathrm{pr}_X\times H_B)_*(P^+c))/\varphi
\\&\qquad +r_a(((H_A\times\mathrm{pr}_X)_*(\bound P^-c))/\varphi)
+((H_A\times\mathrm{pr}_X)_*(P^-c))/\varphi
\\&= (-1)^m c\cap (s_b\varphi)
+ ((\mathrm{pr}_X\times H_B)_*(P^+c))/\varphi
\\&\qquad +r_a(((H_A\times\mathrm{pr}_X)_*(\iota_-)_*c)/\varphi)
+((H_A\times\mathrm{pr}_X)_*(P^-c))/\varphi
\\&= (-1)^m c\cap (s_b\varphi)
+ ((\mathrm{pr}_X\times H_B)_*(P^+c))/\varphi
\\&\qquad +r_a(c\cap\varphi)
+((H_A\times\mathrm{pr}_X)_*(P^-c))/\varphi
\end{align*}
finishes the proof.
\end{proof}

\section{Transgression maps}
\label{sec:transgression}

A big class of coarse \mbox{(co-)}homology theories is obtained by coarsifying topological \mbox{(co-)}homology theories for $\sigma$-locally compact spaces. 
We refer to \cite[Section 5]{wulff2020equivariant} for a very detailed introduction into this topic, which even treats the more general equivariant case. It thereby widely generalizes the construction for the absolute groups in \cite[Section 4]{EngelWulff}.
Here we shall only recall the rough outlines of the construction, based on the general perception that $\sigma$-locally compact spaces are countable direct limits of locally compact Hausdorff spaces in some sense. 
The category of pairs of $\sigma$-locally compact spaces, on which we will consider \mbox{(co-)}homology theories in the sense of Eilenberg and Steenrod \cite{EilenbergSteenrod}, has the pairs $(\cX,\cA)$ of $\sigma$-locally compact spaces with $\cA\subset\cX$ closed as objects and its morphisms are the proper and continuous $\sigma$-maps.
The purpose of this section is to show that the associated transgression maps relate the secondary products on the coarse space to the primary products on its corona.

In a nutshell, the coarsification procedure is as follows:
Let $(X,A)$ be a pair of countably generated coarse spaces of bornologically bounded geometry. The latter condition means that there exist discretizations $(X',A')\subset (X,A)$, that is, $X',A'$ are locally finite (i.\,e.\ bounded subsets of $X',A'$ are finite) and the inclusions $X'\subset X$ and $A'\subset A$ are coarse equivalences (cf.\ \cite[Definition 2.4]{wulff2020equivariant}).
The Rips complex construction applied to such a discretization yields a countable directed system of pairs of locally compact Hausdorff spaces $(P_n(X'),P_n(A'))$ which constitutes a pair of $\sigma$-locally compact spaces $(\cP(X'),\cP(A'))$. If $\Homol_*$ is a homology theory or $\Cohom^*$ is a cohomology theory for $\sigma$-locally compact spaces and proper continuous maps which satisfy the homotopy, exactness and excision axioms, then the associated coarse homology and cohomology groups are defined by
\[\HomolX_*(X,A)\coloneqq\Homol_*(\cP(X'),\cP(A'))\quad\text{and}\quad \CohomX^*(X,A)\coloneqq\Cohom^*(\cP(X'),\cP(A'))\,,\]
respectively. They are independent of the choice of discretizations up to canonical isomorphism and together with certain induced maps they constitute a coarse \mbox{(co-)}homology theory satisfying the strong homotopy axiom (see \cite[Theorem 5.15]{wulff2020equivariant}).

Now, in order to introduce the products on the topological side we first have to generalize the notion of topological excisiveness.
We call $(\cX;\cA,\cB)$ with $\cA,\cB\subset\cX$ a (topologically) excisive triad of $\sigma$-locally compact spaces if the interiors of $\cA,\cB$ in $\cA\cup\cB$ cover all of $\cA\cup\cB$. 
In this case, the two inclusions $(\cA,\cA\cap\cB)\subset(\cA\cup\cB,\cB)$ and $(\cB,\cA\cap\cB)\subset(\cA\cup\cB,\cA)$ are topological excisions and hence induce isomorphisms under $\Homol_*$ and $\Cohom^*$. Generalizing this notion, we call $(\cX;\cA,\cB)$  an $\Homol_*$-excisive triad or an $\Cohom^*$-excisive triad in $\cX$ if the two inclusions of pairs induce isomorphisms under $\Homol_*$ or $\Cohom^*$, respectively (cf.\ \cite[Proposition and Definition 8.1]{DoldTopology}). 
Important examples for this notion come from coarse excisiveness, as the following reformulation of the excisiveness part of \cite[Theorem 5.15]{wulff2020equivariant} illustrates.

\begin{lem}\label{lem:Eexcisive}
If $(X';A',B')$ is a coarsely excisive triad of countably generated locally finite coarse spaces, then $(\cP(X');\cP(A'),\cP(B'))$ is an $\Homol_*$- and $\Cohom^*$-excisive triad of $\sigma$-locally compact spaces for all homology and cohomology theories $\Homol_*$, $\Cohom^*$. \qed
\end{lem}

Furthermore, we say that two pairs of $\sigma$-locally compact spaces $(\cX,\cA)$ and $(\cY,\cB)$ are $\Homol_*$-productable or $\Cohom^*$-productable if 
$(\cA\times\cY\cup\cX\times\cB;\cA\times\cY,\cX\times\cB)$ is an $\Homol_*$- or an $\Cohom^*$-excisive triad, respectively, and hence the two inclusions  $i\colon (\cA\times \cY,\cA\times \cB)\to (\cA\times \cY\cup \cX\times \cB,\cX\times \cB)$ and $j\colon (\cX\times \cB,\cA\times \cB)\to (\cA\times \cY\cup \cX\times \cB,\cA\times \cY)$ induce isomorphisms.
In contrast to the coarse excisiveness from \Cref{lem:crossproductexcisiveness}, this is not always true.

\begin{defn}\label{defn:topologicalcrossslant}
A cross product $\times\colon\Cohom[1]^*\otimes\Cohom[2]^*\to\Cohom[3]^*$ between cohomology theories for $\sigma$-locally compact spaces consists of a family of natural homomorphisms
\[\times\colon\Cohom[1]^m(\cX,\cA)\otimes\Cohom[2]^n(\cY,\cB)\to\Cohom[3]^{m+n}((\cX,\cA)\times(\cY,\cB))\]
for all $\Cohom[3]^*$-productable pairs of $\sigma$-locally compact spaces $(\cX,\cA)$ and $(\cY,\cB)$ which satisfy the obvious topological analogues of the diagrams in \Cref{defn:crossproducts}.

Similarily, a slant product $/\colon\Homol[1]_*\otimes\Cohom[2]^*\to\Homol[3]_*$ between homology and cohomology theories for $\sigma$-locally compact spaces consists of a family of natural transformations
\[/\colon\Homol[1]_m((\cX,\cA)\times(\cY,\cB))\otimes\Cohom[2]^n(\cY,\cB)\to\Homol[3]_{m-n}(\cX,\cA)\]
for all $\Homol[1]_*$-productable pairs of $\sigma$-locally compact spaces $(\cX,\cA)$ and $(\cY,\cB)$ which satisfy the obvious topological analogues of the diagrams in \Cref{defn:slantproducts}.
\end{defn}

We now use the ideas of \cite[Section 4.6]{EngelWulffZeidler} to turn them into coarse cross and slant products.
To this end we need to pass from the product of rips complexes to the Rips complex of products. 
\begin{lem}\label{lem:Ripsproducthomotopyequivalences}
Let $(X',A')$ and $(Y',B')$ be pairs of countably generated locally finite coarse spaces. Then the canonical inclusion maps 
\begin{align*}
\cP(X')\times\cP(Y')&\subset\cP(X'\times Y')
&\cP(A')\times\cP(Y')&\subset\cP(A'\times Y')
\\\cP(X')\times\cP(B')&\subset\cP(X'\times B')
&\cP(A')\times\cP(B')&\subset\cP(A'\times B')
\end{align*}
\begin{align*}
\cP(A')\times\cP(Y')\cup\cP(X')\times\cP(B')&\subset\cP(A'\times Y')\cup\cP(X'\times B')
\\\cP(A'\times Y')\cup\cP(X'\times B')&\subset \cP(A'\times Y'\cup X'\times B')
\end{align*}
are all homotopy equivalences. 
\end{lem}
\begin{proof}
All but the last one are a direct consequence of the simple construction in \cite[Lemma 4.43]{EngelWulffZeidler}.

For the last one we use the barycentric subdivision from the proof of \Cref{lem:cochaininclusion}:
The simplex of $\cP(A'\times Y'\cup X'\times B')$ spanned by an $(m+1)$-tuple $\mathbf{x}\in (A'\times Y'\cup X'\times B')^{m+1}$ can be mapped piecewise linearily to the union of simplices in $\cP(A'\times Y')\cup \cP(X'\times B')$ spanned by the $(m+1)$-tuples $\mathbf{x}_\sigma\in (A'\times Y')^{m+1}\cup (X'\times B')^{m+1}$. This yields a proper continuous $\sigma$-map which is a homotopy inverse to the inclusion, where the homotopies are simply linear interpolation.
\end{proof}

\begin{cor}\label{cor:Ripsexcisivecouple}
Let $(X',A')$ and $(Y',B')$ be pairs of countably generated locally finite coarse spaces. Then $(\cP(X'),\cP(A'))$ and $(\cP(Y'),\cP(B'))$ are $\Homol_*$- and $\Cohom^*$-productable for all homology and cohomology theories $\Homol_*$, $\Cohom^*$.
\end{cor}
\begin{proof}
 \Cref{lem:crossproductexcisiveness,lem:Eexcisive} together show that 
\[(\cP(A'\times Y'\cup X'\times B'))\,;\quad\cP(A'\times Y')\,,\quad\cP(X'\times B'))\]
is an $\Homol_*$- and an $\Cohom^*$-excisive triad. 
The isomorphisms which are necessary to prove that 
\[(\cP(A')\times\cP(Y')\cup\cP(X')\times\cP(B')\,;\quad\cP(A')\times\cP(Y')\,,\quad\cP(X')\times\cP(B'))\]
is an $\Homol_*$- and an $\Cohom^*$-excisive triad can now be tracked down using the isomorphisms induced by the homotopy equivalences in the lemma.
\end{proof}

\begin{defn}
For all pairs $(X,A)$ and $(Y,B)$ of countably generated coarse spaces of bornologically bounded geometry we choose discretizations $(X',A')\subset(X,A)$ and $(Y',B')\subset (Y,B)$. \Cref{cor:Ripsexcisivecouple} then ensures that the topological cross and slant products from \Cref{defn:topologicalcrossslant} exist for the pairs $(\cP(X'),\cP(A'))$ and $(\cP(Y'),\cP(B'))$.  \Cref{lem:Ripsproducthomotopyequivalences} allows us to identify them with the homomorphisms appearing in \Cref{defn:crossproducts,defn:slantproducts}. The family of all of them constitute a coarse cross product $\times\colon\CohomX[1]^*\otimes\CohomX[2]^*\to\CohomX[3]^*$ or a coarse slant product $/\colon\HomolX[1]_*\otimes\CohomX[2]^*\to\HomolX[3]_*$, respectively, and it is called the \emph{coarsification} of the topological cross or slant product.
\end{defn}

The purpose of this section is to analyze the resulting secondary cup and cap products of these coarsified cross and slant products. The first step is to reformulate them in a more topological way. As before we write $I\coloneqq [-1,1]$. Note that we can define the following topological analogue to \Cref{def:secondarycupcap}.
\begin{defn}
\label{defn:topGamma}
Let $\cX$ be a $\sigma$-locally compact space, $\cA,\cB\subset \cX$ closed subspaces and 
$\widetilde H^\cA\colon(\cX,\cA)\times([0,1],\{1\})\to (\cX,\cA)$ and $\widetilde H^\cB\colon(\cX,\cB)\times([0,1],\{1\})\to (\cX,\cB)$ continuous homotopies with $\widetilde H^\cA|_{\cX\times\{0\}}=\id=\widetilde H^\cB|_{\cX\times\{0\}}$.
Then 
\[\widetilde \Gamma^{\cX,\cA,\cB}\coloneqq(\widetilde H^\cA_-,\widetilde H^\cB_+)\colon (\cX,\cA\cup\cB)\times(I,\partial I)\to (\cX,\cA)\times(\cX,\cB)\]
is a proper continuous $\sigma$-map. Given a topological cross product $\times\colon\Cohom[1]^*\otimes\Cohom[2]^*\to\Cohom[3]^*$
we define the associated topological secondary cup product $\Cup$ as $(-1)^m$ times the composition
\begin{align*}
\Cohom[1]^m(\cX,\cA)\otimes\Cohom[2]^n(\cX,\cB)&\xrightarrow{\times}\Cohom[3]^{m+n}((\cX,\cA)\times (\cX,\cB))
\\&\xrightarrow{(\widetilde\Gamma^{\cX,\cA,\cB})^*}\CohomX[3]^{m+n}((\cX,\cA\cup \cB)\times(I,\partial I))
\\&\xrightarrow{\cosusp[3]}\Cohom[3]^{m+n-1}(\cX,\cA\cup \cB)\,.
\end{align*}
and given a topological slant product $/\colon\Homol[1]_*\otimes\Cohom[2]^*\to\Homol[3]_*$ we define the associated secondary cap product $\Cap$ as $(-1)^{m+1}$ times the composition
\begin{align*}
\Homol[1]_m(\cX,\cA\cup \cB)&\otimes\Cohom[2]^n(\cX,\cB)\to
\\&\xrightarrow{\susp[1]\otimes\id}\Homol[1]_{m+1}((\cX,\cA\cup \cB)\times(I,\partial I))\otimes\Cohom[2]^n(\cX,\cB)
\\&\xrightarrow{(\widetilde\Gamma^{\cX,\cA,\cB})_*\otimes\id}\Homol[1]_{m+1}((\cX,\cA)\times (\cX,\cB))\otimes\Cohom[2]^n(\cX,\cB)
\\&\xrightarrow{/}\Homol[3]_{m-n+1}(\cX,\cA)\,.
\end{align*}
\end{defn}
The topological suspension isomorphisms in this definition are defined in the obvious way and they exist, because the inclusion 
\[(\cX,\cA\cup\cB)\times\{0\}\subset(\cX\times\partial I\cup (\cA\cup\cB)\times I, \cX\times\{1\}\cup (\cA\cup\cB)\times I)\]
is homotopy equivalent to a topological excision.
The topological products fulfil the analogue properties to those of the coarse products proven in \Cref{sec:secondaryfromdeformations}, but their proofs are a lot simpler, because we do not need to cope with the warped products.

\begin{lem}\label{lem:coarsetopologicalsecondarycomparison}
Let $(X;A,B)$ be a deformation triad of countably generated coarse space of bornologically bounded geometry and $(X';A',B')\subset (X;A,B)$ a discretization. Then the coarse secondary products for $X,A,B$ canonically identify with topologically secondary products for $(\cP(X');\cP(A'),\cP(B'))$.
\end{lem}

\begin{proof}
Let $H^A\colon (X,A)\indexcross{\cU_A}([0,1],\{1\})\to (X,A)$ be an associated generalized controlled homotopy with $H^A_0=\id_X$ and we may assume that it restricts to a controlled homotopy $(X',A')\indexcross{\cU}([0,1],\{1\})\to (X',A')$. Then the construction of \cite[Lemma 5.13]{wulff2020equivariant} provides us with a discretization $(Z_A,D_A)$ of $(X,A)\indexcross{\cU_A}([0,1],\{1\})$ which  contains $(X',A')\cup \{0,1\}$ and is contained in $(X',A')\cup [0,1]$ and with a proper continuous map 
\[(\cP(X'),\cP(A'))\times ([0,1],\{1\})\to (\cP(Z_A),\cP(D_A))\]
which extends the canonical inclusion of $(\cP(X'),\cP(A'))\times \{0,1\}$ into $(\cP(Z_A),\cP(D_A))$.
Note that its composition with  $\cP(H^A)\colon (\cP(Z_A),\cP(D_A))\to (\cP(X'),\cP(A'))$ is a continuous homotopy
\[\widetilde H^A\colon (\cP(X'),\cP(A'))\times ([0,1],\{1\})\to (\cP(X'),\cP(A'))\]
with $\widetilde H|_{\cP(X')\times\{0\}}=\id$.

Similarily, we pick a controlled homotopy $H^B$ associated to $(X,B)$ and turn it into a continuous homotopy $\widetilde H^B$ for $(\cP(X'),\cP(B'))$.
Combining these two constructions we obtain a discretization $(Z,D)$ of $(X,A\cup B)\indexcross{\cU}(I,\partial I)$ and a proper continuous map $i\colon (\cP(X'),\cP(A')\cup \cP(B'))\times (I,\partial I)\to (\cP(Z),\cP(D))$ extending the canonical inclusion of $(\cP(X'),\cP(A'))\times \partial I$ into $(\cP(Z),\cP(D))$ such that we have
\begin{align}
\widetilde \Gamma^{X,A,B}&\coloneqq (\widetilde H^A_-,\widetilde H^B_+)=\cP(\Gamma^{X,A,B})\circ i\colon \nonumber
\\&(\cP(X'),\cP(A')\cup \cP(B'))\times (I,\partial I)\to (\cP(X'),\cP(A'))\times (\cP(X'),\cP(B'))\,. 
\nonumber\end{align}
The identification of the coarse secondary cup product with the topological one is now given by the diagram
\[\xymatrix{
\CohomX[1]^m(X,A)\otimes\CohomX[2]^n(X,B)\ar[d]\ar@{=}[r]
&\Cohom[1]^m(\cP(X'),\cP(A'))\otimes\Cohom[2]^n(\cP(X'),\cP(B'))\ar[d]
\\\CohomX[3]^{m+n}((X,A)\times(X,B))\ar[d]^{(\Gamma^{X,A,B})^*}\ar[r]^-{\cong}
&\Cohom[3]^{m+n}((\cP(X'),\cP(A'))\times(\cP(X'),\cP(B')))\ar[d]^{(\widetilde\Gamma^{X,A,B})^*}
\\\CohomX[3]^{m+n}((X,A\cup B)\indexcross{\cU}(I,\partial I))\ar[r]^-{i^*}\ar[d]^{\cosusp[3]}
&\Cohom[3]^{m+n}((\cP(X'),\cP(A')\cup\cP(B'))\times(I,\partial I))\ar[d]^{\cosusp[3]}
\\\CohomX[3]^{m+n-1}(X,A\cup B)\ar[r]^-{\cong}
&\Cohom[3]^{m+n-1}(\cP(X'),\cP(A')\cup\cP(B'))
}\]
where the second and fourth horizontal arrows are isomorphisms induced by homotopy equivalences from \Cref{lem:Ripsproducthomotopyequivalences}. Commutativity of the upper square is the definition of the coarsification of the cross product. The middle square commutes by definition of $\widetilde\Gamma^{X,A,B}$ and the lower square commutes by naturality of the connecting homomorphism under $i$.

The proof for the identification of the secondary cap products is completely dual.
\end{proof}

We now turn our attention to transgression maps (cf.\ \cite[Section 5.3]{wulff2020equivariant}). They only exist for the so-called \emph{single-space \mbox{(co-)}homology theories} (cf.\ \cite[Definition 5.16]{wulff2020equivariant}, that is, \mbox{(co-)}homology theories $\Homol_*,\Cohom^*$ for $\sigma$-locally compact spaces which satisfy the \emph{strong excision axiom}: For all pairs of $\sigma$-locally compact spaces $(\cX,\cA)$ there are natural isomorphisms $\Homol_*(\cX,\cA)\cong\Homol_*(\cX\setminus\cA)$ and $\Cohom^*(\cX,\cA)\cong\Cohom^*(\cX\setminus\cA)$ and hence these theories can be expressed completely by their absolute groups. The most important examples are $\K$-theory, $\K$-homology and Alexander--Spanier cohomology.

Note also that we do not have to worry about excisiveness conditions for single space \mbox{(co-)}homology theories $\Homol_*,\Cohom^*$, because for every two closed subspaces $\cA,\cB\subset \cX$ the triad $(\cX;\cA,\cB)$ is $\Homol_*$- and $\Cohom^*$-excisive.

Let us briefly recall the definition of the transgression maps. Let $(X,A)$ be a pair of countably generated coarse spaces of bornologically bounded geometry and $(X',A')\subset (X,A)$ a discretization. 
Canonically associated to $X$ there is a compact Hausdorff space $\partial_hX$, the so-called \emph{Higson corona}. 
By a \emph{Higson dominated corona} we then mean a compact Hausdorff space $\partial X$ together with a natural surjection $\partial_hX\twoheadrightarrow\partial X$. It has the property that we can glue it to the Rips complex $\cP(X')$  to obtain a $\sigma$-compactification $\overline{\cP(X')}$ of $\cP(X')$ with corona $\partial X=\overline{\cP(X')}\setminus\cP(X')$. The subspace $A$ then also has an associated Higson-dominated corona $\partial A\subset\partial X$, namely $\partial A=\partial X\cap\overline{\cP(A')}$. With all this data the transgression maps are the connecting homomorphisms for the pair $(\overline{\cP(X')}\setminus\overline{\cP(A')},\partial X\setminus\partial A)$, that is,
\begin{align*}
T_*^{X,A}\colon\HomolX_{*+1}(X,A)&\cong\Homol_{*+1}(\cP(X')\setminus\cP(A'))\to \Homol_*(\partial X\setminus\partial A)\cong \Homol_*(\partial X,\partial A)
\end{align*}
and dually $T^*_{X,A}\colon\Cohom^*(\partial X,\partial A)\to\CohomX^{*+1}(X,A)$.
The main result of this section is the following.

\begin{thm}\label{thm:transgressioncompatiblewithcupandcap}
Given a topological cross product $\times\colon\Cohom[1]^*\otimes\Cohom[2]^*\to\Cohom[3]^*$ or a topological slant product $/\colon\Homol[1]_*\otimes\Cohom[2]^*\to\Homol[3]_*$ between single space \mbox{(co-)}homology theories for $\sigma$-locally compact spaces, the induced secondary cup and cap products on coarse spaces (if existent) and the induced primary cup and cap products on the coronas are related by the transgression maps:
\begin{align*}
\forall x\in\Cohom[1]^*(\partial X,\partial A),&y\in\Cohom[2]^*(\partial X,\partial B)\colon &T^*_{X,A\cup B}(x\cup y)&=T^*_{X,A}(x)\Cup T^*_{X,B}(y)
\\\forall x\in\HomolX[1]_*(X,A\cup B),&y\in\Cohom[2]^*(\partial X,\partial B)\colon&T_*^{X,A}(x\Cap T^*_{X,B}(y))&=T_*^{X,A\cup B}(x)\cap y
\end{align*}
\end{thm}

\begin{example}
\label{ex:ConesTransgression}
Every compact metric space $K$ is a Higson dominated corona of the open cone $\cO K$ over itself and if moreover $L\subset K$ is a closed subset then $L$ is the corresponding Higson dominated corona of the subspace $\cO L\subset \cO K$.
Combining  \cite[Section 3.1, Theorem 5.7, Lemma 4.17]{EngelWulff} and the five-lemma to pass from the reduced to the relative case shows that the transgression maps $T_*^{\cO K,\cO L}$, $T^*_{\cO K,\cO L}$ are isomorphisms.

Now, if two closed subsets $L,M\subset K$ are given then $(\cO K,\cO L,\cO M)$ is a deformation triad by \Cref{ex:conedeformations}.
Therefore, \Cref{thm:transgressioncompatiblewithcupandcap} shows that the secondary cup and cap products of $(\cO K,\cO L,\cO M)$ can be identified canonically with the primary cup and cap products of $(K,L,M)$. This shows that our theory of coarse secondary products is at least as rich as the theory of the primary cup and cap products on compact metrizable spaces.
\end{example}

Instead of proving \Cref{thm:transgressioncompatiblewithcupandcap} directly we note that by \Cref{lem:coarsetopologicalsecondarycomparison} it is nothing but the special case $\cX=\cP(X')$, $\cA=\cP(A')$, $\cB=\cP(B')$, $\partial\cX=\partial X$, $\partial\cA=\partial A$, $\partial\cB=\partial B$ of the following theorem.

\begin{thm}\label{thm:boundariescompatiblewithcupandcap}
Let $(\cX;\cA,\cB)$ be a topological deformation triad of $\sigma$-locally compact spaces, that is, the assumptions of \Cref{defn:topGamma} hold.
Furthermore, let $\overline\cX$ be a $\sigma$-locally compact space which contains $\cX$ as an open subset and let $\overline\cA$ and $\overline\cB$ be closed subspaces of $\overline\cX$ such that $\cA=\overline\cA\cap\cX$ and $\cB=\overline\cB\cap\cX$. We define $\partial\cX=\overline\cX\setminus\cX$, $\partial\cA=\overline\cA\setminus\cA$ and  $\partial\cB=\overline\cB\setminus\cB$.
\begin{enumerate}
\item If $\times\colon\Cohom[1]^*\otimes\Cohom[2]^*\to\Cohom[3]^*$ is a cross product between single-space cohomology theories for $\sigma$-locally compact spaces, then the connecting homomorphisms 
\begin{align*}
\cobound[1]\colon\Cohom[1]^*(\partial\cX\setminus\partial\cA)&\to\Cohom[1]^{*+1}(\cX\setminus\cA)
\\\cobound[2]\colon\Cohom[2]^*(\partial\cX\setminus\partial\cB)&\to\Cohom[2]^{*+1}(\cX\setminus\cB)
\\\cobound[3]\colon\Cohom[3]^*(\partial\cX\setminus(\partial\cA\cup\partial\cB))&\to\Cohom[3]^{*+1}(\cX\setminus(\cA\cup\cB))
\end{align*}
are compatible with the primary and secondary cup product in the sense that $\cobound[3](x\cup y)=\cobound[1](x)\Cup\cobound[2](y)$ for all $x\in \Cohom[1]^*(\partial\cX\setminus\partial\cA)$ and $y\in \Cohom[2]^*(\partial\cX\setminus\partial\cB)$.
\item If $/\colon\Homol[1]_*\otimes\Cohom[2]^*\to\Homol[3]_*$ is a slant product between single-space cohomology theories for $\sigma$-locally compact spaces, then the connecting homomorphisms 
\begin{align*}
\bound[1]\colon\Homol[1]_{*+1}(\cX\setminus(\cA\cup\cB))&\to\Homol[1]_*(\partial\cX\setminus(\partial\cA\cup\partial\cB))
\\\cobound[2]\colon\Cohom[2]^*(\partial\cX\setminus\partial\cB)&\to\Cohom[2]^{*+1}(\cX\setminus\cB)
\\\bound[3]\colon\Homol[3]_{*+1}(\cX\setminus\cA)&\to\Homol[3]_*(\partial\cX\setminus\partial\cA)
\end{align*}
are compatible with the primary and secondary cap product in the sense that $\bound[3](x\Cap\cobound[2](y))=\bound[1](x)\cap y$ for all $x\in \Homol[1]_{*+1}(\cX\setminus(\cA\cup\cB))$ and $y\in \Cohom[2]^*(\partial\cX\setminus\partial\cB)$.
\end{enumerate}
\end{thm}
The notation already suggests that we have the special case in mind in which $\overline{\cX}$ is a $\sigma$-compactification of $\cX$ and $\overline{\cA}$, $\overline{\cB}$ are the closures of $\cA$,$\cB$, respectively, in $\overline{\cX}$, but we are not restricted to this case.

\begin{proof}
In the following, we we let $\mathring J$ denote the interior of $J$ for any interval $J$.
Consider the diagrams of function algebras with exact rows in Figure~\ref{fig:Proofofmultiplicativitylemma}.
\begin{sidewaysfigure}
\resizebox{\textwidth}{!}{$\xymatrix@C=10ex@R=3ex{
 0\ar[r]&\scriptstyle \Cz(\cX\setminus\cA)\otimes\Cz(\cX\setminus\cB)\ar[r]\ar[d]^{(\widetilde\Gamma^{\cX,\cA,\cB})^*}
\ar[r]&\scriptstyle  \Cz(\cX\setminus\cA)\otimes\Cz(\overline\cX\setminus\overline\cB)\ar[r]\ar[d]^{(\widetilde\Gamma^{\cX,\cA,\cB})^*}
\ar[r]&\scriptstyle  \Cz(\cX\setminus\cA)\otimes\frac{\Cz(\overline\cX\setminus\overline\cB)}{\Cz(\cX\setminus\cB)}\ar[d]^{(\widetilde\Gamma^{\cX,\cA,\cB})^*}
\ar[r]& 0
\\ 0\ar[r]&\scriptstyle  \Cz((\cX\setminus(\cA\cup\cB))\times\mathring I)
\ar[r]&\scriptstyle  \Cb(\cX\times I,\cX\times\partial I\cup\cA\times I)
\ar[r]&\frac{\Cb(\cX\times I,\cX\times\partial I\cup\cA\times I)}{\Cz((\cX\setminus(\cA\cup\cB))\times\mathring I)}
\ar[r]&0
\\ 0\ar[r]&\scriptstyle  \Cz((\cX\setminus(\cA\cup\cB))\times\mathring I)\ar[u]_{\gamma^*\simeq\id}
\ar[r]&\scriptstyle  \Cz((\cX\setminus(\cA\cup\cB))\times [-1,1))\ar[u]_{\gamma^*}
\ar[r]_{\ev{-1}}&\scriptstyle \Cz(\cX\setminus(\cA\cup\cB))\ar@{-->}[u]\ar[r]&0
\\ 0\ar[r]&\scriptstyle \Cz(\cX\setminus\cA)\otimes\frac{\Cz(\overline\cX\setminus\overline\cB)}{\Cz(\cX\setminus\cB)}\ar[d]^{(\widetilde\Gamma^{\cX,\cA,\cB})^*}
\ar[r]&\scriptstyle \Cz(\overline\cX\setminus\overline\cA)\otimes\frac{\Cz(\overline\cX\setminus\overline\cB)}{\Cz(\cX\setminus\cB)}\ar[d]^{\beta^*}
\ar[r]&\scriptstyle \frac{\Cz(\overline\cX\setminus\overline\cA)}{\Cz(\cX\setminus\cA)}\otimes\frac{\Cz(\overline\cX\setminus\overline\cB)}{\Cz(\cX\setminus\cB)}\ar[d]_{\Delta^*}\ar@/^4pc/[dd]|\hole^(.3){\Delta^*}
\ar[r]& 0
\\ 0\ar[r]&\frac{\Cb(\cX\times I,\cX\times\partial I\cup\cA\times I)}{\Cz((\cX\setminus(\cA\cup\cB))\times\mathring I)}
\ar[r]&\frac{\Cb(\cX\times I,\cX\times\{-1\}\cup \cB\times\{1\}\cup\cA\times I)}{\Cz((\cX\setminus(\cA\cup\cB))\times(-1,1])}
\ar[r]_-{\ev{1}}&\frac{\Cb(\cX,\cA\cup\cB)}{\Cz(\cX\setminus(\cA\cup\cB))}
\ar[r]& 0
\\ 0\ar[r]&\scriptstyle \Cz(\cX\setminus(\cA\cup\cB))\ar@{-->}[u]\ar[r]&\scriptstyle \Cz(\overline\cX\setminus(\overline\cA\cup\overline\cB))\ar[u]_{(\widetilde H^{\cA}_-)^*}
\ar[r]
&\frac{\Cz(\overline\cX\setminus(\overline\cA\cup\overline\cB))}{\Cz(\cX\setminus(\cA\cup\cB))}\ar[u]^{\textnormal{incl.}}
\ar[r]&0
\save "1,1"."1,4"."3,5"."3,3"*+<0ex,1ex>[F.]\frm{}\restore
\save "4,1"."4,2"."6,4"."6,5"*[F.]\frm{}\restore
}$}

\bigskip

\resizebox{\textwidth}{!}{$\xymatrix@C=2ex@R=3ex{
(\cX\setminus\cA)\times(\cX\setminus\cB)
&\subset
&(\cX\setminus\cA)\times(\overline\cX\setminus\overline\cB)
&\supset
&(\cX\setminus\cA)\times(\partial\cX\setminus\partial\cB)
&\subset
&(\overline\cX\setminus\overline\cA)\times(\partial\cX\setminus\partial\cB)
&\supset
&(\partial\cX\setminus\partial\cA)\times(\partial\cX\setminus\partial\cB)
\\(\cX\setminus(\cA\cup\cB))\times\mathring I\ar[u]^{\widetilde\Gamma^{\cX,\cA,\cB}}\ar[d]_{\gamma}^{\simeq\,\id}
&\subset
&\cY_1\ar[d]\ar[u]
&\supset
&\cY_2\ar[d]\ar[u]
&\subset
&\cY_3\ar[d]\ar[u]
\supset
&\cY_4\ar[d]\ar[u]
\\(\cX\setminus(\cA\cup\cB))\times\mathring I
&\subset
&(\cX\setminus(\cA\cup\cB))\times[-1,1)
&\supset 
&\cX\setminus(\cA\cup\cB)
&\subset
&\overline{\cX}\setminus(\overline{\cA}\cup\overline{\cB})
&\supset
&\partial\cX\setminus(\partial \cA\cup\partial\cB)
\ar@/_7ex/[uu]_{\Delta}
\save "3,1"."1,9"*+<1ex,0ex>[F.]\frm{}\restore
}$} \caption{Proving Theorem \ref{thm:boundariescompatiblewithcupandcap}}
\label{fig:Proofofmultiplicativitylemma}
\end{sidewaysfigure}
All solid vertical maps are defined by pulling back functions along 
proper continuous maps between pairs of $\sigma$-locally compact spaces,
even the inclusion of function algebras which is obtained by pulling back along the identity. We already know the notations $\widetilde H^A,\widetilde H^B$ for the associated continuous homotopies and the resulting proper continuous map $\widetilde\Gamma^{\cX,\cA,\cB}$ and continuous map $\widetilde H^\cA_-$. Furthermore, $\Delta$ denotes the diagonal map. The other continuous $\sigma$-maps are defined as follows:
\begin{align*}
\gamma\colon\cX\times [-1,1]&\to\cX\times [-1,1]\,,&(x,s)&\mapsto (H^{\cA}(x,s^-),2s^+-1)
\\\beta\colon \cX\times [0,1]&\to \cX\times\cX\,,& (x,s)&\mapsto (H^{\cA}(x,s^-),x)
\end{align*}
It is straightforward to check that they map the function algebras as stated.
Furthermore, the $\gamma^*$ in the bottom left corner of the first upper diagram is simply the canonical $*$-homomorphism
\[\Cz((\cX\setminus(\cA\cup\cB))\times \mathring I)\cong \Cz((\cX\setminus(\cA\cup\cB))\times\mathring I_+)\subset \Cz((\cX\setminus(\cA\cup\cB))\times\mathring I)\]
which is, of course, homotopic to the identity.

The two dashed arrows are both the one which is induced by $\gamma$, that is, the one which makes the lower half of the upper diagram commute. It maps a function $f\in\Cz(\cX\setminus(\cA\cup\cB))$ to the class of the function 
\[\big((x,s)\mapsto f(H^{\cA}(x,s^-))\cdot(1-s^+)\big)\in \Cb(\cX\times I,\cX\times\partial I\cup\cA\times I)\,,\]
 because for all $f\in\Cz((\cX\setminus(\cA\cup\cB))\times [-1,1))$ the function which maps $(x,s)$ to
\begin{align*}
f(\gamma(x,s))-&f(H^{\cA}(x,s^-),-1)\cdot(1-s^+)=\\&=\begin{cases}0&s\leq 0\\f(x,2s^+-1)-f(x,-1)\cdot(1-s^+)&s\geq 0\end{cases}
\end{align*}
is an element of the ideal $\Cz((\cX\setminus(\cA\cup\cB))\times\mathring I)$.

As the upper half of the upper diagram also commutes
by definition, we just have to check commutativity of the middle diagram. For the right half this follows trivially from $\beta\circ\ev{1}=\Delta$ and $H^{\cA}_-\circ\ev{1}=\id$.
The upper left square commutes, because for $f\in\Cz(\cX\setminus\cA)$ and $g\in\Cz(\overline\cX\setminus\overline\cB)$ we have
\[(\widetilde\Gamma^{\cX,\cA,\cB})^*(f\otimes g)-\beta^*(f\otimes g)\colon (x,s)\mapsto\begin{cases}0&x\leq 0\\f(x)\cdot (g(H^{\cB}(x,s^+))-g(x))&x\geq 0\end{cases}\]
and this function lies in the ideal $\Cz((\cX\setminus(\cA\cup\cB))\times(-1,1])$. 
For the bottom  left square of the middle diagram we use that for $f\in \Cz(\cX\setminus(\cA\cup\cB))$ the function
\[(x,s)\mapsto f(H^{\cA}(x,s^-))-f(H^{\cA}(x,s^-))\cdot(1-s^+)=s^+\cdot f(x)\]
lies in $\Cz((\cX\setminus(\cA\cup\cB))\times(-1,1])$.

Dualizing these diagrams we obtain the diagram of spaces at the bottom of \Cref{fig:Proofofmultiplicativitylemma}, where the inclusions going right (i.\,e.\ the inclusions of the first into the third and the fourth into the fifth column) are inclusions of open subsets and the inclusions going left are inclusions of closed subsets. Furthermore, the spaces in the second and fourth column are the disjoint union of their left and right neighbors. This is exactly the data which yields the four upper commuting squares in the diagram
\[\mathclap{\xymatrix{
\Homol[1]_{m+1}((\cX\setminus(\cA\cup\cB))\times\mathring I)\ar[r]^-{\bound[1]}_-{\cong}\ar@{=}[d]
&\Homol[1]_{m}(\cX\setminus(\cA\cup\cB))\ar[r]^-{\bound[1]}
&\Homol[1]_{m-1}(\partial\cX\setminus(\partial \cA\cup\partial\cB))\ar@/^7ex/[dd]^{\Delta_*}
\\\Homol[1]_{m+1}((\cX\setminus(\cA\cup\cB))\times\mathring I)\ar[r]^-{\bound[1]}\ar[d]^{\widetilde\Gamma^{\cX,\cA,\cB}}
&\Homol[1]_m(\cY_2)\ar[u]\ar[r]^-{\bound[1]}\ar[d]
&\Homol[1]_{m-1}(\cY_4)\ar[u]\ar[d]
\\\Estack{\Homol[1]_{m+1}(&(\cX\setminus\cA)\\&\times(\cX\setminus\cB))}\ar[r]^-{\bound[1]}\ar[d]^{/\cobound[2](y)}
&\Estack{\Homol[1]_m(&(\cX\setminus\cA)\\&\times(\partial\cX\setminus\partial\cB))}\ar[r]^-{\bound[1]}\ar[d]^{/y}
&\Estack{\Homol[1]_{m-1}(&(\partial\cX\setminus\partial\cA)\\&\times(\partial\cX\setminus\partial\cB))}\ar[d]^{/y}
\\\Homol[3]_{m-n}(\cX\setminus\cA)\ar@{=}[r]
&\Homol[3]_{m-n}(\cX\setminus\cA)\ar[r]^-{\bound[3]}
&\Homol[3]_{m-n-1}(\partial\cX\setminus\partial\cA)
}}\]
and the bottom right square commutes whereas the bottom left square commutes up to a sign $(-1)^{m+1}$ for all $y\in\Cohom[2]^n(\partial\cX\setminus\partial\cB)$ by the axioms of slant products.
The secondary cap product with $\cobound[2](y)$ is by definition $(-1)^{m+1}$ times the composition from the top middle entry along the left side to the bottom left/middle entry. The  composition along the right hand side on the other hand is the primary cap product with $y$ itself. The second claim follows.

We omit the proof of the first claim, because it is completely analogous to the proof of the second claim and also very similar to the proof of \cite[Theorem 9.1]{WulffCoassemblyRinghomo}
\end{proof}

\section{Assembly and coassembly}
\label{sec:assembly}

The relative versions of assembly and coassembly with coefficients in a \textCstar-algebra $D$,
\begin{align*}
\mu\colon\KX_*(X,A;D)&\to \K_*(\Roe(X,A;D))\quad\text{and}
\\\mu^*\colon\K_{-*}(\sHigCor(X,A;D))&\to\KX^{1+*}(X,A;D)\,,
\end{align*}
respectively, are natural transformations of coarse homology and cohomology theories on the category of coarsely connected proper metric spaces.
(Of course, coarse connectedness of metric spaces is equivalent to saying that the metric takes on only finite values.)
We refer to \cite[Sections 6.3,6.4]{wulff2020equivariant} for their definition and note that the technical Assumption 6.8 made there is always true in the non-equivariant case, which is the only case we are interested in here.
We assume that the reader is already familiar with this theory and will only recall the relevant definitions in a very superficial manner.

The goal of this section is to show that canonical primary cup and cap products
\begin{align}
\cup\colon\K_{-m}(\sHigCor(X,A;D))\otimes\K_{-n}(\sHigCor(X,B;E))&\to \K_{-m-n}(\sHigCor(X,A\cup B;D\otimes E))\label{eq:primarycupHigson}
\\\cap\colon\K_m(\Roe(X,A\cup B;D))\otimes \K_{-n}(\sHigCor(X,B;E))&\to \K_{m-n}(\Roe(X,A;D\otimes E))\label{eq:primarycapRoe}
\end{align}
correspond to our secondary cup and cap products
\begin{align}
\Cup\colon\KX^{1+m}(X,A;D)\otimes\KX^{1+n}(X,B;D)&\to\KX^{1+m+n}(X,A\cup B;D\otimes E)\label{eq:KXsecondarycup}
\\\Cap\colon\KX_m(X,A\cup B;D)\otimes\KX^{1+n}(X,B;D)&\to\KX_{m-n}(X,A;D\otimes E)\label{eq:KXsecondarycap}
\end{align}
under assembly and co-assembly. 
Regarding the tensor products on the right hand sides we have to mention that throughout this section we deviate from the usual conventions and use the symbol $\otimes$ for the \emph{maximal} tensor product of \textCstar-algebras instead of the minimal one.

\begin{thm}\label{thm:cupcapassembly}
Let $(X;A,B)$ be a deformation triad of coarsely connected proper metric spaces and let $D,E$ be \textCstar-algebras.
Then 
\[\mu^*(x\cup y)=\mu^*(x)\Cup\mu^*(y)\]
for all $x\in \K_{-m}(\sHigCor(X,A;D)), y\in \K_{-n}(\sHigCor(X,B;E))$. If in addition $X$ has coarsely bounded geometry, then
\[\mu(x\Cap\mu^*(y))=\mu(x)\cap y\]
for all $x\in \KX_m(X,A\cup B;D), y\in \K_{-n}(\sHigCor(X,B;E))$.
\end{thm}

Here, coarsely bounded geometry denotes the usual notion of bounded geometry for metric spaces: A space $X$ has coarsely bounded geometry if it contains a coarsely dense uniformly locally finite subset $X'\subset X$, which we call a uniform discretization. Uniform finiteness means that for each $R>0$ there is an upper bound on the number of points in $B_R(x)\cap X'$ for all $x\in X'$.

Before we can prove this theorem, we have to briefly recall the definitions of the objects appearing here and introduce the primary cup and cap products. 

We start with the stable Higson coronas, which were first introduced in \cite{EmeMeyDualizing}. A continuous function $f\colon X\to D$ from a proper metric space $X$ into a \textCstar-algebra $D$ is said to have \emph{vanishing variation} if the function
\[\Var_E(f)\colon X\to[0,\infty]\,,\quad x\mapsto \sup\{\|f(x)-f(y)\|\mid (x,y)\in E\}\]
vanishes at infinity for every entourage $E$. Let $\ell^2\coloneqq\ell^2(\N)$ denote our favorite infinite dimensional separable Hilbert space and $\Kom\coloneqq\Kom(\ell^2)$ the \textCstar-algebra of compact operators on it. The \emph{stable Higson compactification} of a proper metric space $X$ with coefficients in a \textCstar-algebra $D$ is the \textCstar-algebra $\sHigCom(X;D)$ of all bounded continuous functions $f\colon X\to D\otimes \K$ of vanishing variation.
It contains $\Cz(X)\otimes D\otimes\K$ as an ideal and the quotient \textCstar-algebra $\sHigCor(X;D)\coloneqq\sHigCom(X;D)/\Cz(X)\otimes D\otimes\K$ is called the \emph{stable Higson corona}. Given furthermore a closed subspace $A\subset X$, we can define the relative versions of these \textCstar-algebras as
\begin{align*}
\sHigCom(X,A;D)&\coloneqq\ker(\sHigCom(X;D)\to\sHigCom(A;D))
\\\sHigCor(X,A;D)&\coloneqq\ker(\sHigCor(X;D)\to\sHigCor(A;D))\,.
\end{align*}
\begin{defn}
Given any coarsely connected proper metric space $X$ with closed subspaces $A,B\subset X$ and \textCstar-algebras $D,E$, the primary cap product \eqref{eq:primarycupHigson} is defined as the composition of the exterior product with the homomorphism induced by the multiplication $*$-homomorphism 
\begin{align*}
\nabla\colon&\sHigCor(X,A;D) \otimes \sHigCor(X,B;E)\to \sHigCor(X,A\cup B;D\otimes E)
\\ &[f]\otimes [g]\mapsto [x\mapsto f(x)\otimes g(x)\in D\otimes \Kom\otimes E\otimes\Kom\cong D\otimes E\otimes\Kom]\,.
\end{align*}
\end{defn}
The definition involved the choice of an identification $\Kom\otimes\Kom\cong \Kom$, but this is unique up to homotopy. Hence the cup product is unambiguous.

The coarse $\K$-theory $\KX^*(-,-;D)$ with coefficients in $D$ is defined by coarsification as in the previous section of the cohomology theory for $\sigma$-locally compact spaces defined by
\[\K^*(\cX,\cA;D)\coloneqq \K_{-*}(\Cz(\cX\setminus\cA)\otimes D)\,.\]
Here, $\Cz(\cX\setminus\cA)$ is a $\sigma$-\textCstar-algebra and hence use Phillips' $\K$-theory for $\sigma$-\textCstar-algebras on the right hand side (see \cite{PhiRep}, but also \cite{PhiFre}). The latter can be equipped with an exterior tensor product (see \cite[Section 5]{WulffCoassemblyRinghomo}) and if $E$ is another \textCstar-algebra, then the resulting cross products
\[\times\colon \K^m(\cX,\cA;D)\otimes \K^n(\cY,\cB;E)\to \K^{m+n}((\cX,\cA)\times(\cY,\cB);D\otimes E)\]
indeed constitute a cross product in the sense of \Cref{defn:topologicalcrossslant}. Therefore, we also obtain cross products 
\[\times\colon\KX^*(-,-;D)\otimes \KX^*(-,-;E)\to \KX^*(-,-;D\otimes E)\]
in the sense of \Cref{defn:crossproducts}, which induce the secondary cup products \eqref{eq:KXsecondarycup} as in the previous section.

Now, the \emph{uncoarsified co-assembly map} is nothing but the connecting homomorphism $\K_{-*}(\sHigCor(X,A;D))\to\K^{1+*}(X,A;D)$ associated to the short exact sequence
\[0\to\Cz(X\setminus A)\otimes D\otimes \Kom\to \sHigCom(X,A;D)\to \sHigCor(X,A;D)\to 0\,.\]
In order to obtain the coarsified co-assembly map $\mu^*$, we need to perform the same construction on the pair of Rips complexes $(\cP(X'),\cP(A'))$ of a discretization $(X',A')\subset (X,A)$ instead of on $(X,A)$ itself. 
To this end, we equip $\cP(X')$ with the metric constructed in \cite[Lemma 6.10]{wulff2020equivariant} whose restriction to each finite scale Rips complex $P_n(X')$ is a proper metric and for which all the inclusions $X\supset X'\subset X_m\subset X_n$ are isometric coarse equivalences for all $0\leq m\leq n$.
Then we have $\sHigCor(X,A;D)\cong\varprojlim_{n\in\N}\sHigCor(P_{E_n}(X'),P_{E_n}(A');D)$ and we define the $\sigma$-\textCstar-algebra 
\(\sHigCom(\cP(X'),\cP(A');D)\coloneqq \varprojlim_{n\in\N}\sHigCom(P_{E_n}(X'),P_{E_n}(A');D)\).
The \emph{co-assembly map} $\mu^*$ is defined as the connecting homomorphism in $\K$-theory for $\sigma$-\textCstar-algebras associated to the short exact sequence
\[0\to\Cz(\cP(X')\setminus \cP(A'))\otimes D\otimes \Kom\to \sHigCom(\cP(X'),\cP(A');D)\to \sHigCor(X,A;D)\to 0\,.\]

\begin{proof}[How to prove the first part of \Cref{thm:cupcapassembly}]
The special case of the theorem with $X$ being coarsely contractible to $A=B=\{p\}$ was the main result of \cite{WulffCoassemblyRinghomo}. 
In order to prove \Cref{thm:cupcapassembly} one only needs to generalize \cite[Theorem 9.1]{WulffCoassemblyRinghomo}, which is a straightforward task. The essential ideas which are necessary to do so can be found in our proof of \Cref{thm:boundariescompatiblewithcupandcap} above.
\end{proof}

Next up, we briefly recall the definition of the relative Roe algebras $\Roe(X,A;D)$ with coefficients in $D$.
We fix an ample representation $\rho_X\colon\Cz(X)\to \Lin(H_X)$ on a separable Hilbert space $H_X$ and consider the Hilbert-$D$-module $H_X\otimes D$ equipped with the representation $\rho_X\otimes \id_D$ of $\Cz(X)$. In the following, we will consider adjointable operators $T\in\Lin(H_X\otimes D)$ on the Hilbert-$D$-module $H_X\otimes D$. We say that:
\begin{itemize}
\item $T$ has \emph{finite propagation} if $\supp(T)$ is an entourage in $X\times X$.
In this case, the \emph{propagation} $\Prop(T)$ of $T$ is the smallest $R\geq 0$ such that $d(x,y)\leq R$ for all $(x,y)\in\supp(T)$.

(Recall that the support of an operator $T\colon\HilbertMod_1\to\HilbertMod_2$ between two Hilbert modules $\HilbertMod_{1,2}$ equipped with representations $\rho_{1,2}\colon \Cz(X_{1,2})\to\Lin(\HilbertMod_{1,2})$ is the largest closed subset  $\supp(T)\subset X_1\times X_2$ such that $\rho_2(g)\circ T\circ\rho_1(f)=0$ for all $f\in \Cz(X_1)$ and $g\in\Cz(X_2)$ with $\supp(f)\times\supp(g)\cap \supp(T)=\emptyset$.)
\item $T$ is locally compact if $(\rho_X(f)\otimes \id_D)\circ T,T\circ(\rho_X(f)\otimes \id_D)\in \Kom(H_X\otimes D)=\Kom(H_X)\otimes D$ for all $f\in\Cz(X)$.
\item $T$ is supported near $A$
 if there is $R\geq 0$ such that $\supp(T)\subset \Pen_R(A)\times \Pen_R(A)$, where $\Pen_R(A)\coloneqq \Pen_{E_R}(A)$ denotes the $R$-neighborhood of $A$ or equivalently
  $(\rho_X(f)\otimes \id_D)\circ T=0=T\circ(\rho_X(f)\otimes \id_D)$ for all $f\in\Cz(X)$ with $\dist(\supp(f),A)\geq R$.
\end{itemize}
Now we can define the following \emph{Roe algebras}: The \textCstar-algebra $\Roe(X;D)\subset\Lin(H_X\otimes D)$ is the norm closure of all locally compact operators of finite propagation. It contains the ideal $\Roe(A\subset X;D)\subset \Roe(X;D)$ which is the norm closure of all locally compact operators of finite propagation which are supported near $A$.
Furthermore, the quotient $\Roe(X,A;D)\coloneqq\Roe(X;D)/\Roe(A\subset X;D)$ is the relative Roe algebra. Note that $\Roe(X,\emptyset;D)=\Roe(X;D)$.

If $Y$ is another proper metric spaces and $\rho_Y\colon\Cz(Y)\to\Lin(H_Y)$ an ample representation on another separable Hilbert space $H_Y$, then an isometry $V\colon H_X\to H_Y$ is said to \emph{cover} a coarse map $\alpha\colon X\to Y$ if $(\id\times\alpha)(\supp(V))\subset Y\times Y$ is an entourage. 
Adjoining with $V\otimes \id_D$ yields a $*$-homomorphism $\Ad_{V\otimes \id_D}\colon\Roe(X;D)\to \Roe(Y;D)$. If $\alpha$ maps $A\subset X$ to $B\subset Y$, then $\Ad_{V\otimes \id_D}(\Roe(A\subset X;D))\subset\Roe(B\subset Y;D)$ and we also obtain a $*$-homomorphism between the relative Roe algebras.

Covering isometries always exist and the induced maps on $\K$-theory
\[\alpha_*\coloneqq (\Ad_{V\otimes \id_D})_*\colon\K_*(\Roe(X,A;D))\to K_*(\Roe(Y,B;D))\]
are independent of the choice of $V$. 
Applied to $\alpha=\id$ one can deduce that the $\K$-theory of the Roe algebras are independent of the choice of $\rho_X,H_X$ up to canonical isomorphism.
Furthermore, for an inclusion map $\alpha\colon A\subset X$ we have $\im(\Ad_{V\otimes \id_D})\subset \Roe(A\subset X;D)$ and the induced map on $\K$-theory is a canonical natural isomorphism $\K_*(\Roe(A;D))\cong\K_*(\Roe(A\subset X;D))$.

We can now introduce the cap product \eqref{eq:primarycapRoe} as a generalization of the module multiplication defined in \cite[Section 8]{WulffTwisted}.

\begin{lem}
\label{lem:definitionfm}
Let $X$ be a coarsely connected proper metric space of coarsely bounded geometry, $A,B\subset X$ closed subspaces with $A$ non-empty and
$D,E$ \textCstar-algebras. Choose an ample representation $\rho_X$ of $\Cz(X)$ on a separable Hilbert space $H_X$ and assume that the Roe algebras $\Roe(X,A\cup B;D)$ and $\Roe(X,A;D\otimes E)$ have been constructed using the representations $\rho_X\otimes\id_D$ on $H_X\otimes D$ and $\rho_X\otimes\id_{D\otimes E\otimes \ell^2}$ on $H_X\otimes D\otimes E\otimes \ell^2$, respectively. Furthermore, $\rho_X$ gives rise to a representation $\Cz(X)\otimes E\otimes\Kom$ on $H_X\otimes E\otimes \ell^2$ which extends by strong continuity to a representation $\bar\rho_X$ of $\multiplier(\Cz(X)\otimes E\otimes\Kom)$.
Then 
\begin{align*}
\fm\colon\Roe(X,A\cup B;D)\otimes\sHigCor(X,B;E)&\to\Roe(X,A;D\otimes E)
\\ [T]\otimes [f]&\mapsto [(T\otimes \id_{E\otimes \ell^2})\circ (\bar\rho_X(f)\otimes\id_D)]
\end{align*}
defines a $*$-homomorphism.
\end{lem}

\begin{proof}
The lemma is a generalization of the first part of \cite[Lemma 8.5]{WulffTwisted}, just without the gradings.
Thanks to the fact that $\rho_X$ always extends to a representation of all bounded Borel functions, the same arguments can be used in our case where $X$ is not necessarily a manifold and $H_X$ not necessarily an $L^2$-space. 
The proof shows that the commutators $[T\otimes\id_{E\otimes \ell^2},\bar\rho_X(f)\otimes\id_D]$ are compact operators for all $T\in \Roe(X;D)$ and all $f\in \sHigCom(X;E)$. This step requires the assumption of coarsely bounded geometry on $X$. Now, as $A$ is non-empty the compact operators are contained in $\Roe(A\subset X;D\otimes E)$ and hence we obtain a $*$-homomorphism
\begin{equation}\label{eq:fminducingstarhom}
\Roe(X;D)\otimes\sHigCom(X;E)\to\Roe(X,A;D\otimes E)
\end{equation}
by the universal property of the maximal tensor product, which we denoted by $\otimes$.
It clearly maps $\Roe(A\subset X;D)\otimes \sHigCom(X;E)$ to zero and the same ist true for $\Roe(B\subset X;D)\otimes \sHigCom(X,B;E)$. To see the latter, note that for all $R\geq 0$ each element of $\sHigCom(X,B;E)$ and can be represented by a function $f$ which vanishes on an $R$-neighborhood of $B$. Therefore, if $T$ is supported near $B$ we just have to choose $R$ sufficiently large and obtain $(T\otimes \id_{E\otimes \ell^2})\circ \bar\rho_X(f)=0$.
We now use the fact that 
\begin{equation}
\label{eq:SumRoeIdeals}
\Roe(A\subset X;D)+\Roe(B\subset X;D)=\Roe(A\cup B\subset X;D)
\end{equation}
and see that $\fm$ is well-defined as the quotient of a restriction of \eqref{eq:fminducingstarhom}.
\end{proof}

To avoid possible confusion about an alleged necessity for excisiveness, we note that \eqref{eq:SumRoeIdeals} holds for \emph{all} triads: If $T\in \Roe(X;D)$ is supported in $\Pen_R(A\cup B)\times \Pen_R(A\cup B)$ and $P\in\Lin(H_X\otimes D)$ is the projection corresponding to the characteristic function of $\Pen_R(A)$, then $PT\in\Roe(A\subset X;D)$ and $(1-P)T\in\Roe(B\subset X;D)$. Excisiveness would only be necessary for the other condition for coarse Mayer--Vietoris, i.\,e.\ $\Roe(A\subset X;D)\cap \Roe(B\subset X;D)=\Roe(A\cap B\subset X;D)$, which we do not need here.

\begin{defn}
Given a triad $(X;A,B)$ of coarsely connected proper metric spaces of coarsely bounded geometry with $A$ non-empty and \textCstar-algebras $D,E$, the cap product \eqref{eq:primarycapRoe} is defined as the composition of the exterior product with the homomorphism induced by $\fm$.
\end{defn}

\begin{lem}
The cap product is natural under coarse maps $\alpha\colon X\to X'$ taking the subspaces $A,B$ into subspaces $A',B'$ in the sense that
\[\alpha_*(x\cap \alpha^*y)=(\alpha_*x)\cap y\]
for all $x\in\K_m(\Roe(X,A\cup B;D))$ and $y\in \K_{-n}(\sHigCor(X',B';E))$.
\end{lem}
\begin{proof}
It is easy to see that it suffices to show that $(V\otimes\id_{E\otimes \ell^2})\circ\bar\rho_X(g)-\bar\rho_{X'}(f)\circ (V\otimes\id_{E\otimes \ell^2})$ is compact for all $g\in\sHigCom(X;D)$ and $f\in\sHigCom(X';E)$ such that the (possibly non-continuous) function $g-\alpha^*f$ converges to zero at infinity. This can be done exactly as in the proof of \cite[Theorem 4.28]{EngelWulffZeidler}.
\end{proof}

The second part of \cite[Lemma 8.5]{WulffTwisted} also directly generalizes to our case, yielding the following 
\begin{lem}
Let $X$ be a coarsely connected proper metric space of coarsely bounded geometry with  non-empty closed subspaces $A,B,C\subset X$
 and let and $D,E,F$ be \textCstar-algebras. 
Then the diagram
\[\xymatrix@C=8em{
{\Estack{&\Roe(X,A\cup B\cup C;D)\otimes\\&\sHigCor(X,B;E)\otimes\sHigCor(X,C;F)}}
\ar[r]^{\fm\otimes\id}
\ar[d]_{\id\otimes\nabla}
&{\Estack{\Roe(X,A\cup C;D\otimes E)&\\\otimes\sHigCor(X,C;F)&}}
\ar[d]^{\fm}
\\{\Estack{\Roe(X,A\cup B\cup C;D)&\\\otimes\sHigCor(X,B\cup C;E\otimes F)&}}
\ar[r]^{\fm}
&\Roe(X,A;D\otimes E\otimes F)
}\]
commutes after adequately identifying the underlying Hilbert modules. Consequently, the cup and cap products \eqref{eq:primarycupHigson},\eqref{eq:primarycapRoe} are associative: $x\cap(y\cup z)=(x\cap y)\cap z$.
\qed
\end{lem}

The coarse $\K$-homology $\KX_*(-,-;D)$ with coefficients in $D$ is defined by coarsification of the $\K$-homology with coefficients in $D$. We define the latter for pairs of second countable locally compact Hausdorff spaces via $\EE$-theory for \textCstar-algebras by
\[\K_*(X,A;D)\coloneqq \EE_*(\Cz(X\setminus A),D)\]
and, subsequently, if a pair $(\cX,\cA)$ of $\sigma$-locally compact spaces is given by a sequence of pairs $(X_r,A_r)$ (with $A_r=\cA\cap X_r$) of locally compact spaces which is second countable, then we define it by
\[\K_*(\cX,\cA;D)\coloneqq \varinjlim_{r\in\N} \K_*(X_r,A_r;D)\,.\]
As the finite scale Rips complexes of countable discrete coarse spaces are always second countable, the coarsification process of \Cref{sec:transgression} still works out fine for $\K$-homology.
Using the fact that on the cohomological side there are canonical natural epimorphisms
\[\K^*(\cX,\cA;D)\to \varprojlim_{r\in\N} \EE_{-*}(\C,\Cz(X_r\setminus A_r)\otimes D)\]
which are compatible with cross products,
the composition product in $\EE$-theory immediately gives rise to slant products
\[/\colon\K_m((\cX,\cA)\times(\cY,\cB);D\otimes E)\otimes\K^n(\cY,\cB;E)\to\K_{m-n}(\cX,\cA;D)\]
in the sense of \Cref{defn:topologicalcrossslant}, and the cross and slant products are compatible in a topological analogue sense to the condition in \Cref{thm:AssociativityCupCap}.
This provides us with all the data needed to obtain the secondary cap products and show that they are even associative with the secondary cup products.

The coarse assembly map $\mu$ can only be defined by using other pictures of $\K$-homology. 
One such picture is provided by the localization algebras introduced in \cite{YuLocalization}. 
See also \cite{WillettYuHigherIndexTheory} for a comprehensive exposition on this topic, but be aware that the ``localized Roe algebras'' introduced there are slightly bigger than Yu's original algebras.
We briefly recall their definition, but we additionally implement the coefficients.

With the notation introduced above, we denote by $\Loc(X;D)$ the sub-\textCstar-algebra of $\Cb([1,\infty),\Roe(X;D))$ generated by all uniformly continuous families $(T_t)_{t\geq 1}$ of locally compact adjointable operators of finite propagation with $\Prop(T_t)\xrightarrow{t\to\infty}0$. It contains the ideal $\Loc(A\subset X;D)$ generated by those families for which there is a function $\varepsilon\colon [1,\infty)\to(0,\infty)$ with $\varepsilon(t)\xrightarrow{t\to\infty}0$ such that $T_t$ is supported in the $\varepsilon(t)$-neighborhood of $A\times A$ for all $t\geq 1$ and we define the relative localization algebra $\Loc(X,A;D)\coloneqq\Loc(X;D)/\Loc(A\subset X;D)$.

Now, as disscussed in the context of \cite[Assumption 6.8]{wulff2020equivariant}, which was only formulated as an assumption because of the unclear situation in the equivariant set-up, there are canonical natural isomorphisms 
\begin{equation*}\Delta\colon \K_*(\Loc(X,A;D))\cong\K_*(X,A;D)\,.
\end{equation*}
Here, naturality is with respect to uniformly continuous coarse maps $\alpha\colon (X,A)\to(Y,B)$, which induce homomorphisms between the 
domains of $\Delta$
 as follows:
We say that a uniformly continuous family of isometries $V\colon [1,\infty)\to\Lin(H_X,H_Y),t\mapsto V_t$ \emph{covers} $\alpha$ if 
\begin{equation}\label{eq:continuouscoveringfamily}
\sup\{d_Y(y,\alpha(x))\mid (y,x)\in\supp(V_t)\}\begin{cases}<\infty&\text{for all }t\geq 1\text{ and}\\\to 0&\text{for }t\to\infty\,.\end{cases}
\end{equation}
Such families of covering isometries always exist, because we have assumed that $\rho_Y$ is ample (compare \cite[Proposition 3.2]{QiaoRoe}, \cite[Theorem 6.6.3]{WillettYuHigherIndexTheory}), and adjoining with $V\otimes\id_D$ gives rise to the induced maps
\[\alpha_*\coloneqq (\Ad_{V\otimes \id_D})_*\colon\K_*(\Loc(X,A;D))\to K_*(\Loc(Y,B;D))\,.\]
As a special case of functoriality, the inclusions $A\subset X$ induce canonical natural isomorphisms $\K_*(\Loc(A;D))\cong\K_*(\Loc(A\subset X;D))$, which is in complete analogy to what we wrote about of Roe algebras. 

Now, the \emph{uncoarsified assembly map} is simply the map 
\[\K_*(X,A;D)\cong\K_*(\Loc(X,A;D))\xrightarrow{(\ev{1})_*}\K_*(\Roe(X,A;D))\]
induced by evaluation at 1, and this construction is readily coarsified:
Again, we let $(X',A')\subset (X,A)$ be a discretization and equip the Rips complexes $\cP_n(X')$ with the metric from \cite[Lemma 6.10]{wulff2020equivariant} such that all the inclusions $X\supset X'\subset P_n(X)\subset P_m(X)$ (for $0\leq m\leq n$) are isometric coarse equivalences.
By exploiting the functoriality of Roe and localization algebras under coarse and uniformly continuous coarse maps, respectively, the \emph{assembly map} $\mu$ is simply defined as
\begin{align*}
\mu\colon\KX_*(X,A;D)\cong &\varinjlim_{r\in\N} \K_*(\Loc(P_r(X'),P_r(A');D))
\\\xrightarrow{(\ev1)_*} &\varinjlim_{r\in\N} \K_*(\Roe(P_r(X'),P_r(A');D))\cong \K_*(\Roe(X,A;D))\,,
\end{align*}
and we see that this definition is independent of the choice of refinement.
Furthermore, by using a sufficiently thinned out discretization $(X',A')$, the fourth part of \cite[Lemma 6.10]{wulff2020equivariant} then directly implies that the assembly map $\mu$ is natural under coarse maps.

We have now introduced all objects appearing in the second part of \Cref{thm:cupcapassembly}, but instead of proving it directly, we prove a stronger localized version. It involves the following localized cap product.
To this end, we note that similarily to \eqref{eq:SumRoeIdeals} we also have 
$\Loc(A\subset X;D)+\Loc(B\subset X;D)=\Loc(A\cup B\subset X;D)$ for arbitrary triads $(X;A,B)$ of proper metric spaces and hence the construction of the multiplication map $\fm$ in \Cref{lem:definitionfm} can be performed pointwise over $[1,\infty)$ to obtain a $*$-homomorphism
\[\fm_\rL\colon\Loc(X,A\cup B;D)\otimes\sHigCor(X,B;E)\to\Loc(X,A;D\otimes E)\,.\]
It induce a natural primary cap product
\begin{equation}\label{eq:capKhomstableHigsoncorona}
\cap\colon\K_m(\Loc(X,A\cup B;D))\otimes\K_{-n}(\sHigCor(X,B;E))\to \K_m(\Loc(X,A;D\otimes E))\,.
\end{equation}

\begin{thm}\label{thm:localizedcap}
Let $(X;A,B)$ be a deformation triad of proper metric space of coarsely bounded geometry and let $D,E$ be \textCstar-algebras.
Then the coarsification 
\[\cap\colon\KX_m(X,A\cup B;D)\otimes\K_{-n}(\sHigCor(X,B;E))\to \KX_{m-n}(X,A;D\otimes E)\]
of \eqref{eq:capKhomstableHigsoncorona} corresponds to the secondary cap product under coassembly in the sense that
$x\Cap\mu^*(y)=x\cap y$
for all $x\in \KX_m(X,A\cup B;D), y\in \K_{-n}(\sHigCor(X,B;E))$.
\end{thm}

\begin{proof}[Proof of the second part of \Cref{thm:cupcapassembly}]
The claim follows immediately from \Cref{thm:localizedcap}, because the cap product \eqref{eq:capKhomstableHigsoncorona} clearly corresponds to the cap product on the $\K$-theory of the Roe-algebras under evaluation at $1$.
\end{proof}

In order to prove \Cref{thm:localizedcap}, we first need to pass from the $\EE$-theoretic slant product to a slant product on the $\K$-theory of the localization algebras.
In the absolute case without coefficients, this was done in \cite[Section 4.3]{EngelWulffZeidler}.
We shall briefly recall the constructions, simultaneously adapting them to the relative case with coefficients.
The proofs work exactly the same way in this more general set-up.

Let $(X,A)$ and $(Y,B)$ be pairs of proper metric spaces and $\rho_X\colon \Cz(X)\to\Lin(H_X)$, $\rho_Y\colon\Cz(Y)\to\Lin(H_Y)$ be ample representations as before. 
The tensor product of $\rho_Y$, and the canonical representations of $E$ on itself and of $\Kom$ on $\ell^2$ is a non-degenerate representation of $\Cz(Y,E\otimes \Kom)$ on the Hilbert module $H_Y\otimes E\otimes\ell^2$. 
By \cite[Theorem II.7.3.9]{BlackadarOA} it extends uniquely to a strictly continuous representation 
\[\bar\rho_Y\colon\multiplier(\Cz(Y,E\otimes\Kom))\to\Lin(H_Y\otimes E\otimes\ell^2)\]
of the multiplier algebra.
Furthermore, we assume that the localization algebra $\Loc(X\times Y;D)$ has been constructed using the representation $\rho_X\otimes\rho_Y\otimes\id_D$ on $H_X\otimes H_Y\otimes D$ and $\Loc(X;D\otimes E)$ has been constructed using the representation $\tilde\rho\coloneqq\rho_X\otimes\id_{H_Y\otimes D\otimes E\otimes\ell^2}$ on $H_X\otimes H_Y\otimes D\otimes E\otimes\ell^2$.
The latter is an ideal in the sub-\textCstar-algebra 
\[\FiProLoc(X;D\otimes E)\subset\Cb( [1,\infty),\Lin(H_X\otimes H_Y\otimes D\otimes E\otimes\ell^2))\]
generated by all bounded and uniformly continuous functions $T\colon t\mapsto T_t$ such that the propagation $\Prop(T_t)$ with respect to the representation $\tilde\rho_X$ is finite for all $t\geq 1$ and tends to zero as $t\to\infty$.

Now, if $Y$ has bounded geometry, then implementing coefficients into \cite[Lemma 4.8(iii)]{EngelWulffZeidler} shows that there is a canonical $*$-homomorphism
\begin{align*}
\Psi_\rL\colon\Loc(X\times Y;D)\otimes\sHigCor(Y;E)&\to {\FiProLoc(X;D\otimes E)}/{\Loc(X;D\otimes E)}
\\T\otimes [f]&\mapsto [t\mapsto (T_t\otimes\id_{E\otimes\ell^2})\circ (\id_{H_X\otimes D}\otimes\bar\rho_Y(f))]\,.
\end{align*}
It maps $\Loc(X\times B\subset X\times Y;D)\otimes\sHigCor(Y,B;E)$ to zero, because each element of $\sHigCor(Y,B;E)$ can be represented by a function $f$ which vanishes on the $1$-neighborhood $\Pen_1(B)$ of $B$ and for each $T\in \Loc(X\times B\subset X\times Y;D)$ the $T_t$ are supported within $\Pen_1(X\times B)\times \Pen_1(X\times B)$ for $t$ large enough.
Furthermore, 
it maps $\Loc(A\times Y\subset X\times Y;D)\otimes\sHigCor(Y;E)$ to the subset of equivalence classes represented by elements of
the ideal $\FiProLoc(A\subset X;D\otimes E)\subset \FiProLoc(X;D\otimes E)$
 which we define
 as the closure of all $T$ for which there is in addition a function $\varepsilon\colon [1,\infty)\to(0,\infty)$ with $\varepsilon(t)\xrightarrow{t\to\infty}0$ such that $T_t$ is supported in the $\varepsilon(t)$-neighborhood of $A\times A$ for all $t\geq 1$. Let $\FiProLoc(X,A;D\otimes E)\coloneqq\FiProLoc(X;D\otimes E)/\FiProLoc(A\subset X;D\otimes E)$ and note that it contains $\Loc(X,A;D\otimes E)$ canonically as an ideal, because $\Loc(A\subset X;D\otimes E)=\FiProLoc(A\subset X;D\otimes E)\cap\Loc(X;D\otimes E)$.
Then the above considerations show that $\Psi_\rL$ induces a $*$-homomorphism 
\[\Psi_\rL\colon\Loc((X,A)\times (Y,B);D)\otimes\sHigCor(Y,B;E)\to {\FiProLoc(X,A;D\otimes E)}/{\Loc(X,A;D\otimes E)}\]
which we denote by the same letter.
\begin{defn}[Relative version with coefficients of {\cite[Definition 4.9]{EngelWulffZeidler}}]
\label{defn:slantproductKhomologyKstableHigsonCorona}
We define a slant product between the $\K$-theory of the localization algebra and the $\K$-theory of the stable Higson corona
as $(-1)^m$ times the composition
\begin{align*}
\K_m(\Loc((X,A)\times (Y,B);D))&\otimes \K_{1-n}(\sHigCorRed(Y,B;E))\to
\\&\to \K_{m+1-n}(\Loc((X,A)\times (Y,B);D)\otimes\sHigCorRed(Y,B;E))
\\&\xrightarrow{(\Psi_\rL)_*}\K_{m+1-n}\left({\FiProLoc(X,A;D\otimes E)}/{\Loc(X,A;D\otimes E)} \right)
\\&\xrightarrow{\partial} \K_{m+1-n}(\Loc(X,A;D\otimes E))\,,
\end{align*}
where the first arrow is the external tensor product and the last arrow is the boundary map of $\K$-theory.
\end{defn}

In the next lemma we have to assume that $Y$ has continuously bounded geometry (cf.\ \cite[Definition 4.1 (b)]{EngelWulffZeidler}), which is a notion of bounded geometry for metric spaces which is modelled after bounded geometry for complete Riemannian manifolds and implies coarsely bounded geometry.
It says the following: For every $r>0$ and $R>0$ there exists a constant $K_{r,R}>0$ such that
\begin{itemize}
\item for every $r>0$ there is a subset $\hat Y_r\subset Y$ such that $Y=\bigcup_{\hat y\in \hat Y_r}B_r(\hat y)$ and such that for all $r,R>0$ and $y\in Y$ the number $\#(\hat Y_r\cap \overline{B}_R(y))$ is bounded by $K_{r,R}$ and
\item for all $\alpha>0$ we have $\limsup_{r\to 0}K_{r,\alpha r}<\infty$.
\end{itemize}
Important to us is that this property holds for all Rips complexes $P_n(Y')$ equipped with the metric from \cite[Lemma 6.10]{wulff2020equivariant} of uniformly locally finite proper metric spaces $Y'$: 
In this case, the Rips complexes are locally finite and finite dimensional and hence it is easy to write down subsets $\hat Y_r\subset P_n(Y')$ witnessing the continuously bounded geometry, e.\,g.\ the subset of all points whose barycentric coordinates are multiples of $\frac1N$ for $N\in\N$ large enough. 

\begin{lem}[Relative version with coefficients of {\cite[Theorem 4.13]{EngelWulffZeidler}}]
Let $Y$ have continuously bounded geometry. Then the slant product from \ref{defn:slantproductKhomologyKstableHigsonCorona} and the slant product coming from $\EE$-theory are related via the uncoarsified coassembly map $\mu^*$ and the isomorphism $\Delta$ by the commutative diagram
\[\xymatrix{
\K_m(\Loc((X,A)\times (Y,B);D))\otimes \K_{1-n}(\sHigCorRed(Y,B;E))
\ar[r]^-{/}\ar[d]_{\Delta\otimes\mu^*}
&\K_{m-n}(\Loc(X,A;D\otimes E))\ar[d]^{\cong}_{\Delta}
\\\K_m((X,A)\times (Y,B);D)\otimes\K^{n}(Y,B;E)
\ar[r]^-{/}
&\K_{m-n}(X,A;D\otimes E)\,.
}\]
\end{lem}

\begin{proof}Continuously bounded geometry of $Y$ implies that there is a $*$-homo\-mor\-phism
\begin{align*}
\Upsilon_\rL\colon\Loc(X\times Y;D)\otimes\Cz(Y;E\otimes\Kom)&\to \Loc(X;D\otimes E)/\Cz([1,\infty),\Roe(X;D\otimes E))
\\T\otimes [f]&\mapsto [t\mapsto (T_t\otimes\id_{E\otimes\ell^2})\circ (\id_{H_X\otimes D}\otimes\bar\rho_Y(f))]
\end{align*}
as one easily sees by implementing coefficients into the arguments leading to \cite[Display (4.10)]{EngelWulffZeidler}.
It clearly vanishes for $T\in\Loc(X\times B\subset X\times Y;B)$ and $f\in\Cz(Y\setminus B;E\otimes\Kom)$, because for every $\varepsilon>0$ the support of $T_t$ will be contained in $f^{-1}(-\varepsilon,\varepsilon)\times f^{-1}(-\varepsilon,\varepsilon)$ for $t$ large enought and hence $\|(T_t\otimes\id_{E\otimes\ell^2})\circ (\id_{H_X\otimes D}\otimes\bar\rho_Y(f))\|\leq \|T\|\cdot\varepsilon$.
It clearly also maps $\Loc(A\times Y\subset X\times Y;B)\otimes\Cz(Y;E\otimes\Kom)$ to the ideal $\Loc(A\subset X;D\otimes E)/\Cz([1,\infty),\Roe(A\subset X;D\otimes E))$ and hence we obtain an induced $*$-homomorphism
\[\Upsilon_\rL\colon\Loc((X,A)\times (Y,B);D)\otimes\Cz(Y,B;E\otimes\Kom)\to \frac{\Loc(X,A;D\otimes E)}{\Cz([1,\infty),\Roe(X,A;D\otimes E))}\,.\]
The proof can then be finished just like in \cite[Section 4.3]{EngelWulffZeidler} by showing that the two slant products agree with the composition
\begin{align*}
\K_m(\Loc((X,A)&\times (Y,B);D))\otimes \K^{n}(Y,B;E)\to
\\&\to \K_{m-n}(\Loc((X,A)\times (Y,B);D)\otimes\Cz(Y,B;E\otimes\Kom))
\\&\xrightarrow{(\Upsilon_\rL)_*}\K_{m+1-n}\left(\frac{\Loc(X,A;D\otimes E)}{\Cz([1,\infty),\Roe(X,A;D\otimes E))} \right)
\\&\cong \K_{m-n}(\Loc(X,A;D\otimes E))\,.\qedhere
\end{align*}
\end{proof}

\begin{proof}[Proof of \Cref{thm:localizedcap}]
Let $X'\subset X$ be a uniform discretization such that $A'\coloneqq A\cap X'$ and $B'\coloneqq B\cap X'$ are also uniform discretizations. We may also assume that $X'$ was chosen in a $\delta$-separated manner for some $\delta>0$, that is, the distance between distinct points of $X'$ is at least $\delta$. 
For each $r\geq 0$ we introduce the short notations $X_r\coloneqq P_r(X')$, $A_r\coloneqq P_r(A')$, $B_r\coloneqq P_r(B')$ for the associated Rips complexes at scale $r$.
As in the last few sections, we also write $I\coloneqq [-1,1]$, $I_+\coloneqq[0,1]$, $I_-\coloneqq[-1,0]$.

In \Cref{lem:coarsetopologicalsecondarycomparison} we had reformulated the coarse secondary products as special cases of the topological secondary products. 
Recall that its proof involved the proper continuous $\sigma$-map $\widetilde\Gamma^{X,A,B}=(\widetilde H^A_-,\widetilde H^B_+)\colon \cP(X')\times I\to \cP(X')\times \cP(X')$.
In order to express the topological secondary product in terms of localization algebras and covering isometries, we thus have to  to metrize every $X_r\times I$ in such a way that all the restrictions of $\widetilde\Gamma^{X,A,B}$ of the form $X_r\times I\to X_R\times X_R$ are uniformly continuous coarse maps.  
This can be done as follows. First we equip $X\times I$ with the largest metric such that all the slices $X\times\{t\}\subset X\times I$ are isometric to $X$ and the points $(x,s),(x,t)$ have distance at most $1$ if $(s,t)\in U_x$. We do not care that this metric is not proper (it induces the topology which is the disjoint union topology of the slices $X\times\{t\}$) but it obviously induces the coarse structure of $X\indexcross{\cU}I$ and its restriction to the discretization $Z\subset X\indexcross{\cU}I$ appearing in the proof of \Cref{lem:coarsetopologicalsecondarycomparison} is a proper metric. 
Now, we may furthermore assume that the discretization $Z\subset X\indexcross{\cU}I$ containing $X'\cup\{-1,0,1\}$ is also $\delta$-separated, because otherwise we can simply thin it out.
Thus, if we equip all the Rips complexes with the metrics from \cite[Lemma 6.10]{wulff2020equivariant}, then all the restrictions of $\cP(\Gamma^{X,A,B})\colon\cP(Z)\to \cP(X')\times \cP(X')$ to finite scale Rips complexes are uniformly continuous coarse maps by the fourth part of that lemma. As $\widetilde\Gamma^{X,A,B}$ is the composition of $\cP(\Gamma^{X,A,B})$ with the embedding $i\colon\cP(X')\times I\to \cP(Z)$, the pull-back metric on $\cP(X')\times I$ does the job.
We denote the subspaces $X_r\times I$ equipped with the restrictions of this metric by $X_r\indexcross{\cU}I$.
Recall in particular that $Z$ contained $X'\times\{0\}$, so the subspace $\cP(X')\times\{0\}\subset \cP(X')\times I$ is isometric to $\cP(X')$ by construction.

We do have some degree of freedom in the choice of the representations in the construction of the various localization algebras. 
We chose $\rho_r$ to be the canonical ample representation of $\Cz(X_r)$ on $H_{X_r}\coloneqq \ell^2(X_r^\Q)\otimes\ell^2$ by multiplication operators, where $X_r^\Q\subset X_r$ denotes the subset of all points whose barycentric coordinates with respect to the simplex they lie in are rational.
The big advantage of this particular representation is that $H_X$ has an obvious canonical basis of vectors which are supported at single points, which makes working with covering isometries much easier later on.
For the same reason we choose for each closed interval $J$ the canonical representation $\rho_J\colon \Ct(J)\to\Lin(\ell^2(J\cap\Q))$ by multiplication. 
In the following, $J$ is always one of $I,I_\pm$.

We now consider a fixed $r\geq 0$ and choose $R\geq 0$ large enough such that $\widetilde\Gamma^{X,A,B}$ maps $X_r\times I$ into $X_R\times X_R$.
For these particular fixed $r,R$, we use the representations indicated in the following table to construct the various localization algebras.
\begin{center}
\begin{tabular}{|l|l|l|}
\hline 
\textCstar-algebra & representation & Hilbert module \\ 
\hline 
$\Loc(X_r;D)$ & $\rho_r\otimes\id_D$ & $H_{X_r}\otimes D$ \\ 
\hline 
$\Loc(X_r;D\otimes E)$ & $\rho_r\otimes\id_{D\otimes E\otimes \ell^2}$ & $H_{X_r}\otimes D\otimes E\otimes \ell^2$ \\ 
\hline 
$\Loc(X_r\indexcross{\cU}J;D)$&$\rho_r\otimes\rho_J\otimes\id_D$& $H_{X_r}\otimes\ell^2(J\cap\Q)\otimes D$ \\
\hline 
$\Loc(X_R\times X_R;D)$ & $\rho_R\otimes\rho_R\otimes\id_D$ & $H_{X_R}\otimes H_{X_R}\otimes D$ \\ 
\hline 
${\begin{matrix}\FiProLoc(X_R;D\otimes E)\\\Loc(X_R;D\otimes E)\end{matrix}}$ & $\rho_R\otimes\id_{H_{X_r}\otimes D\otimes E\otimes \ell^2}$ & $H_{X_R}\otimes H_{X_R}\otimes D\otimes E\otimes \ell^2$ \\ 
\hline 
\end{tabular} 
\end{center}
Of course, all of the derive ideals and quotients as well as the corresponding Roe algebras are assumed to be constructed using the same representations. 
Note also that we use different representation for $\Loc(X_r;D\otimes E)$ and $\Loc(X_R;D\otimes E)$ even if $r=R$.

With these choices of representations we have \emph{canonical} inclusions of \textCstar-algebras 
\[\Loc(X_r\indexcross{\cU}I_\pm;D)\subset\Loc(X_r\times I_\pm\subset X_r\indexcross{\cU}I;D)\]
with $\Loc(X_r\indexcross{\cU}I_+;D)\cap \Loc(X_r\indexcross{\cU}I_-;D)\cong\Loc(X_r;D)$.
The quotients of these \textCstar-algebras associated to closed subspaces clearly satisfy the analogue statements. More precisely, we are interested in the \textCstar-algebras 
\begin{align*}
L\coloneqq&\Loc(X_r,A_r\cup B_r)\indexcross{\cU}(I,\partial I);D)
\\L_\pm\coloneqq&\Loc((X_r,A_r\cup B_r)\indexcross{\cU}(I_\pm,\{\pm1\});D)
\\L_0\coloneqq&\Loc(X_r,A_r\cup B_r;D)
\end{align*}
and for these the above inclusions induce $*$-monomorphisms $L_\pm\hookrightarrow L$ and $L_0\hookrightarrow L_\pm$, which we write as inclusions, and then we have $L_0=L_+\cap L_-$, too.
Furthermore, the three \textCstar-algebras
\begin{align*}
L_\pm^\subset\coloneqq&\Loc((X_r,A_r\cup B_r)\times(I_\pm,\{\pm1\})\subset X_r\indexcross{\cU}I;D)
\\L_0^\subset\coloneqq&\Loc((X_r,A_r\cup B_r)\times\{0\}\subset X_r\indexcross{\cU}I;D)
\end{align*}
can be identified canonically with ideals in $L$ which satisfy $L=L_+^\subset+L_-^\subset$ and $L_0=L_+^\subset\cap L_-^\subset$.

Given any \textCstar-algebra $C$ with subalgebras $C_0,C_1$, we define the two sided mapping cylinder
\[\Zyl(C,C_0,C_1)\coloneqq\{f\in\Ct([0,1],C)\mid f(0)\in C_0,f(1)\in C_1\}\,.\]
Furthermore we recall that the mapping cone of a $*$-epimorphism $\varphi\colon C_1\to C_2$ is
\[\Cone(\varphi)\coloneqq \{(T,f)\in C_1\oplus\Cz([0,1),C_2)\mid \varphi(T)=f(0)\}\,.\]
We can then consider the following diagram, in which all arrows marked with $\simeq$ induce isomorphisms on $\K$-theory:
\[\xymatrix{
\Cz((0,1),L_-^\subset/L_0^\subset)\ar[r]
&\Cone\left(L_-^\subset\to L_-^\subset/L_0^\subset\right)
&L_0^\subset\ar[l]_-{\simeq}
\\\Cz((0,1),L)\ar[r]\ar[u]_-{\simeq}
&\Zyl(L,L_-^\subset,L_+^\subset)\ar[u]_-{\simeq}
&\Ct([0,1], L_0^\subset)\ar[l]_-{\simeq}\ar[u]^-{\ev{0}}_-{\simeq}
\\\Cz((0,1),L)\ar[r]\ar@{=}[u]
&\Zyl(L,L_-,L_+)\ar[u]_-{\simeq}
&L_0\ar[l]_-{\simeq}\ar[u]_-{\simeq}^-{\begin{smallmatrix}\text{incl. as}\\\text{const. fu.}\end{smallmatrix}}
}\]
The left and middle vertical arrows from the second to the first row are induced by the canonical $*$-epimorphism $L\to L/L_+^\subset\cong L_-^\subset/L_0^\subset$. They induce isomorphisms on $\K$-theory, because homotopy invariance of $\K$-homology implies $\K_*(L_+^\subset;D)\cong \K_*((X,A\cup B)\times(I_+,\{1\};D)=0$ and hence the respective kernels $\Cz((0,1),L_+^\subset)$ and $\Cz((0,1],L_+^\subset)$ of the two arrows have zero $\K$-theory, too.

The other arrows are defined in the obvious way. Recall that the first two rows are actually well known, as well as the fact that the arrows from the third to the second column in these rows induce isomorphisms on $\K$-theory: 
The first row is the one inducing the connecting homomorphism $\K_{m+1}(L_-^\subset/L_0^\subset)\to \K_m(L_0^\subset)$
 and the second row induces the boundary map in the Mayer--Vietoris sequence associated to the decomposition $L=L_+^\subset+L_-^\subset$ with $L_0=L_+^\subset\cap L_-^\subset$  (cf.\ \cite[Exercise 4.10.21]{HigRoe}).

Using the fact that the inclusions $L_\pm\subset L_\pm^\subset$ and $L_0\subset L_0^\subset$ induce isomorphisms in $\K$-theory, it is easy to deduce that the arrows from the third to the second row (and hence also the arrow going left in the third row) induce isomorphisms as well. 

Now, note that the connecting homomorphism $\K_{m+1}(L_-^\subset/L_0^\subset)\to \K_m(L_0^\subset)$ induced by the first row identifies with the connecting homomorphism in $\K$-homology $\K_*(-,-;D)$ associated to the triple 
\[(X_r\times [-1,0]\,,\quad X_r\times\{-1,0\}\cup (A_r\cup B_r)\times [-1,0]\,,\quad X_r\times\{-1\}\cup (A_r\cup B_r)\times [-1,0])\,.\]
Therefore, combining it with orientation preserving homeomorphism $[-1,0]\to I$ we see that the top row and left column together induce exactly the \emph{negative} of the inverse of suspension
\[\susp\colon \K_m(X_r,A_r\cup B_r;D)\to \K_{m+1}((X_r,A_r\cup B_r)\times (I,\partial I);D)\,.\]
The sign appeared because we have flipped the roles of the endpoints of the interval.
Hence, the bottom row of the diagram is (up to sign) a reformulation of suspension and it will enable us to compare the cap products.

To finish the proof we construct the following diagram.
\[\xymatrix@R=8ex{
\Estack{&\Cz(0,1)\otimes L\\&\otimes\sHigCor(X_R,B_R;E)}\ar[r]\ar[d]^{\Estack{&\scriptstyle\id_{\Cz(0,1)}\otimes\\&\scriptstyle\Psi_\rL\circ((\widetilde\Gamma^{X,A,B})_*\otimes\id)}}
&\Estack{&\Zyl(L,L_-,L_+)\\&\otimes\sHigCor(X_R,B_R;E)}\ar[d]^{\fM}
&L_0\otimes\sHigCor(X_R,B_R;E)\ar[l]_-{\simeq}\ar[d]^{\fM_0}
\\\Estack{&\Cz(0,1)\otimes\\&\frac{\FiProLoc(X_R,A_R;D\otimes E)}{\Loc(X_R,A_R;D\otimes E)}}\ar[r]
&\Cone{\begin{pmatrix}
{\scriptstyle\FiProLoc(X_R,A_R;D\otimes E)}\\\downarrow\\\frac{\FiProLoc(X_R,A_R;D\otimes E)}{\Loc(X_R,A_R;D\otimes E)}
\end{pmatrix}}
&\Loc(X_R,A_R;D\otimes E)\ar[l]_-{\simeq}
}\]
The left vertical arrow is obtained by applying first 
\[(\widetilde\Gamma^{X,A,B})_*\colon L=\Loc(X_r,A_r\cup B_r)\indexcross{\cU}(I,\partial I);D)\to \Loc((X_R,A_R)\times(X_R,B_R);D)\]
and then 
\[\Psi_\rL\colon\Loc((X_R,A_R)\times(X_R,B_R);D)\otimes \sHigCor(X_R,B_R;E)\to \frac{\FiProLoc(X_R,A_R;D\otimes E)}{\Loc(X_R,A_R;D\otimes E)}\]
pointwise. 
The right vertical arrow $\fM_0$ is the composition of
\[\fm\colon\Loc(X_r,A_r\cup B_r;D)\otimes \sHigCor(X_r,B_r;E)\to \Loc(X_r,A_r;D\otimes E)\subset \Loc(X_R,A_R;D\otimes E)\]
with the canonical identification $\iota_{r,R}^*\colon\sHigCor(X_R,B_R;E)\cong \sHigCor(X_r,B_r;E)$ and the $*$-homomorphism
\[(\iota_{r,R})_*\colon \Loc(X_r,A_r;D\otimes E)\to \Loc(X_R,A_R;D\otimes E)\]
induced by the inclusion $\iota_{r,R}\colon X_r\to X_R$.

It remains to combine these two constructions to obtain the middle vertical arrow $\fM$, which is still a lot of work. At this point it is very helpful to choose a very specific covering isometry yielding $(\widetilde\Gamma^{X,A,B})_*$. We may assume that the discretization $Z\subset X\indexcross{\cU}I$ is even contained in $X\times (I\cap \Q)$ and then $\widetilde\Gamma^{X,A,B}=(\widetilde H^A_-,\widetilde H^B_+)$ maps $X_r^\Q\times(I\cap\Q)$ into $X_R^\Q\times X_R^\Q$.
This allows us to directly write down the isometry 
\begin{alignat*}{3}
V\colon \underbrace{\ell^2(X_r^\Q)\otimes\ell^2}_{=H_{X_r}}&\otimes\ell^2(I\cap \Q) &\to&& \underbrace{\ell^2(X_R^\Q)\otimes\ell^2}_{=H_{X_R}} &\otimes \underbrace{\ell^2(X_R^\Q)\otimes\ell^2}_{=H_{X_R}}
\\\delta_x\otimes\delta_n&\otimes\delta_t &\mapsto&& \delta_{\widetilde H^A_-(x,t)}\otimes\delta_{n}&\otimes\delta_{\widetilde H^B_+(x,t)}\otimes\delta_{\beta(t)}\,.
\end{alignat*}
Here, $\{\delta_x\}$, $\{\delta_n\}$ and $\{\delta_t\}$ denote the obvious bases of the $\ell^2$ space and $\beta\colon I\cap\Q\to\N$ denotes some fixed bijection. In fact, it is rather irrelevant that we have defined $V$ using the isometry $\delta_n\otimes\delta_t\mapsto \delta_n\otimes\delta_{\beta(t)}$ from $\ell^2\otimes\ell^2(I\cap \Q)$ to $\ell^2\otimes\ell^2$, as any other one also has the same properties that we use in the following.

Note first of all that the constant family of isometries $t\mapsto V$ clearly covers $\widetilde\Gamma^{X,A,B}$, because the associated function \eqref{eq:continuouscoveringfamily} is the zero function. Hence we can define $(\widetilde\Gamma^{X,A,B})_*$ as adjoining with the constant families  $V_D\coloneqq V\otimes\id_D$ or $V_{D\otimes E}\coloneqq V\otimes\id_{D\otimes E\otimes\ell^2}$, depending on the coefficient algebras.

The second important property is the following. Let $\bar\rho_R\colon\Cb(X_R,E\otimes\Kom)\to\Lin(H_{X_R}\otimes E\otimes\ell^2)$ and $\bar\rho_{r,J}\colon\Cb(X_r\times J, E\otimes\Kom)\to\Lin(H_{X_r}\otimes\ell^2(J\cap\Q)\otimes E\otimes\ell^2)$ be the canonical representations. Then we have
\begin{align*}
(\id_{H_{X_R}\otimes D}\otimes\bar\rho_R(f))\circ V_{D\otimes E}=V_{D\otimes E}\circ(\bar\rho_{r,I}((\widetilde H^B_+|_{X_r})^*f)\otimes \id_D)
\end{align*}
for all $f\in\sHigCom(X_R;E)$ and therefore
\begin{align}
(V_DTV_D^*&\otimes\id_{E\otimes \ell^2})\circ(\id_{H_R\otimes D}\otimes\bar\rho_R(f))= \nonumber
\\&=V_{D\otimes E}\circ(T\otimes\id_{E\otimes\ell^2})\circ(\bar\rho_{r,I}((\widetilde H^B_+|_{X_r})^*f)\otimes \id_D)\circ V_{D\otimes E}^* \label{eq:RelatingPsiwithfm}
\end{align}
holds for all $f$ as above and $T\in\Lin(H_{X_r}\otimes\ell^2(I\cap\Q)\otimes D)$.
Similarily, we have 
\begin{align*}
(V_DTV_D^*&\otimes\id_{E\otimes \ell^2})\circ(\rho_R(g)\otimes\id_{H_R\otimes D\otimes E\otimes \ell^2})=
\\&=V_{D\otimes E}\circ(T\otimes\id_{E\otimes\ell^2})\circ( (\rho_r\otimes\rho_I)((\widetilde H^A_-|_{X_r})^*g)\otimes\id_{D\otimes E\otimes \ell^2} )\circ V_{D\otimes E}^*
\end{align*}
and
\begin{align*}
(\rho_R(g)\otimes&\id_{H_R\otimes D\otimes E\otimes \ell^2})\circ(V_DTV_D^*\otimes\id_{E\otimes \ell^2})=
\\&=V_{D\otimes E}\circ( (\rho_r\otimes\rho_I)((\widetilde H^A_-|_{X_r})^*g)\otimes\id_{D\otimes E\otimes \ell^2} )\circ(T\otimes\id_{E\otimes\ell^2})\circ V_{D\otimes E}^*
\end{align*}
for all $g\in\Cz(X_R)$ and $T\in\Lin(H_{X_r}\otimes\ell^2(I\cap\Q)\otimes D)$.

On $X_r\indexcross{\cU} I_+$ the map $\widetilde H^A_-$ agrees with the projection $\pi(x,t)\coloneqq x$. Therefore, if $T\in\Roe(X_r\indexcross{\cU} I_+;D)$ then 
\begin{align*}
(V_DTV_D^*&\otimes\id_{E\otimes \ell^2})\circ(\id_{H_R\otimes D}\otimes\bar\rho_R(f))\circ(\rho_R(g)\otimes\id_{H_R\otimes D\otimes E\otimes \ell^2})=
\\&=V_{D\otimes E}\circ(T\otimes\id_{E\otimes\ell^2})\circ(\bar\rho_{r,I}(\underbrace{(\widetilde H^B_+|_{X_r})^*f)\cdot\pi^*g}_{\in\Cz(X_r\indexcross{\cU}I_+)\otimes E\otimes \Kom})\otimes \id_D)\circ V_{D\otimes E}^*
\end{align*}
is compact for all $g\in\Cz(X_R)$ due to the locally compactneess of $T$. 
If additionally $g$ is compactly supported and $T$ has finite propagation, then we can choose a function $g'\in\Cz(X_r\indexcross{\cU}I_+)$ which is constantly one on the $\Prop(T)$-neighborhood of the compact subset $\supp(\pi^*g)$ and thus $((\rho_r\otimes\rho_I)(\pi^*g)\otimes\id_D)\circ T=((\rho_r\otimes\rho_I)(\pi^*g)\otimes\id_D)\circ T\circ((\rho_r\otimes\rho_I)(g')\otimes\id_D)$. Then we also obtain that
\begin{align*}
(\rho_R(g)\otimes&\id_{H_R\otimes D\otimes E\otimes \ell^2})\circ(V_DTV_D^*\otimes\id_{E\otimes \ell^2})\circ(\id_{H_R\otimes D}\otimes\bar\rho_R(f))=
\\&=V_{D\otimes E}\circ(\rho_r\otimes\rho_I)(\pi^*g)\otimes\id_{D\otimes E\otimes \ell^2} )\circ
\\&\phantom{V_{D\otimes E}\circ} \circ(T\otimes\id_{E\otimes\ell^2})\circ(\bar\rho_{r,I}(\underbrace{(\widetilde H^B_+|_{X_r})^*f)\cdot g'}_{\in\Cz(X_r\indexcross{\cU}I_+)\otimes E\otimes \Kom})\otimes \id_D)\circ V_{D\otimes E}^*
\end{align*}
is compact. 
What we have just shown is that $(V_DTV_D^*\otimes\id_{E\otimes \ell^2})\circ(\id_{H_R\otimes D}\otimes\bar\rho_R(f))$ is locally compact for all $T\in \Roe(X_r\indexcross{\cU} I_+;D)$ and $f\in\sHigCom(X_R;E)$ and therefore 
\begin{equation}\label{eq:Lplusvanishing}
\Psi_\rL((\widetilde\Gamma^{X,A,B})_*[T]\otimes [f])=0\text{ for all }[T]\in L_+\text{ and }[f]\in\sHigCor(X_R;E)\,.
\end{equation}

Note that the above computations went through so easily, because $T$ was an element of $\Roe(X_r\indexcross{\cU} I_+;D)$ and not just of $\Roe(X_r\indexcross{\cU} I_+\subset X_r\indexcross{\cU} I;D)$. This is the reason why we have passed from the ideal $L_+^\subset$ to the subalgebra $L_+$ in the first big diagram of this proof. Similar reasoning in the next paragraph also explains why we have passed from $L_-^\subset$ to $L_-$.

In order to see what happens with $L_-$, we first make the following observation. Recall that if a uniformly continuous family of isometries $V$ covers a uniformly continuous coarse map $\alpha\colon (X,A)\to (Y,B)$ in the sense of \eqref{eq:continuouscoveringfamily}, then adjoining with $V\otimes D$ maps $\Loc(X,A;D)$ into $\Loc(Y,B;D)$ and $\FiProLoc(X,A;D)$ into $\FiProLoc(Y,B;D)$. If however, $\alpha$ is only a uniformly continuous controlled map (i.\,e.\ it lacks properness), then we can perform the same definitions and the same proofs will still show that adjoining with $V\otimes D$ still maps $\FiProLoc(X,A;D)$ into $\FiProLoc(Y,B;D)$, but it does not map  $\Loc(X,A;D)$ into $\Loc(Y,B;D)$ anymore.

We exploit this fact in our case as follows. Note that our constant family of isometries $V\colon H_{X_r}\otimes\ell^2(I\cap\Q)\to H_{X_R}\otimes H_{X_R}$ not only covers the uniformly continuous coarse map $\widetilde\Gamma^{X,A,B}$ with respect to the representations $\rho_r\otimes\rho_I$ and $\rho_R\otimes\rho_R$, but its restriction to $H_{X_r}\otimes\ell^2(I_-\cap\Q)$ also covers the uniformly continuous controlled map $\widetilde H^A_-|_{X_r\indexcross{\cU} I_-}$ with respect to the representations $\rho_r\otimes\rho_{I_-}$ and $\rho_R\otimes\id_{H_{X_R}}$. Thus, adjoining with $V\otimes\id_{D\otimes E\otimes \ell^2}$ yields a $*$-homomorphism
\[(\widetilde H^A_-|_{X_r\indexcross{\cU} I_-})_*\colon \Loc((X_r,A_r)\indexcross{\cU}(I_-,\{-1\});D\otimes E)\to \FiProLoc(X_R,A_R;D\otimes E)\,.\]
Composing it with the $*$-homomorphism 
\[\fm_\rL\colon L_-\otimes \sHigCor((X_r,B_r)\indexcross{\cU}I_-;E)\to \Loc((X_r,A_r)\indexcross{\cU}(I_-,\{-1\});D\otimes E)\,,\]
which is associated to the space $X_r\indexcross{\cU}I_-$ and its subspaces $X_r\times\{-1\}\cup A_r\times I_-$ and $B_r\times I_-$, as well as with the $*$-homomorphism $\pi^*\colon\sHigCor(X_R,B_R;E)\to \sHigCor((X_r,B_r)\indexcross{\cU}I_-;E)$ induced by the projection $\pi(x,t)=x$, which happens to agree with $\widetilde H^B_+$ on $X_r\indexcross{\cU} I_-$, we obtain a $*$-homomorphism
\[\fM_-\colon (\widetilde H^A_-|_{X_r\indexcross{\cU} I_-})_*\circ(\fm_\rL\otimes\pi^*)\colon L_-\otimes \sHigCor(X_R,B_R;E)\to \FiProLoc(X_R,A_R;D\otimes E)\,.\]

Note that the restriction of $V$ to $H_{X_r}\cong H_{X_r}\otimes\ell^2(\{0\})$ also covers the inclusion $\iota_{r,R}\colon X_r\to X_R$ with respect to the representations $\rho_r$ and $\rho_R\otimes\id_{H_{X_R}}$, and if define 
$(\iota_{r,R})_*$ in the definition of $\fM_0$ as adjoining with 
$V\otimes\id_{D\otimes E\otimes \ell^2}$, then $\fM_-$ clearly extends $\fM_0$. 
Furthermore, \eqref{eq:RelatingPsiwithfm} immediately implies that $\fM_-$ lifts the restriction of 
\[\Psi_\rL\circ((\widetilde\Gamma^{X,A,B})_*\otimes\id)\colon L\otimes \sHigCor(X_R,B_R;E)\to\frac{\FiProLoc(X_R,A_R;D\otimes E)}{\Loc(X_R,A_R;D\otimes E)}\]
to $L_-\otimes \sHigCor(X_R,B_R;E)$.
Combining these two properties and taking \eqref{eq:Lplusvanishing} into account, we see that a $*$-homomorphisms $\fM$ making the second diagram commute is obtained by applying $\Psi_\rL\circ((\widetilde\Gamma^{X,A,B})_*\otimes\id)$ pointwise over $(0,1]$ and $\fM_-$ over $\{0\}$.

Now, what the two diagrams together give us in $\K$-theory is the commutative rectangle
\[\xymatrix@C=10ex{
\Estack{&\K_{m+1}((X_r,A_r\cup B_r)\times (I,\partial I);D)\\&\otimes\K_{-n}(\sHigCor(X_R,B_R;E))}
\ar[r]^-{-(\susp)^{-1}}_-{\cong}\ar[d]^{(\widetilde\Gamma^{X,A,B})_*\otimes\id}
&\Estack{&\K_m(X_r,A_r\cup B_r;D)\\&\otimes\K_{-n}(\sHigCor(X_R,B_R;E))}
\ar[dd]_{(\iota_{r,R})_*\circ\cap\circ(\id\otimes\iota_{r,R}^*)}^{=\cap\circ((\iota_{r,R})_*\otimes\id)}
\\\Estack{&\K_{m+1}(X_R,A_R)\times (X_R,B_R);D)\\&\otimes\K_{-n}(\sHigCor(X_R,B_R;E))}
\ar[d]^{(\Psi_\rL)_*\circ\times}
&\\\K_{m+1-n}\left(\frac{\FiProLoc(X_R,A_R;D\otimes E)}{\Loc(X_R,A_R;D\otimes E)}\right)
\ar[r]^{\partial}
&\K_{m-n}(X_R,A_R;D\otimes E)
}\]
which is obviously natural under enlarging $R$ and (if $R$ is large enough) also natural under enlarging $r$.
Recall that the two arrows in the bottom left corner together are nothing but $(-1)^{m+1}$ times the slant product.

Now if we use the canonical identification $\sHigCor(X,B;E)\cong \sHigCor(X_R,B_R;E)$, then we can first let $R$ go to infinity and subsequently also let $r$ go to infinity and the diagram becomes the following.
\[\xymatrix@C=10ex{
\Estack{&\KX_{m+1}((X,A\cup B)\indexcross{\cU} (I,\partial I);D)\\&\otimes\K_{-n}(\sHigCor(X,B;E))}
\ar[r]^-{-(\susp)^{-1}}_-{\cong}
\ar[d]^{(\Gamma^{X,A,B})_*\otimes\id}
&\Estack{&\K_m(X_r,A_r\cup B_r;D)\\&\otimes\K_{-n}(\sHigCor(X,B;E))}
\ar[d]^{\cap}
\\\Estack{&\KX_{m+1}(X_R,A_R)\times (X_R,B_R);D)\\&\otimes\K_{-n}(\sHigCor(X_R,B_R;E))}
\ar[r]^-{(-1)^{m+1}\cdot /}
&\KX_{m-n}(X_R,A_R;D\otimes E)
}\]
The composition of the left vertical and the two horizontal arrows is exactly the secondary cap product, because even the sign $(-1)^m=-(-1)^{m+1}$ is the one we need.
\end{proof}

\bibliographystyle{alpha}

\begin{thebibliography}{{Phi}91}

\bibitem[BE16]{BunkeEngel_homotopy}
Ulrich Bunke and Alexander Engel.
\newblock Homotopy theory with bornological coarse spaces.
\newblock arXiv:1607.03657, 2016.

\bibitem[BG96]{BridsonGilman_FormalLanguageTheory}
Martin~R. Bridson and Robert~H. Gilman.
\newblock Formal language theory and the geometry of {$3$}-manifolds.
\newblock {\em Comment. Math. Helv.}, 71(4):525--555, 1996.

\bibitem[Bla06]{BlackadarOA}
B.~Blackadar.
\newblock {\em Operator algebras}, volume 122 of {\em Encyclopaedia of
  Mathematical Sciences}.
\newblock Springer-Verlag, Berlin, 2006.
\newblock Theory of $C^*$-algebras and von Neumann algebras, Operator Algebras
  and Non-commutative Geometry, III.

\bibitem[DG17]{DeeleyGoffeng_RealizingI}
Robin~J. Deeley and Magnus Goffeng.
\newblock Realizing the analytic surgery group of {H}igson and {R}oe
  geometrically, part {I}: the geometric model.
\newblock {\em J. Homotopy Relat. Struct.}, 12(1):109--142, 2017.

\bibitem[Dol95]{DoldTopology}
Albrecht Dold.
\newblock {\em Lectures on algebraic topology}.
\newblock Classics in Mathematics. Springer-Verlag, Berlin, 1995.
\newblock Reprint of the 1972 edition.

\bibitem[EM06]{EmeMeyDualizing}
Heath Emerson and Ralf Meyer.
\newblock Dualizing the coarse assembly map.
\newblock {\em J. Inst. Math. Jussieu}, 5(2):161--186, 2006.

\bibitem[ES52]{EilenbergSteenrod}
Samuel Eilenberg and Norman Steenrod.
\newblock {\em Foundations of algebraic topology}.
\newblock Princeton University Press, Princeton, New Jersey, 1952.

\bibitem[EW17]{EngelWulff}
Alexander Engel and Christopher Wulff.
\newblock {Coronas for properly combable spaces}.
\newblock arXiv:1711.06836, 2017.

\bibitem[EWZ19]{EngelWulffZeidler}
Alexander Engel, Christopher Wulff, and Rudolf Zeidler.
\newblock Slant products on the {H}igson-{R}oe exact sequence.
\newblock To appear in Ann. Inst. Fourier, arXiv:1909.03777, 2019.

\bibitem[HR00]{HigRoe}
Nigel Higson and John Roe.
\newblock {\em Analytic {$K$}-homology}.
\newblock Oxford Mathematical Monographs. Oxford University Press, Oxford,
  2000.
\newblock Oxford Science Publications.

\bibitem[{Phi}89]{PhiRep}
N.~Christopher {Phillips}.
\newblock {Representable K-theory for $\sigma$-C${}\sp*$-algebras.}
\newblock {\em {$K$-Theory}}, 3(5):441--478, 1989.

\bibitem[{Phi}91]{PhiFre}
N.~Christopher {Phillips}.
\newblock {$K$-theory for Fr\'echet algebras.}
\newblock {\em {Int. J. Math.}}, 2(1):77--129, 1991.

\bibitem[QR10]{QiaoRoe}
Yu~Qiao and John Roe.
\newblock On the localization algebra of {G}uoliang {Y}u.
\newblock {\em Forum Math.}, 22(4):657--665, 2010.

\bibitem[Roe93]{RoeCoarseCohomIndexTheory}
John Roe.
\newblock Coarse cohomology and index theory on complete {R}iemannian
  manifolds.
\newblock {\em Mem. Amer. Math. Soc.}, 104(497):x+90, 1993.

\bibitem[Roe95]{RoeFoliations}
John Roe.
\newblock From foliations to coarse geometry and back.
\newblock In {\em Analysis and geometry in foliated manifolds ({S}antiago de
  {C}ompostela, 1994)}, pages 195--205. World Sci. Publ., River Edge, NJ, 1995.

\bibitem[Sie12]{Siegel_Thesis}
Paul Siegel.
\newblock {\em Homological calculations with the analytic structure group}.
\newblock ProQuest LLC, Ann Arbor, MI, 2012.
\newblock Thesis (Ph.D.)--The Pennsylvania State University.

\bibitem[Wul16a]{WulffCoassemblyRinghomo}
Christopher Wulff.
\newblock Coarse co-assembly as a ring homomorphism.
\newblock {\em J. Noncommut. Geom.}, 10(2):471--514, 2016.

\bibitem[Wul16b]{WulffFoliations}
Christopher Wulff.
\newblock Ring and module structures on {$K$}-theory of leaf spaces and their
  application to longitudinal index theory.
\newblock {\em J. Topol.}, 9(4):1074--1108, 2016.

\bibitem[Wul19]{WulffTwisted}
Christopher Wulff.
\newblock Coarse indices of twisted operators.
\newblock {\em J. Topol. Anal.}, 11(4):823--873, 2019.

\bibitem[Wul20]{wulff2020equivariant}
Christopher Wulff.
\newblock Equivariant coarse (co-)homology theories.
\newblock arXiv:2006.02053, to be updated to version 2 shortly, 2020.

\bibitem[WY20]{WillettYuHigherIndexTheory}
Rufus Willett and Guoliang Yu.
\newblock {\em Higher Index Theory}.
\newblock Cambridge Studies in Advanced Mathematics. Cambridge University
  Press, 2020.

\bibitem[XY14]{XieYu_PSCrhoLoc}
Zhizhang Xie and Guoliang Yu.
\newblock Positive scalar curvature, higher rho invariants and localization
  algebras.
\newblock {\em Adv. Math.}, 262:823--866, 2014.

\bibitem[Yu95]{YuCyclicCohomology}
Guoliang Yu.
\newblock Cyclic cohomology and higher indices for noncompact complete
  manifolds.
\newblock {\em J. Funct. Anal.}, 133(2):442--473, 1995.

\bibitem[Yu97]{YuLocalization}
Guoliang Yu.
\newblock Localization algebras and the coarse {B}aum-{C}onnes conjecture.
\newblock {\em $K$-Theory}, 11(4):307--318, 1997.

\bibitem[Zei16]{ZeidlerPSCProductSecondary}
Rudolf Zeidler.
\newblock Positive scalar curvature and product formulas for secondary index
  invariants.
\newblock {\em J. Topol.}, 9(3):687--724, 2016.

\bibitem[Zen17]{Zenobi_SurgerytoAnalysis}
Vito~Felice Zenobi.
\newblock Mapping the surgery exact sequence for topological manifolds to
  analysis.
\newblock {\em J. Topol. Anal.}, 9(2):329--361, 2017.

\end{thebibliography}

\ \\
\textsc{
Mathematisches Institut, 
Georg--August--Universit\"at G\"ottingen,
Bunsenstr. 3-5, 
D-37073 G\"ottingen, 
Germany}

\noindent
\textit{E-mail address:} \url{christopher.wulff@mathematik.uni-goettingen.de}

\end{document}